\documentclass[reqno,11pt]{amsart}
\usepackage{amsmath,amssymb,mathrsfs,amsthm,amsfonts}
\usepackage[usenames,dvipsnames]{xcolor}
\usepackage{comment}
\usepackage{hyperref}
\usepackage{graphicx}
\usepackage{enumerate}
\usepackage{stmaryrd}
\usepackage[all]{xy}
\usepackage{bbm}
\usepackage{hyperref}
\usepackage[capitalise]{cleveref}
\usepackage{caption}
\usepackage{subcaption}
\hypersetup{
     colorlinks=true,
     linkcolor=blue,
     filecolor=black,
     citecolor = blue,      
     urlcolor=blue,
     }

\usepackage{fullpage}
\usepackage[margin=22mm]{geometry}
\usepackage{tikz}
\usetikzlibrary{decorations.markings}
\usetikzlibrary{arrows}
\usetikzlibrary{calc}

\theoremstyle{plain}
\newtheorem{thm}{Theorem}
\newtheorem{lemm}[thm]{Lemma}
\newtheorem{cor}[thm]{Corollary}
\newtheorem{conj}[thm]{Conjecture}
\newtheorem{prop}[thm]{Proposition}

\theoremstyle{definition}
\newtheorem{defn}[thm]{Definition}
\newtheorem{q}[thm]{Question}

\newtheorem{rem}[thm]{Remark}

\numberwithin{equation}{section}
\numberwithin{thm}{section}

\graphicspath{{Images/}}

\newcommand{\s}{\mathbf{s}}

\newcommand{\Cov}{\operatorname{Cov}}

\theoremstyle{plain}

\theoremstyle{definition}

\theoremstyle{remark}

\numberwithin{equation}{section}
\numberwithin{thm}{section}

\begin{document}

\title{The directed landscape is a black noise}
\author[Himwich]{Zoe Himwich}
\author[Parekh]{Shalin Parekh}

\address{Zoe Himwich: Department of Mathematics, Columbia University, 2990 Broadway, New York, NY 10027, USA}
\email{himwich@math.columbia.edu}
\address{Shalin Parekh: Department of Mathematics, University of Maryland, College Park, 4176 Campus Drive, College Park, MD 20742 USA}
\email{parekh@umd.edu}
\date{\today}

\begin{abstract}
    We show that the directed landscape is a black noise in the sense of Tsirelson and Vershik \cite{TV}. As a corollary, we show that for any microscopic system in which the height profile converges in law to the directed landscape under 1:2:3 scaling, the driving noise becomes asymptotically independent of the height profile. This decoupling result provides one answer to the question of what happens to the driving noise in the limit under the KPZ scaling, and illustrates a type of noise sensitivity for systems in the KPZ universality class. Such decoupling and sensitivity phenomena are not present in the intermediate-disorder or weak-asymmetry regime, and this result illustrates a contrast between the ``strong KPZ" and ``weak KPZ" scaling regimes. Along the way, we prove a strong mixing property for the directed landscape on a bounded time interval under spatial shifts, with a mixing rate $\alpha(k)\leq Ce^{-dk^3}$ for some $C,d>0$. 
\end{abstract}
\maketitle
\setcounter{tocdepth}{1}{
  \hypersetup{linkcolor=black}
  \tableofcontents
}

\section{Introduction}
\subsection{Preface}
Over the past several decades, mathematicians and physicists who study interacting particle systems, integrable systems, and the Kardar-Parisi-Zhang (KPZ) equation have discovered connections between the large-scale behavior of a wide-ranging class of models. It is known that the height functions of asymmetric exclusion processes and stochastic vertex models exhibit universal scaling behavior in $1+1$ dimensions. A common feature in these models is the appearance of $1:2:3$ scaling exponents and Tracy-Widom limit laws. Similar results are expected for the Eden model and last passage percolation models with arbitrary weight distributions, as well as Burgers-type stochastic PDEs and ballistic deposition models (e.g. \textit{Tetris}). This phenomenon is called KPZ universality \cite{C12,Q11}. 

In recent years, there has been a great deal of work involved in constructing and understanding the universal limit objects in the KPZ universality class. Matetski, Quastel, and Remenik \cite{MQR} constructed the KPZ fixed point, the scale-invariant limiting Markov process (conjecturally) of all models in the KPZ universality class. Dauvergne, Ortmann and Vir\'ag \cite{dov} then constructed the directed landscape, a generalization which encodes the natural joint couplings of all KPZ fixed points started from all possible initial data simultaneously.
The directed landscape is a four-parameter field, denoted in this paper as $\mathcal{L}_{s,t}(x,y)$, see \cref{d:dl}.

It is a longstanding project of mathematical physicists to have a better understanding of this phenomenon of universality. It is conjectured that any model for which it is possible to associate a height function $h(t,x)$ with the basic properties of local dynamics, a smoothing mechanism, lateral growth, and space-time forcing with rapid correlation decay, will converge to the KPZ fixed point under $1:2:3$ scaling $\epsilon h(\epsilon^{-3}t,\epsilon^{-2}x)-C_{\epsilon}t$, see for example the introduction of \cite{MQR}. More generally, the height mechanisms started simultaneously from several distinct initial data should converge to KPZ fixed points coupled through the directed landscape. 
Models which satisfy these conditions, and which have been proved or conjectured to have the KPZ fixed point and directed landscape as a scaling limit, are said to belong to the KPZ universality class.  

The directed landscape (\cref{d:dl}), which is the main object in this paper, was constructed by Dauvergne, Ortmann, and Vir\'ag \cite{dov} as the scaling limit of the last passage metric between points on a Brownian line ensemble (Brownian LPP). It is a random four-parameter continuous function  $\mathcal{L}:\mathbb{R}^{4}_{\uparrow} \to \mathbb{R}$, with $\mathbb{R}^{4}_{\uparrow}:=\{(x,s;y,t)\in\mathbb{R}^{4} | s<t\},$ which may be viewed as a metric between space-time points $(s,x)$ and $(t,y)$ that satisfies a reverse triangle inequality, $\mathcal{L}_{s,u}(x,z)\geq \mathcal{L}_{s,t}(x,y) + \mathcal{L}_{t,u}(y,z)$ if $s<t<u$ and $x,y,z\in\Bbb R$. More precisely the following ``metric composition" law holds:
\begin{align}\label{e:metriccomp}
    \mathcal{L}_{s,u}(x,z) & = \max_{y\in\mathbb{R}}\left(\mathcal{L}_{s,t}(x,y) + \mathcal{L}_{t,u}(y,z)\right), & & \text{for all } (x,z)\in\mathbb{R}^{2}, \;\;s<t<u.
\end{align}
A series of recent works have demonstrated that the directed landscape is the scaling limit of many well-known models with simple microscopic dynamics. 
Works such as \cite{dov, DNV23,dv2} demonstrated that it is the limit of the Brownian, Poissonian, exponential, and geometric last passage percolation, and Poisson and Sepp\"al\"ainen-Johansson line ensembles. A work by Wu \cite{xuan} combined with a result of Quastel and Sarkar \cite{qs} demonstrates that the directed landscape is the scaling limit of the KPZ equation started from all possible initial data but coupled through the same realization of the driving noise. Most recently, a work by Aggarwal, Corwin, and Hegde \cite{ACH24} demonstrated a similar convergence for the asymmetric simple exclusion process (ASEP) and stochastic six vertex models. The KPZ fixed point of Matetski, Quastel, and Remenik \cite{MQR} is a marginal of the directed landscape by the variational formula
$h_{t}(y)= \sup_{x\in\mathbb{R}}\left(h_{0}(x)+\mathcal{L}_{0,t}(x,y)\right)$, see \cite{nqr}. 

There are still many basic properties of the directed landscape that have yet to be explored. A natural question to ask is whether it has some simpler Gaussian-driven structure, for example as the solution of a stochastic partial differential equation (SPDE).

\begin{q}\label{q:1} Is the directed landscape the adapted strong solution of a white-noise-driven SPDE?
\end{q}
This article was motivated by an interest in better understanding what happens to the driving noise of the KPZ equation under the directed landscape limit. Another natural question in this vein is the following. 
\begin{q}\label{q:2}
    What happens to the microscopic driving noise of systems like last passage percolation models, asymmetric exclusion processes, and the KPZ equation as we take the limit to the directed landscape? In particular, is any memory of this microscopic noise retained by the limiting directed landscape?
\end{q}
This paper aims to address both of the preceding questions. We study the directed landscape as an abstract noise in the sense of Tsirelson and Vershik \cite{TV}, demonstrating that it is a \textit{black noise} (\cref{t:1}). Using this property, we show that the directed landscape is independent of any white noise field defined on the same filtered probability space (\cref{t:2}). These results have a few corollaries. We conclude that the height function associated to any model which converges to the directed landscape will asymptotically decouple from the random environment (\cref{c:nokpz}), which answers \cref{q:2}. We can also answer \cref{q:1} in the negative (\cref{c:nospde}).

\subsection{Background and Results}\label{s:defns}

\subsubsection{Black Noise} Black noise was first introduced by Tsirelson and Vershik \cite{TV}, and subsequently explored at length in a survey by Tsirelson \cite{Tsir}. Building towards the definition of black noise, we begin by introducing the concept of a noise.

\begin{defn}[Definition 3d(1), \cite{Tsir}]\label{d:noise}
    A noise, denoted in this paper as a 4-tuple $(\Omega,(\mathcal{F}_{s,t})_{s<t},\mathbb{P},(\theta_{h})_{h})$, is a probability space $(\Omega,\mathcal{F},\mathbb{P})$ with sub-$\sigma$-fields $\mathcal{F}_{s,t}\subset\mathcal{F}$ ($-\infty<s<t<\infty$) and measurable maps $\theta_{h}:\Omega\to\Omega$ ($h\in\mathbb{R}$) such that the following properties are satisfied.
    \begin{enumerate}
        \item $\mathcal F_{s,t}$ and $\mathcal F_{t,u}$ are independent under the probability measure $\mathbb P$ if $s<t<u$.
        \item For all $s<t<u$, $\mathcal{F}_{s,t}$ and $\mathcal{F}_{t,u}$ together generate $\mathcal{F}_{s,u}$, i.e., $\mathcal{F}_{s,t}\vee \mathcal{F}_{t,u}=\mathcal{F}_{s,u}$.
        \item $\mathcal{F}$ is generated by the union of all $\mathcal{F}_{s,t}$ ($-\infty<s<t<\infty$).
        \item $\theta_0=\mathrm{Id}$, and $\theta_h \theta_k = \theta_{h+k}$ for all $h,k\in \mathbb{R}.$
        \item If $A\in \mathcal F_{s,t}$ and $h\in\Bbb R$, then $\theta_h(A)\in \mathcal F_{s+h,t+h}$, and $\mathbb P(\theta_h(A)) = \mathbb P(A).$
    \end{enumerate}
\end{defn} This definition of noise uses the language of stochastic flows, in the sense of Tsirelson \cite{Tsir}. We can relate this to other, possibly more familiar, definitions of noise. For example, if $B$ is a one-dimensional two-sided Brownian motion, the $\sigma$-algebra $\mathcal F_{s,t}$ would be generated by random variables of the form $B(u)-B(s)$ for $u\in [s,t]$, and $\theta_h B(u) := B(u+h)-B(h)$. More generally a Gaussian space-time white noise on $\mathbb{R}\times \mathbb{R}^d$, as discussed in Walsh's book on SPDEs \cite{Wal}, associates a Gaussian random variable to every Borel set, and defines a structure of a noise in a similar and straightforward fashion. These are both examples of white noise.

We can now give a definition of black noise. There are several equivalent definitions; Tsirelson \cite[Definition 7a1]{Tsir} uses the notion of a stable $\sigma$-algebra to define white noise and black noise. The definition we use in this paper is more direct. 

\begin{defn}\label{d:bn}
Let $(\Omega, (\mathcal F_{s,t})_{s<t}, \mathbb{P}, (\theta_h)_{h\in \mathbb{R}})$ be a noise. A random variable $F:\Omega\to \mathbb{R}$ is called \textit{linear} if 
\begin{align*}
    \mathbb{E}\left[F|\mathcal F_{s,t}\right] + \mathbb{E}\left[F|\mathcal F_{t,u}\right] = \mathbb{E}\left[F|\mathcal F_{s,u}\right],\;\;\;\;\;\;\;\;a.s.  \;\;\;\;\;\; \mathrm{for\;\; all\;\;\;\;} s<t<u.
\end{align*}
We say that $(\Omega, (\mathcal F_{s,t})_{s<t}, \mathbb{P}, (\theta_h)_{h\in \mathbb{R}})$ is a \textit{white noise} if $\mathcal F$ is generated by linear random variables. Likewise, we say that $(\Omega, (\mathcal F_{s,t})_{s<t}, \mathbb{P}, (\theta_h)_{h\in \mathbb{R}})$ is a \textit{black noise} 
if $0$ is the only linear random variable.
\end{defn}

In general, white noise consists of Gaussian and Poissonian components, whereas black noise can be much more complex and admits no simple description. In resources which discuss abstract noise such as \cite{Tsir}, white noise is sometimes also called ``classical'' or ``stable.'' 
White noise and black noise do not form a dichotomy: there are noises that satisfy \cref{d:noise} that have both a ``white noise part'' and a ``black noise part," see \cite[Section 6]{Tsir}. Well-known examples of black noise include sticky and coalescing Brownian flows (see \cite[Section 7]{Tsir} and the foundational works \cite{lejan, HW09}), the scaling limit of critical planar percolation \cite{SS11}, and the Brownian web \cite{EF16}. In the following subsection, we will provide further discussion and examples of black noise.

\subsubsection{The KPZ Universality Class and the directed landscape} In this section, we define the directed landscape and several other important objects in the KPZ universality class. We begin by discussing the Airy process and Airy line ensemble.

\begin{defn}\label{d:ale}
    The Airy Line Ensemble is a collection of random functions $\left\{\mathcal{A}_{i}\right\}_{i\in\mathbb{N}}$ with law determined by the Airy kernel \eqref{e:airyker}. For any finite set $I\subset\mathbb{R}$, we define the point process $\{(x,\mathcal{A}_{i}(x))|x\in I\}$ on $I\times\mathbb{R}$ as the determinantal point process 
    with kernel $K(x_{1},t_{1};x_{2},t_{2})$ given by the extended Airy kernel
    \begin{align}\label{e:airyker}
        K(x_{1},t_{1};x_{2},t_{2}) & := \begin{cases}
            \int_{0}^{\infty}e^{-r(t_{1}-t_{2})}\mathrm{Ai}(x_{1}+r)\mathrm{Ai}(x_{2}+r)dr & t_{1}\geq t_{2} \\
            - \int_{-\infty}^{0}e^{-r(t_{1}-t_{2})}\mathrm{Ai}(x_{1}+r)\mathrm{Ai}(x_{2}+r)dr & t_{1}< t_{2} .
        \end{cases}
    \end{align}
    We refer to \cite[Section 4]{HKPV09} for a discussion of how a kernel $K$ uniquely determines the $n$-point correlation functions for all $n$, which in turn uniquely characterize a the law of the determinantal point process. In this formula, $\mathrm{Ai}(\cdot)$ is the Airy function,
    \begin{align*}
        \textrm{Ai}(x) & = \frac{1}{\pi}\lim_{b\to\infty}\int_{0}^{b}\cos{\left(\frac{t^{3}}{3} + xt\right)}dt.
    \end{align*}
    The parabolic Airy line ensemble is defined by $\mathcal{P}_{i}(x)=\mathcal{A}_{i}(x)-x^{2}$. We commonly refer to the top curve $\mathcal{P}_{1}(\cdot)$ as the parabolic Airy$_2$ process. 
\end{defn} 
The parabolic Airy line ensemble can also be defined as the unique line ensemble with a Brownian Gibbs resampling property and with the parabolic Airy$_2$ process as its top line. Existence of such a line ensemble was shown by Corwin and Hammond \cite{CH} and uniqueness was shown by Dimitrov and Matetski \cite{DM}. 

For a metric space $X$, we let $C(X)$ denote the space of continuous functions $X\to\mathbb R$, equipped with the topology of uniform convergence on compacts. For the choices of $X$ used in this paper, the space $C(X)$ is always separable and metrizable, hence Polish. We can now introduce the Airy sheet, a two-variable random function that can be viewed as the universal limit of the last passage times for pairs of points of models in the KPZ class.

\begin{defn} The Airy sheet is a random continuous function $\mathcal{S}:\mathbb{R}^{2}\to \mathbb{R}$, i.e., a $C(\mathbb{R}^{2})$-valued random variable, such that the following properties hold: 
\begin{enumerate}
    \item $\mathcal{S}(\cdot,\cdot)$ has the same law as $\mathcal{S}(\cdot +t,\cdot +t)$ for all $t\in\mathbb{R}$, as $C(\mathbb R^2)$-valued random variables. 
    \item $\mathcal{S}$ can be coupled with a parabolic Airy line ensemble so that $\mathcal{S}(0,\cdot)=\mathcal{P}_{1}(\cdot)$ and for all $(x,y,z)\in\mathbb{R}^{+}\times\mathbb{R}^{2}$, 
    \begin{align*}
        \lim_{k\to\infty} \mathcal{P}\left[\left(-\sqrt{\frac{k}{2x}},k\right)\to \left(z,1\right)\right]-\mathcal{P}\left[\left(-\sqrt{\frac{k}{2x}},k\right)\to\left(y,1\right)\right] = \mathcal{S}(x,z)-\mathcal{S}(x,y),
    \end{align*} 
    where $\mathcal{P}\left[\left(a,k\right)\to \left(b,1\right)\right]$ is the last passage time between points $(a,k)$ and $(b,1)$ on the Airy line ensemble,
    \begin{align*}
        \mathcal{P}[(a,k)\to(b,1)] :=\sup_{a=t_{k}<...<t_{0}=b}\sum_{i=1}^{k}\left(\mathcal{P}_{i}(t_{i-1})-\mathcal{P}_{i}(t_{i})\right).
    \end{align*} 
\end{enumerate}
\end{defn}
Dauvergne, Ortmann, and Vir\'ag \cite{dov} proved that the Airy sheet exists and that it is unique in law. In fact, it was later shown to be a deterministic functional of the Airy line ensemble \cite{dv2}. With these definitions in hand, we can discuss the directed landscape. As explained in \cite[Example 1.6]{dv2}, the directed landscape can be interpreted as a random directed spacetime metric of negative sign. We can think of the directed landscape as a universal limit of the metric defined by last passage times on a last passage model. We begin by defining the canonical probability space for the directed landscape. 

\begin{defn}\label{d:dirlprob}
We define the canonical probability space of the directed landscape $(\Omega^{\mathrm{DL}},\mathcal{F}^\mathrm{DL})$. Let $\mathbb{R}^{4}_{\uparrow}:=\{(x,s;y,t)\in\mathbb{R}^{4} | s<t\}$. Define the underlying probability space $\Omega^{\mathrm{DL}}=C(\mathbb{R}^4_\uparrow)$, equipped with its Borel sets from the topology of uniform convergence on compact sets. The canonical process on this space is the identity map $\mathcal L:C(\mathbb{R}^4_\uparrow)\to C(\mathbb{R}^4_\uparrow)$. The evaluation map at a point $(x,s,y,t)$ applied to $\mathcal{L}$ will be denoted in this paper as $\mathcal{L}_{s,t}(x,y)$. In this notation, we can define $\mathcal{F}^\mathrm{DL}$ as the $\sigma$-algebra generated by $$\mathcal F^\mathrm{DL}_{s,t}:=\sigma\left(\{\mathcal{L}_{a,b}(x,y):  s\le a<b\le t, (x,y)\in \mathbb{R}^2\}\right),\;\;\;\;\;\;\;\;-\infty<s<t<\infty.$$ 
\end{defn}
\begin{defn}\label{d:dl} 
    The directed landscape is a $C(\Bbb R^4_\uparrow)$-valued random variable $\mathcal L$ such that
\begin{enumerate}
    \item $\mathcal L_{s_j,t_j}$ are independent if $(s_j,t_j]$ are disjoint intervals.

    \item $\mathcal L_{s,t} \circ \mathcal L_{t,u} = \mathcal L_{s,u} $ for all $s<t<u$ where the ``metric composition" is defined for two functions $f,g:\mathbb{R}^2\to \mathbb{R}$ as $$(f\circ g)(x,z) = \sup_{y\in \mathbb{R}} f(x,y)+g(y,z).$$

    \item For all $s<t$, $\mathcal L_{s,t}$ has the same law as $R_{t-s}\mathcal L_{0,1}$ where $R_\epsilon f(x,y) := \epsilon^{1/3} f(\epsilon^{-2/3}x,\epsilon^{-2/3} y).$

    \item $\mathcal L_{0,1}$ is distributed as the Airy Sheet.
\end{enumerate}
\end{defn}
As an immediate consequence of this definition, 
the directed landscape defines a noise. 
\begin{lemm}\label{c:dlnoise} The directed landscape defines a noise $(\Omega^{\mathrm{DL}},\left(\mathcal{F}^{\mathrm{DL}}_{s,t}\right)_{s<t},\mathbb{P},\left(\theta_{h}^{\mathrm{DL}}\right)_{h})$ in the sense of \cref{d:noise}, where $(\theta_{h}^{\mathrm{DL}} \mathcal L)_{s,t}(x,y):= \mathcal L_{s+h,t+h}(x,y)$.
\end{lemm}
\begin{proof} Let $\Omega^{\mathrm{DL}}$ and $\mathcal{F}^{\mathrm{DL}}_{s,t}$ be as in \cref{d:dirlprob}. The measurable maps $\theta_{h}^{\mathrm{DL}}:\Omega^{\mathrm{DL}}\to\Omega^{\mathrm{DL}}$ clearly satisfy the group property $\theta_{h}^{\mathrm{DL}}\theta_{k}^{\mathrm{DL}} = \theta_{h+k}^{\mathrm{DL}}$. 
Properties $(1)$ and $(2)$ of \cref{d:dl} verify properties $(1)$ and $(2)$ of \cref{d:noise}, respectively. Properties $(4)$ and $(5)$ of \cref{d:noise} are satisfied directly by the definition of the map $\theta_{h}^{\mathrm{DL}}$. Finally, we note that $\mathcal{F}^{\mathrm{DL}}=\bigvee_{-\infty<s<t<\infty}\mathcal{F}^{\mathrm{DL}}_{s,t}$ by definition.
\end{proof}

We can now state our main theorem.

\begin{thm}\label{t:1} The noise associated to the directed landscape is a black noise in the sense of \cref{d:bn}. 
\end{thm}

There are a few equivalent conditions for demonstrating that a noise is a black noise. We will use the one stated in \cref{c:bncondition}, which involves a variance bound that resembles a noise sensitivity condition with respect to all functions in $L^{2}(\Omega)$. Noise sensitivity was first studied in the context of Boolean functions in the foundational work of \cite{BKS99}, which in turn built earlier papers such as \cite{KKL88}. This theory of noise sensitivity was then applied to problems in physics and computer science, in works such as \cite{ST99,SS10,MOO10}. Noise sensitivity is closely related to black noise. In fact, another way to understand the concept of black noise is to consider it as a system with respect to which all observables (functionals) are noise sensitive. The latter definition of black noise can be made precise through the spectral sample as defined in \cite[Definition XII.61]{GS}.

It would be interesting to explore these noise sensitivity properties further in the context of KPZ-related models. Ganguly and Hammond \cite{GH20}, for instance, study chaos and stability of Brownian last passage percolation (LPP), using the notion of the overlap of perturbed geodesics. The works of \cite{Ahlberg2023, Ahlberg2024, Ahlberg3} then obtain some chaos-and-stability results for the geodesic paths in LPP models with more general weights. However, despite this progress, it remains an open problem to prove quantitative noise sensitivity results for the last passage value itself (as opposed to the geodesic) in LPP models. 

While we do not study more specific noise sensitivity phenomena in this paper, we do obtain a few corollaries of \cref{t:1}. The next result, \cref{t:2}, demonstrates the independence of the directed landscape from any white noise defined on the same probability space. For the purpose of this text, a Gaussian white noise on an open subset $U\subset \mathbb{R}^d$ is a family of mean-zero jointly Gaussian random variables $\{\xi(f)\}_{f\in L^2(U)} $ all defined on the same probability space, satisfying $\mathbb E[\xi(f)\xi(g)] = (f,g)_{L^2(U)}.$ We often write $\xi(f):= \int_U f(x)\xi(dx)$ or similar. A Gaussian space-time white noise is just a special case $U=\mathbb R^2,$ where the first coordinate is interpreted as ``time" and the second as ``space."

\begin{thm}\label{t:2}
    Let $(\Omega, (\mathcal F_{s,t})_{s<t},\mathbb P, (\theta_h)_h)$ be a noise in the sense of \cref{d:noise}. Let $\mathcal L$ and $\xi$ respectively be a copy of a directed landscape and a Gaussian space-time white noise, defined on this same probability space, which together generate the entire $\sigma$-algebra $\mathcal F:=\bigvee_{s<t} \mathcal F_{s,t}$, and which are both adapted to the filtration $(\mathcal F_{s,t})_{s<t}$, in the sense that
    \begin{itemize}
        \item $\int_{\mathbb{R}^2} \phi(u,x) \xi(du,dx)$ is $\mathcal F_{s,t}$-measurable if $\phi$ is supported on $[s,t]\times\mathbb{R}.$
    
        \item $\mathcal L_{a,b}(x,y)$ is $\mathcal F_{s,t}$-measurable for all $s\le a<b\le t$ and $x,y\in\mathbb{R}.$
    \end{itemize} Then $\mathcal{L}$ and $\xi$ are independent under $\mathbb P$.
\end{thm}

The directed landscape is not unique in this respect. The proof will show that one could replace $\mathcal L$ by any space-time field that is known to be a black noise. Likewise, the space-time white noise can be replaced by any Gaussian noise that is white in time, with arbitrary spatial covariance structure. 

This theorem may be surprising, since it illustrates a contrast with many prelimiting models in the KPZ class, in which case one often has $\mathcal L_N=F_N(\xi_N)$ for some $(\mathcal F_{s,t})$-adapted path functionals $F_N$ of the discrete driving noise $\xi_N$. As we take the limit $N\to \infty$, \cref{t:2} shows that this dependence no longer holds. In fact, the two fields $\mathcal L_N$ and $\xi_N$ become independent. The following corollary illustrates this phenomenon in the context of the KPZ equation.

\begin{cor}[The environment decouples from the height process when scaling to the landscape] \label{c:nokpz} Let $\xi$ be a Gaussian space-time white noise on $\mathbb R_+\times \mathbb R$, and let $H^\varepsilon(t,x)$ denote the Hopf-Cole solution of the KPZ equation $\partial_t H^\varepsilon = \frac14(\partial_xH^\varepsilon)^2 + \frac{\varepsilon}4 \partial_x^2 H^\varepsilon +\varepsilon^{1/2} \xi + \frac1{12}$, started from any deterministic sequence of upper semicontinuous initial data $H^\varepsilon(0,x)$ satisfying $H^\varepsilon(0,x) \leq C(1+|x|)$ for some $C>0$, and converging as $\varepsilon\to 0$ to some upper semicontinuous profile $H_0(x)$ in the sense of local hypograph convergence \cite[Section 3.1]{MQR}. The pair $(H^\varepsilon,\xi)$ converges in law as $\varepsilon\to 0$, in the topology of $C((0,\infty)\times \mathbb{R})\times \mathcal S'(\mathbb{R}^2).$ The limit is given by the law of $(\mathfrak h,\xi)$, where  $\mathfrak h$ is distributed as the KPZ fixed point started from $H_0(x)$, and $\xi$ is an \textbf{independent} space-time white noise.
\end{cor}

Before proving the \cref{c:nokpz}, we note that 
the local hypograph convergence for the sequence of initial data is a weaker assumption than uniform convergence on compacts if all profiles are continuous. We also remark that the scaling of $H^\varepsilon$ is equivalent (in distribution, but not pathwise) to the 1:2:3 scaling of the height field given by $(t,x)\mapsto \varepsilon^{1/2} H(\varepsilon^{-3/2}t, \varepsilon^{-1}x)$, as noted in \cite{qs}. 

\begin{proof}
    We define a four-parameter field $\mathscr H^\varepsilon_{s,t}(x,y)$ for each $\varepsilon>0$ such that $(t,y)\mapsto \mathscr H^\varepsilon_{s,t}(x,y)$ is the Hopf-Cole solution of the same equation solved by $H^\varepsilon$, with the same realization of the noise $\xi$, started at time $s$ from initial condition given by a narrow wedge at position $x$ (meaning that the exponential is started from a Dirac mass centered at location $x$). The main results of \cite{qs,xuan} imply that this four-parameter field $\mathscr H^\varepsilon$ converges in law (in the topology of $C(\mathbb{R}^4_\uparrow)$) to the directed landscape. 
    
    We note that joint limit points of the pair $(\mathscr H^\varepsilon,\xi)$ do exist on $\mathcal S'(\Bbb R^2) \times C(\Bbb R^4_\uparrow)$, because each coordinate is tight, which implies that the pair is tight. If we can verify the conditions of \cref{t:2}, it would imply that any joint limit point of $(\mathscr H^\varepsilon,\xi)$ must be given by the law of $(\mathcal L,\xi)$ where $\mathcal L$ is a directed landscape \textit{independent} of $\xi$. To verify those conditions, one can define the $\sigma$-algebras $\mathcal F_{s,t}$ on the canonical space $\mathcal S'(\Bbb R^2) \times C(\Bbb R^4_\uparrow)$ to be precisely the ones generated by random variables of the two forms given by the two bullet points in \cref{t:2}, namely $\mathcal F_{s,t} = \sigma( \{ \xi(\phi)| \mathrm{supp}(\phi) \subset [s,t]\times \Bbb R\} \cup \{\mathcal L_{a,b} | s\le a<b \le t\})$. We need to verify that the noise property holds with respect to the limit point, more precisely that $\mathcal F_{s,t}$ and $\mathcal F_{t,u}$ are independent with respect to the limit point for $s<t<u$. This property holds in the $\varepsilon\to 0$ limit simply because it holds trivially in the prelimit. 
    Thus, the conditions of \cref{t:2} have been verified.
    
    Now we consider any joint limit point $(\mathfrak h,\mathcal L,\xi)$ of the triple $(H^\varepsilon,\mathscr H^\varepsilon,\xi)$. On one hand, we already know that $\mathcal L$ and $\xi$ must be independent under this limit point. On the other hand, the results of \cite{nqr,xuan} imply that $\mathfrak h(t,y) = \sup_{x\in\mathbb{R}} H_0(x) + \mathcal L_{0,t}(x,y) .$ Therefore $\mathfrak h$ is $\mathcal L$-measurable, which implies that it is independent of $\xi.$ Projecting onto the first and third marginals $\mathfrak h$ and $\xi$ gives the result.
\end{proof}
While \cref{c:nokpz} was formulated purely in terms of the KPZ equation, the result of \cref{t:2} implies something much stronger: in any system for which there is a height process which converges to the directed landscape, the underlying field of environment variables necessarily decouples from the height process under that limit. In the case of the ASEP, the environment consists of the Poisson clocks that result in executed jumps for the system. In the case of last passage percolation or directed polymers, the underlying environment consists of the independent and identically distributed (IID) weights through which the last passage paths travel. \cref{t:2} implies that for such systems, the environment and the height process asymptotically decouple as they respectively scale to the white noise and to the landscape \cite{ACH24, dv2}. 

These decoupling and sensitivity phenomena illustrate a contrast between the intermediate-disorder regime and the strong-disorder regime for $(1+1)$-dimensional systems in the KPZ universality class. For example, there are many ``weak KPZ" scaling limit results where particle systems, SPDEs, or directed polymers are known to converge to the KPZ equation \cite{BG97,AKQ14, DT16, DGP, CST18, GJ14, HQ, Yang} as opposed to the KPZ fixed point. In these results, it is straightforward to show from the methods of those papers that the field of environment variables and the height process converge \textit{jointly} in law to the driving noise of the limiting KPZ equation. This is always the case in the presence of intermediate disorder scaling or weakly asymmetric scaling of models in the KPZ university class. There are a few intuitive reasons why the same phenomenon does \textit{not} happen when scaling to the directed landscape, i.e., in the \textit{strong} disorder or \textit{fixed} asymmetry regime. At the level of last passage percolation or directed polymers, the height process under strong disorder scaling only sees the largest variables in the environment field, whereas most of the smaller weights are ignored. The height process thus ignores ``most" of the environment under the strong disorder scaling, with only a few special exceptions. In contrast, all of the noise variables contribute to some extent for the intermediate disorder scaling, as individual large contributions have a less dramatic effect on the overall height process.

We discuss another, perhaps more precise, conjecture about why the strong disorder scaling is so different from the intermediate disorder regime with regard to retaining the memory of the microscopic environment. We work through the lens of \cref{c:nokpz}. Recall that the noise $\xi$ generates a Fock space structure for $L^2$ of the underlying probability space, and the layers are called the \textit{homogeneous chaoses} over $\xi$ \cite{nualart}. 

\begin{conj}
    With $H^\varepsilon$ as in \cref{c:nokpz}, we conjecture that the $L^2$-mass of $H^\varepsilon(1,0)$ is concentrated in the homogeneous chaos of order $\varepsilon^{-1/2}$ over $\xi$. 
\end{conj}

And in fact, we conjecture that this ``mass escape'' phenomenon will not just occur for the height process itself, but also for every bounded measurable functional of the height process. The ``mass escape'' phenomenon is another way to formulate noise sensitivity of a sequence of observables \cite[Proposition 2.2]{Gar}. Very precise bounds for the rate of escape of the Fourier-Walsh mass have been obtained in the case of a different model given by the crossing probabilities of a box in Boolean percolation, see \cite{BKS99, SS10, GPS, TV1} or the textbook \cite{GS}. 

As another corollary of \cref{t:2}, we show that the directed landscape is not an SPDE driven by space-time Gaussian white noise (in fact, any Gaussian noise that is white in time would not be sufficient to generate the landscape). To formalize the notion of being an SPDE, we note that any probabilistic strong solution of an SPDE defines an adapted path functional on the probability space of the white noise. We therefore show that no such adapted functional can exist.

\begin{cor}[The directed landscape is not an SPDE]\label{c:nospde} Let $\xi$ be a Gaussian space-time white noise on $\mathbb{R}^2,$ defined on some probability space $(\Omega, \mathcal F^\xi,{\mathbb{P}})$. For $s<t$ let $\mathcal F_{s,t}^\xi$ denote the $\sigma$-algebra generated by the random variables $\int_{\mathbb{R}^2} \phi(u,y)\xi(du,dy)$ as $\phi$ ranges over all smooth functions of support contained in $[s,t]\times \mathbb{R}.$ Then there does not exist any random variable $\mathcal L$ taking values in $C(\mathbb{R}^4_\uparrow)$, defined on the same probability space $\Omega$, satisfying the following properties:
\begin{itemize}
    \item $\mathcal L$ is distributed as the directed landscape.
    \item $\mathcal L_{a,b}(x,y)$ is $\mathcal F^\xi_{s,t}$-measurable whenever $s\leq a<b\le t$ and $x,y\in \mathbb{R}.$
\end{itemize}
\end{cor}

The proof is immediate from \cref{t:2}, as $\mathcal L$ cannot be simultaneously independent of $\xi$ and adapted to its filtration: this would imply $\mathcal L$ is independent of itself and thus deterministic. We remark that this corollary is false without the adaptedness assumptions. Indeed a recent work of \cite{dv3} shows that the landscape can be written as some \textit{non-adapted} functional of white noise, in fact through a measurable bijection.

Next, we discuss a few open questions that are related to the content of this paper. The two-dimensional critical stochastic heat flow is a continuum model recently constructed by \cite{CSZ23}. It is a flow of random measures on $\mathbb R^2$, written as $Z_{s,t}(\vec x,d\vec y)$ indexed by $-\infty<s<t<\infty$ and $\vec x,\vec y\in\mathbb R^2$. We can think of it as the random field obtained from the universal scaling limit of partition functions of intermediate-disorder directed polymers in a space-time random environment in spatial dimension $d=2$. This is another example of a random field which should define a noise in the sense of \cref{d:noise}, and it exhibits certain properties which strongly suggest that it will not be a white noise \cite{CSZ23b}. 

\begin{conj} The two-dimensional critical stochastic heat flow constructed in \cite{CSZ23} is a black noise.
\end{conj}

By the remark after \cref{t:2}, if this conjecture is true, it would imply a decoupling theorem analogous to \cref{c:nokpz} for two-dimensional polymers. Our next conjecture is about the directed landscape as a noise in its space variables as well as in its time variables. To formulate this precisely, we need to define the notion of a two-dimensional noise. Our definition is adapted from \cite{Tsir} or Ellis and Feldheim \cite{EF16}. 
\begin{defn}\label{d:2dbl} A two-dimensional noise is a probability space $(\Omega,\mathcal{F},\mathbb{P})$ with a collection of sub-$\sigma$-algebras $\mathcal{F}_{\vec{x},\vec{y}}\subset\mathcal{F}$ associated to all open two-dimensional rectangles in $\mathbb{R}^{2}$, and a collection of measurable maps $\theta_{\vec{h}}:\Omega\to \Omega $ indexed by $\vec h \in \mathbb{R}^{2}$, such that the following properties hold.
\begin{enumerate}
    \item $\mathcal{F}_{\vec{x}_{1},\vec{y}_{1}}$ and $\mathcal{F}_{\vec{x}_{2},\vec{y}_{2}}$ are independent under $\mathbb{P}$ when $R_{1}:=(x_{1}^{(1)},x_{1}^{(2)})\times (y_{1}^{(1)},y_{1}^{(2)})$ and $R_{2}:=(x_{2}^{(1)},x_{2}^{(2)})\times (y_{2}^{(1)},y_{2}^{(2)})$ satisfy $R_{1}\cap R_{2}=\emptyset$.
    \item $\mathcal{F}_{\vec{x}_{1},\vec{y}_{1}}\vee\mathcal{F}_{\vec{x}_{2},\vec{y}_{2}}=\mathcal{F}_{\vec{x}_{3},\vec{y}_{3}}$ whenever the rectangles $R_{1}$ and $R_{2}$ satisfy $R_{1}\cap R_{2}=\emptyset$ and $\overline{R_{1}\cup R_{2}}=\overline{R_{3}}.$
    \item $\mathcal{F}$ is generated by the union of all of the $\mathcal{F}_{\vec{x},\vec{y}}$.
    \item If $A\in\mathcal{F}_{\vec{x},\vec{y}}$, then $\theta_{\vec{h}}(A)\in \mathcal{F}_{\vec{x}+(h_1,h_1),\vec{y}+(h_2,h_2)}$ and $\mathbb{P}(\theta_{\vec{h}}(A))=\mathbb{P}(A)$.
    \item $\theta_{\vec{h}_{1}}\theta_{\vec{h}_{2}}=\theta_{\vec{h}_{1}+\vec{h}_{2}}$ for all $\vec{h}_{1},\vec{h}_{2}\in\mathbb{R}^{2}$, and $\theta_{\vec{0}}=\text{Id}.$ 
\end{enumerate}
We say this two-dimensional noise is a black noise when the only linear random variable is $0$, i.e., any $F\in L^2(\Omega)$ satisfying $\mathbb E[ F| \mathcal F_{R_1\cup R_2}] = \mathbb E[ F| \mathcal F_{R_1} ] + \mathbb E[ F| \mathcal F_{R_2} ]$ for all rectangles $R_1,R_2$ as in Item (2), implies that $F=0$.
\end{defn}
There are two well-known examples of two-dimensional black noises. 
The first is the scaling limit of critical planar percolation, which was conjectured to be black noise by Tsirelson \cite{Tsir} and later proved by Schramm and Smirnov \cite{SS11}. The second is the Brownian web, which was proved to be a two-dimensional black noise by Ellis and Feldheim \cite{EF16}.
In the context of this definition, we can state the following conjecture, due to B\'alint Vir\'ag (in conversation). Define the $\sigma$-algebras $\mathcal F_{(s,t),(x,y)}$ to be the ones generated by the random variables $\int_a^b d\mathcal L\circ \pi$ (see \cite[Eq. (5)]{dov}) as we vary over all $s\le a<b\le t$ and all continuous paths $\pi: [a,b]\to [x,y].$ Then we have the following conjecture.

\begin{conj} The directed landscape is a two-dimensional black noise. 
\end{conj}

The difficulty of the two-dimensional problem is showing that the noise property holds in the first place, in particular Item (2) in \cref{d:2dbl}. Once Item (2) is proved, it is automatically a two-dimensional black noise by our one-dimensional result. If true, this would address a problem posed by Ellis and Feldheim \cite{EF16}. Their work demonstrates that the Brownian web is an example of a two-dimensional black noise, and they ask for further relevant examples of two-dimensional black noise arising from simple probabilistic systems.

Finally, let us briefly discuss the proof of Theorem \ref{t:1}. There are several key properties of the directed landscape that are crucial to the proof that it is a black noise, and which distinguish it from many prelimiting models in the KPZ universality class that are inherently white noise.

\begin{enumerate} 
\item  Scale invariance: Obtaining a rescaled version of the same space-time process at every temporal scale is a crucial ingredient in many parts of the proof. For example, in \cref{p:alphamixingestimate} we get a spatial mixing rate for the directed landscape of order $e^{-d k^3}$ on a time interval $[0,1]$. This mixing rate immediately implies a mixing rate of $e^{-d \epsilon^{-2} k^3}$ for the landscape on a shorter time interval $[0,\epsilon]$ by using the 1:2:3 scale invariance of the landscape.

\item  Zero Temperature: Many of our estimates would be weakened or fail for positive-temperature analogues like the KPZ equation, as the maximum in the definition of the directed landscape would get replaced by a convolution operation which is not as local. And indeed, the KPZ equation is not a black noise since it is measurable with respect to the driving noise.

\item  Local Brownianity: Even for many zero-temperature 
conjectural prelimits (such as inviscid stochastic Burgers equations driven by spatially smoothed white noise) the local Brownianity of the spatial process fails. Improved spatial regularity of the solution actually hurts many of the estimates one needs to prove the black noise property, as the better regularity “widens” the set of $x$-values where the spatial height process is close to the argmax. Controlling the size of the latter set of $x$-values is really the heart of the proof, see \cref{crucial}.
\end{enumerate} 

\begin{rem} The black noise property seems intimately related to the degeneracy and fractal structure of the landscape, for instance the tendency of geodesics to overlap. The papers \cite{heg,ganguly1,duncan1, dasghosal} study this fractal structure in some depth. It remains an open problem to precisely and quantitatively relate the black noise property of the landscape to its fractal structure. For instance, the optimal exponents in the variance bounds needed to prove the black noise property should be related to the Hausdorff dimension of the endpoints of exceptional geodesics. 
\end{rem}

\subsection{Notation}\label{s:notation}
Throughout the paper, we use the font $\mathfrak B, \mathfrak{C}$ as a way to denote random constants. The letters $C,d$ are used, respectively, to denote deterministic constants that may get larger or smaller from line to line when performing an estimate. Typically, such $C,d$ will be (respectively) a multiplying constant and an exponential decay rate of some quantity. For a metric space $X$ we will let $C(X)$ denote the space of continuous functions $X\to\mathbb R$, equipped with the topology of uniform convergence on compacts. We often deal with random variables defined on the underlying probability space $C(\mathbb R)$ or $C(\mathbb R^2)$ or $C(\mathbb R^4_\uparrow)$, and these random variables should be understood as being real-valued Borel-measurable functions with respect to the topology of uniform convergence on compact sets (which is always separable and metrizable, hence Polish).

\subsection{Structure of the Paper}\label{s:structure} In \cref{s:noisetoairy}, we reduce the problem of demonstrating that the directed landscape is a black noise to demonstrating a specific variance estimate about the Airy sheet (\cref{c:airycondition}). \cref{s:spatialmixing} through \cref{s:prooft1} are dedicated to obtaining this estimate, which relies on a mixing property of the directed landscape (\cref{s:spatialmixing}), as well as estimates on argmaxes of Airy sheets (\cref{s:airy}) and estimates on Bessel processes (\cref{s:bessel}). Then in \cref{s:prooft1}, we conclude that the directed landscape is a black noise. Finally in \cref{s:proofadditional}, we prove \cref{t:2} from \cref{t:1}. 

\subsection{Acknowledgements} SP thanks Duncan Dauvergne, Jeremy Quastel, and B\'alint Vir\'ag for discussions and context about this problem during a visit to Toronto. SP and ZH thank Ivan Corwin and Yu Gu for comments on an earlier draft of the paper. SP acknowledges support by the NSF MSPRF (DMS-2401884). ZH was supported by the Fernholz Foundation's Minerva Fellowship Program and Ivan Corwin's grant, NSF DMS-1811143. 

\section{Reducing the black noise problem to an estimate on the Airy sheet}\label{s:noisetoairy}
In this section, we reduce the problem of showing that the directed landscape is a black noise to a 
specific variance bound for the Airy sheet. This is the foundation for the rest of our argument. 

In his survey, Tsirelson \cite{Tsir} gives a few examples in which it is possible to prove something is a black noise, and provides a necessary and sufficient condition for a given noise to be black. We quote the following result from that survey, see \cite[Proposition 7a(3) Item (b)]{Tsir}.

\begin{prop}\label{p:bncondition} Let $(\Omega, (\mathcal F_{s,t})_{s<t}, \mathbb{P}, \theta_{h})$ be a noise in the sense of \cref{d:noise}. The space $M$ of $\mathcal F_{0,1}$-measurable linear random variables (\cref{d:bn}) is a closed linear subspace of $L^2(\Omega)$, and moreover the orthogonal projection of any $F\in L^2(\Omega)$ onto $M$ is given by $$ \lim_{k\to\infty} \sum_{i=1}^{2^k}\big(\mathbb{E}\big[ F\big| \mathcal F_{(i-1)2^{-k},i2^{-k}}\big] - \mathbb E[F]\big) ,$$ where the limit is understood in $L^2(\Omega).$
\end{prop}

\begin{cor}\label{c:bncondition} A noise $(\Omega, (\mathcal F_{s,t})_{s<t}, \mathbb{P}, \theta_{h})$ is a black noise 
if and only if the following condition holds on a dense linear subspace $D\subset L^{2}(\Omega)$: as $k\to\infty$ we have that
$$\sum_{i=1}^{2^k} \big(\mathbb{E}\big[ F\big| \mathcal F_{(i-1)2^{-k},i2^{-k}}\big]-\mathbb E[F]\big) \stackrel{L^2(\Omega)}{\to} 0,\;\;\;\;\;\;\;\;\mathrm{for\;\;all \;\;\;} F\in D.$$ Equivalently, 
\begin{align}\label{e:densecondition}\lim_{k\to\infty} \sum_{i=1}^{2^k} \mathrm{Var}\left(\mathbb{E}\big[ F\big| \mathcal F_{(i-1)2^{-k},i2^{-k}}\big] \right)= 0, \;\;\;\;\;\;\;\;\mathrm{for\;\;all \;\;\;} F\in D.
\end{align}
\end{cor}
We recall the definition of the directed landscape and the associated noise (\cref{d:dl}) and use $\Omega^{\mathrm{DL}}$ to denote the underlying probability space of $\mathcal{L}$. There is a particular class $D$ of functions that will be most useful to us in the context of the directed landscape. 
We define $D\subset L^2(\Omega)$ to be those functionals of the landscape that are finite linear combinations of ``polynomial functions" of the form \begin{equation}\label{e:dfunctions}F=\prod_{j=1}^m \prod_{i=1}^{n_j} \mathcal L_{s_j,t_j}(x_{ij},y_{ij}),
\end{equation} 
where $(s_j,t_j)$ are disjoint intervals of $\mathbb{R}$ whose endpoints $s_j,t_j$ are both dyadic numbers (of the form $p2^{-q}$ with $p\in\mathbb{Z}$ and $q\in \mathbb{N}),$ furthermore $x_{ij},y_{ij}\in \mathbb R$ are all distinct points in $\mathbb{R}$, and $m,n_j\in\mathbb{N}$. To prove that linear combinations of such polynomials are dense, we first prove the following purely measure-theoretic lemma, which will also be important in \cref{s:proofadditional}.


\begin{lemm}\label{new}
    Let $(\Omega, \mathcal G,\mathbb P)$ be a standard probability space, and suppose that $\mathcal G$ is generated by some countably infinite collection of random variables $\{X_n\}_{n=1}^\infty$ taking values in some Polish space $\mathcal X$. Consider any bounded measurable $F:\Omega\to \mathbb R$. Then for each $m\in\mathbb N$ there exists a Borel measurable map $\Psi_m:\mathcal X^m\to\mathbb R $ such that $\sup_{x_1,...,x_m\in \mathcal X} |\Psi_m(x_1,...,x_m)| \leq \|F\|_{L^\infty(\Omega)}$, and moreover the random variables $F_m:= \Psi_m(X_1,...,X_m) $ satisfy $F=\lim_{m\to\infty} F_m$ $\mathbb P$-almost surely as well as $\lim_{m\to\infty} \|F_m-F\|_{L^p(\Omega)}= 0$ for $1\le p<\infty$.
\end{lemm}

\begin{proof}
    We define $\mathcal G_m:=\sigma(X_1,...,X_m)$ and $F_m:= \mathbb E[F|\mathcal G_m].$ Using approximation by $\mathcal G_m$-measurable simple functions, we see that $F_m$ are of the form $\Psi_m(X_1,...,X_m)$ for some Borel measurable $\Psi_m:\mathcal X^m\to\mathbb R $. If needed, we replace $\Psi_m$ by $(-\|F\|_{L^\infty(\Omega)}) \vee \Psi_m \wedge \|F\|_{L^\infty(\Omega)}$, and this still satisfies the same relation $F_m=\Psi_m (X_1,...,X_m)$ on a set of measure $1$. By Doob's martingale convergence theorem, $F_m\to F$ almost surely and in every $L^p(\Omega)$.
\end{proof}

With \cref{new} in hand, we can prove density of the space $D$ from above. 

\begin{lemm}\label{lem2}
     Let $\Omega^{\mathrm{DL}}$ be as in \cref{d:dirlprob}. The set $D\subset L^2(\Omega^{\mathrm{DL}})$, given by finite linear combinations of products of the form \eqref{e:dfunctions}, is a dense linear subspace.
\end{lemm}

\begin{proof} We begin by noting the following \textit{tensor product property} of $L^{2}$-spaces. The space $L^2(\gamma_1)\otimes \cdots \otimes L^2(\gamma_m)$ is isometrically isomorphic as a Hilbert space to $L^2(\gamma_1\otimes \cdots \otimes \gamma_m)$ for any Borel probability measures $\gamma_i$ on a Polish space, under the canonical isomorphism $(F_1\otimes \cdots \otimes F_m) (x_1,...,x_m) \cong F_1(x_1)\cdots F_m(x_m)$. 

By this tensor product property of $L^2$ spaces, in order to show that linear combinations of functions of the form \eqref{e:dfunctions} are dense in $L^{2}(\Omega^{\mathrm{DL}})$, it suffices to show that any function of the form \begin{equation}\label{phij}\prod_{j=1}^m \varphi_j(\mathcal L_{s_j,t_j}(x_1,y_1),...,\mathcal L_{s_j,t_j}(x_{n},y_{n})),\end{equation}
is in the closure of $D$ in $L^2(\Omega^{\mathrm{DL}})$, where $m,n\in\mathbb N$, where $(s_j,t_j)$ are disjoint intervals of $\Bbb R$, and where $\varphi_j$ are bounded measurable functions on $\mathbb{R}^n$. This is because linear combinations of functionals of this latter form are dense in $L^2(\Omega^{\mathrm{DL}})$. 
To see this, we note that by \cref{new} we can take $n\uparrow\infty$ to approximate any functional of the form $\prod_{j=1}^m \Phi_j (\mathcal L_{s_j,t_j})$ where $m\in\mathbb N$, where $(s_j,t_j)$ are disjoint intervals of $\Bbb R$, and where $\Phi_j:C(\mathbb R^2)\to\mathbb R$ are bounded measurable, and $\mathcal L_{s,t}:\mathbb R^2\to\mathbb R$ denotes the random function given by $(x,y)\mapsto \mathcal L_{s,t}(x,y)$. Then, by taking linear combinations and using the tensor product property, we can approximate any functional of the form $\Psi(\mathcal L_{s_1,t_1},...,\mathcal L_{s_m,t_m}) $ where $m\in\mathbb N$, where $(s_j,t_j)$ are disjoint intervals of $\Bbb R$, and where $\Psi: C(\mathbb R^2)^m\to\mathbb R$ is bounded measurable. Finally by \cref{new} we can take $m\uparrow\infty$ and the latter types of functionals can then approximate any general functional $F(\mathcal L).$ 

Therefore, we just need to show that any function of the form \eqref{phij} is in the $L^2(\Omega^{\mathrm{DL}})$-closure of $D$. The terms of this product are independent, so it suffices to prove that for $m=1$, the expression in \eqref{phij} may be approximated in $L^2(\Omega^{\mathrm{DL}})$ by linear combinations of functions of the form \eqref{e:dfunctions} with $m=1$. Showing this is equivalent to showing that polynomial functions on $\mathbb{R}^n$ are dense in $L^2(\mu)$ where $\mu$ is the probability measure on $\mathbb{R}^n$ given by the joint law of $(\mathcal L_{s,t}(x_1,y_1),...,\mathcal L_{s,t}(x_n,y_n))$, where $s,t,x_j,y_j$ are fixed real numbers. The marginals of this joint law are deterministically shifted Tracy-Widom GUE laws, which implies that they have a super-exponential tail decay. Therefore, we conclude that $\mu$ must have a finite moment generating function everywhere. 

Assume that $f\in L^2(\mu)$ is orthogonal to all polynomials, then consider the Fourier transform of the measure $f d\mu$, given by $(\lambda_1, ..., \lambda_n) \mapsto \int_{\mathbb{R}^n} f(x) e^{ i\sum_1^n \lambda_j x_j} \mu(dx).$ Taylor expand the exponential as an infinite series. Applying Fubini's theorem (which is justified because of the super-exponential tail decay on $\mu$), we can interchange the infinite sum with the integral over $\mathbb{R}^n$ to conclude that  $\int_{\mathbb{R}^n} f(x) e^{ i\sum_1^n \lambda_j x_j} \mu(dx) = 0$ for all $\lambda = (\lambda_1,...,\lambda_n)\in \mathbb{R}^n$. By inverting the Fourier transform (which is a linear isomorphism on $\mathcal S'(\mathbb R^n))$, this implies that the measure $f\;d\mu \in \mathcal S'(\mathbb R^n)$ must be the zero measure, i.e., $f=0$ $\mu$-almost everywhere. 
\end{proof}

Due to the independence of the $\sigma$-algebras on disjoint intervals (property $(1)$ in \cref{d:dl}), we can ignore the product over $j$ in \eqref{e:dfunctions}. Thus, without loss of generality, we may assume that $m=1$ in \eqref{e:dfunctions}. We label the single interval by $(s_1,t_1)=(s,t)$. By shifting and rescaling the landscape, we can assume without loss of generality that $(s,t)=(0,1).$ Therefore, we have reduced the black noise problem to demonstrating that 
\begin{align}\label{e:conditionspec}
\lim_{k\to\infty} \sum_{i=1}^{2^k} \mathrm{Var}\left(\mathbb{E}\bigg[ \prod_{j=1}^{n} \mathcal L_{0,1}(x_j,y_j) \bigg| \mathcal F^\mathrm{DL}_{(i-1)2^{-k},i2^{-k}}\bigg]\right)= 0, & & \text{for all } \vec x,\vec y\in \mathbb{R}^n, \text{ and } n\in \mathbb{N}.
\end{align}
The entire problem is now formulated in terms of the Airy sheet $\mathcal S:=\mathcal L_{0,1}$. In fact, we can simplify things further. 
\begin{prop}
    In order to show that the directed landscape is a black noise, it suffices to show that for any $\eta\in (0,1/2)$ and $n\in \mathbb{N}$ and $\vec x,\vec y\in \mathbb{R}^n$,
    \begin{equation}\label{frrr}\lim_{k\to\infty} \sum_{\eta 2^{k}\le i \le (1-\eta)2^k} \mathrm{Var}\left(\mathbb{E}\bigg[ \prod_{j=1}^{n} \mathcal L_{0,1}(x_j,y_j) \bigg| \mathcal F^\mathrm{DL}_{(i-1)2^{-k},i2^{-k}}\bigg]\right)= 0,
    \end{equation}
\end{prop}
\begin{proof} Let $F:= \prod_{j=1}^{n} \mathcal L_{0,1}(x_j,y_j),$ and let $G$ denote the linear random variable obtained by applying the orthogonal projection of $F$ onto the closed linear subspace of $L^2(\Omega)$ consisting of linear $\mathcal F^\mathrm{DL}_{0,1}$-measurable random variables. 

Applying \cref{p:bncondition}, we see from \eqref{frrr} that for any $\eta\in (0,1/2)$, the random variable $G$ must be measurable with respect to $\mathcal F^\mathrm{DL}_{0,\eta}\vee \mathcal F^\mathrm{DL}_{1-\eta,1}$. In other words, all dependency on $\mathcal F^\mathrm{DL}_{\eta,1-\eta}$ vanishes. This is true for all $\eta\in (0,1/2)$. Therefore, $G$ is measurable with respect to $\mathcal H:=\bigcap_{\eta\in (0,1/2)} \mathcal F^\mathrm{DL}_{0,\eta}\vee \mathcal F^\mathrm{DL}_{1-\eta,1}.$ We claim that $\mathcal H$ is a trivial $\sigma$-algebra, only containing sets of measure $0$ or $1$. We note that $\mathcal H$ is contained in $\mathcal F^\mathrm{DL}_{0,1}.$ However, $\mathcal{H}$ is also independent of the $\sigma$-algebra generated by $\bigcup_{\eta\in (0,1/2)} \mathcal F^\mathrm{DL}_{\eta,1-\eta}.$ In particular, we claim that the $\sigma$-algebra generated by $\bigcup_{\eta\in (0,1/2)} \mathcal F^\mathrm{DL}_{\eta,1-\eta}$ is all of $\mathcal F^\mathrm{DL}_{0,1}.$ To see this, note by the continuity of the directed landscape in all parameters that the random variable $\mathcal L_{0,1}(x,y)$ can be written as the almost sure limit as $\eta\to 0$ of $\mathcal L_{\eta,1-\eta}(x,y).$ Thus $\mathcal H\subset \mathcal F^\mathrm{DL}_{0,1}$ is independent of $\mathcal F^\mathrm{DL}_{0,1}$, proving that it is a trivial $\sigma$-algebra.

Since the random variable $G$ is $\mathcal H$-measurable, we conclude that it is almost surely constant. But $G$ is an $L^2$-limit of the sequence of random variables as in \cref{p:bncondition}, and therefore \eqref{e:conditionspec} holds.
\end{proof}

The following corollary is an immediate consequence.

\begin{cor}\label{c:airycondition}
    Fix some $\varrho>0$. In order to show that the directed landscape is a black noise, it suffices to show that for any $\eta\in (0,1/2)$,  $n\in \mathbb{N}$ and $\vec x,\vec y\in \mathbb{R}^n$,  there exists some $C = C(n,\eta,\vec x,\vec y)>0$ such that for all $a,b\in [\eta,1-\eta]$,
    \begin{equation}\label{e:bncondfinal}\mathrm{Var}\left(\mathbb{E}\bigg[ \prod_{j=1}^{n} \mathcal L_{0,1}(x_j,y_j) \bigg| \mathcal F^\mathrm{DL}_{a,b}\bigg]\right) \leq C (b-a)^{1+\varrho},
    \end{equation}
\end{cor}

The rest of the paper will focus on proving \eqref{e:bncondfinal}, for a value of $\varrho$ that is slightly smaller than $\frac1{15}$. This bound may not be sharp (and we suspect $1/3$ will be the optimal value). Even for $n=1$, proving this bound is nontrivial. Let us briefly discuss the ideas for $n=1$ and $x=y=0$ and $a=1-b$ before moving on. Write $f\circ L\circ g:= \max_{x,y\in \Bbb R^2} f(x)+L(x,y)+g(y)$. In this simpler case, we can write
    \begin{align*}
        \mathrm{Var}\left(\mathbb{E}\left[  \mathcal L_{0,1}(0,0) \big| \mathcal F^\mathrm{DL}_{a,b}\right]\right) &= \mathrm{Var} \bigg( \int_{C(\Bbb R)^2} f\circ \mathcal L_{a,b}\circ g \; \mathbb P_{\mathrm{Airy}_2}^{\otimes 2}(df,dg) \bigg) \\ &= \int_{C(\Bbb R)^4}  \mathrm{Cov}(f\circ \mathcal L_{a,b}\circ g, u \circ \mathcal L_{a,b}\circ v)\mathbb P_{\mathrm{Airy}_2}^{\otimes 4}(df,dg,du,dv),
    \end{align*}
    where $\mathbb P_{\mathrm{Airy}_2}$ is the law of the parabolic Airy$_2$ process $\mathcal P_1$ from \cref{d:ale} rescaled by some factor depending on $a$. To analyze such covariances, it will be important to prove a strong mixing property for the landscape, which is done in \cref{s:spatialmixing}. 

    
    In the analysis, we will also need precise estimates for \textit{where} the maximum is achieved for $f\circ \mathcal L_{0,\epsilon} \circ g$. We will also need upper bounds for $p^{th}$ moments of these quantities, both of which we prove in \cref{s:airy}. We obtain fairly sharp bounds on $\mathrm{Cov}(f\circ \mathcal L_{0,\epsilon}\circ g, u \circ \mathcal L_{0,\epsilon}\circ v)$ that are uniform over all \textit{deterministic} $f,g,u,v$ lying in some class of functions (chosen so that the sample paths of the parabolic Airy$_2$ process lie in this class of functions almost surely). Since the Airy$_2$ process is locally Brownian, we will also need precise estimates on Bessel processes, which we prove in \cref{s:bessel}. Finally in \cref{s:prooft1} we prove \eqref{e:bncondfinal}, using the inputs from \cref{s:spatialmixing}, \cref{s:airy}, and \cref{s:bessel}.  

\section{A strong mixing property for the landscape under spatial shifts}\label{s:spatialmixing}

In this section, we prove a strong mixing property for the directed landscape under spatial shifts (\cref{c:diagonalmixing}) that is crucial to the arguments in \cref{s:airy}. First, we state a theorem from \cite{dov} that is used in the proof. 
\begin{thm}[Theorem 1.7 \cite{dov}] \label{dov1.7} Fix $u=(x,t;y,s)\in\mathbb{R}^{4}_{\uparrow}$. There is almost surely a unique directed geodesic $\Pi_{u}$ from $(x,s)$ to $(y,t)$, where a geodesic is defined to be any path $\pi:[s,t]\to\mathbb R$ such that $\pi(s) = x , \pi(t)=y$ and for all $k\in \mathbb N$ and all $s=s_0<...<s_k=t$, $$\mathcal L_{s,t}(x,y) = \sum_{j=1}^k \mathcal L_{t_{k-1},t_k} (\pi(t_{k-1}),\pi(t_k)).$$ Its distribution only depends on $u$ through scaling: as random continuous functions from $[0,1]\to\mathbb{R}$, 
 \begin{align*}
     \Pi_{(x,t;y,s)}(s+(t-s)r)\overset{d}{=}\Pi_{(0,0;0,1)}(r)+x+(y-x)r.
 \end{align*} Moreover, for $u=(0,0;0,1)$, there is a random constant $\mathfrak{C}$ such that, for all $s,t\in [0,1]$  with $s\ne t$,
 \begin{align*}
     \big|\Pi_{u}(s)-\Pi_{u}(t)\big| \leq \mathfrak{C}|t-s|^{\frac{2}{3}}\log^{\frac{1}{3}}{\left(\frac{2}{t-s}\right)}.
 \end{align*} The random constant satisfies $\mathbb{E}\big[a^{\mathfrak{C}^{3}}\big]<\infty$ for some $a>1$. 
    
 \end{thm}

In other words, the geodesic path has a H\"older norm that has tails which decay like $e^{-u^{3}}$, which will be very useful in deriving mixing estimates. 

With \cref{dov1.7} in hand, we prove that there exists a probability space on which we can couple three copies of the directed landscape in such a way that two of them are independent, and the third agrees with the other two sufficiently far from the origin. 
\begin{lemm}\label{l:mixing} There exist three copies of the directed landscape $\mathcal L^i=(\mathcal L^i_{s,t})_{0\le s<t\le 1}$ with $i=0,1,2$, all coupled onto the same probability space $(\Omega,\mathcal F,\mathbb{P}_{\mathrm{triple}})$ so that:
\begin{enumerate}
    \item $\mathcal L^1, \mathcal L^2$ are independent under $\mathbb{P}_{\mathrm{triple}}$. 
    \item There exists a positive random variable $X:\Omega\to (0,\infty)$ such that $\mathcal L_{s,t}^0(x,y) = \mathcal L_{s,t}^1(x,y)$ for all $x,y>X$ and $0\le s<t\le 1$, and furthermore $\mathcal L_{s,t}^0(x,y) = \mathcal L_{s,t}^2(x,y)$ for all $x,y<-X$ and $0\le s<t\le 1$.
    \item $\mathbb{P}_{\mathrm{triple}}(X>u)\leq Ce^{-du^3}$ where $C,d>0$ do not depend on $u>0$.
\end{enumerate}
\end{lemm}
This lemma is inspired by \cite[Proposition 2.6]{duncan1}, which proves a coupling result for copies of the directed landscape via approximation by a last passage model. 
\begin{proof}[Proof of \cref{l:mixing}] 
    We consider two fields of IID rate-one exponential random variables $\omega^i = \{\omega^i_{t,x}\}_{t\ge 0,x\in \mathbb{Z}}$ for $i=1,2$, independent of one another. We define a third field of IID variables $\omega^0_{t,x}$ by $\omega^0_{t,x} := \omega^1_{t,x}$ if $x\ge 0$ and $\omega^0_{t,x} := \omega^2_{t,x}$ if $x< 0$. 

    For $i=0,1,2$ and $x,y\in \mathbb{Z}$ and $0\le m<n\in \mathbb{Z}$, we consider the fields of last passage times $$L_{m,n}^{i}(x,y):= \max\{ W^i(\pi)| \pi:\{m,...,n\}\to \mathbb{Z} ; \pi(m)=x,\pi(n)=y\},$$ where $\pi$ denotes a nearest neighbor path $\pi:\{m,...,n\}\to \mathbb{Z}$ and its length is defined by $$W^i(\pi):=  \sum_{\ell=m}^n \omega^i_{\ell,\pi(\ell)}.$$
    For $x,y\in \varepsilon^{2/3}\mathbb{Z}$ and $s,t\in \varepsilon\mathbb{Z}$, we then define a rescaled last passage time $$\mathcal L^{i,\varepsilon}_{s,t} (x,y):= \varepsilon^{1/3} \big(L_{\varepsilon^{-1}s,\varepsilon^{-1}t}^{i}(\varepsilon^{-2/3}x,\varepsilon^{-2/3}y) - \varepsilon^{-1}(t-s)\big).$$
    \cite[Theorem 1.20]{dv2} shows that the marginal laws of each $\mathcal L^{i,\varepsilon}_{s,t} (x,y)$ \textit{individually} converge in law as $\varepsilon\to 0$ to the directed landscape in the uniform-on-compact topology. Here, the compact set can be taken to be any compact subset of $\{(s,t,x,y): 0\le s<t\le 1, (x,y)\in \mathbb{R}^2\},$ see \cref{d:dirlprob}.
    
    Furthermore, if $\pi_{\mathrm{opt}}^{i,m,n,x,y} $ denotes the
    maximizing path, i.e., satisfying $L^{i}_{m,n}(x,y) = W^i(\pi_{\mathrm{opt}}^{i,m,n,x,y})$, then 
    \cite[Theorem 8.7]{dv2} implies that the paths $$\Pi^{i,s,t,x,y;\varepsilon}(r):= \varepsilon^{2/3} \pi_{\mathrm{opt}}^{i,\varepsilon^{-1}s,\varepsilon^{-1}t,\varepsilon^{-2/3}x,\varepsilon^{-2/3}y}(\varepsilon^{-1} r),$$ converge in law as $\varepsilon\to 0$, jointly in all variables $(s,t,x,y,r)$, with respect to the uniform-on-compact topology on $\{ (s,t,x,y,r):0\le s<t\le 1, s\le r\le t, (x,y)\in\mathbb{Z}^2 \}$, to geodesic paths $\Pi^{i,s,t,x,y}(r)$ of the corresponding directed landscapes $\mathcal L^i$, for $i=0,1,2$. This convergence occurs jointly with the height profiles $\mathcal L^{i,\varepsilon}_{s,t} (x,y)$.

    Now let $(\mathcal L^0,\mathcal L^1,\mathcal L^2)$ be any joint limit point of $(\mathcal L^{0,\varepsilon},\mathcal L^{1,\varepsilon},\mathcal L^{2,\varepsilon})$ as $\varepsilon\to 0.$ We define the random variable $$X:= \inf\big\{x\in \mathbb{N}\big| \inf_{r\in [0,1]} \Pi^{0,0,1,x,x}(r)>0 , \sup_{r\in [0,1]} \Pi^{0,0,1,-x,-x} (r)<0\big\}, $$ which tracks the smallest positive integer value of $x$ for which both of the directed geodesic paths from $x\to x$ and from $-x\to-x$ do not cross $0$ in $\mathcal L^0$. We verify that $\mathcal L^1,\mathcal L^2$ are independent and that $\mathcal L_{s,t}^0(x,y) = \mathcal L_{s,t}^1(x,y)$ for all $x,y>X$ and $0\le s<t\le 1$, by noting that the analogous facts hold trivially in the prelimit. Furthermore, $\mathcal L_{s,t}^0(x,y) = \mathcal L_{s,t}^2(x,y)$ for all $x,y<-X$ and $0\le s<t\le 1$, again because the analogous fact holds in the prelimit. This follows from the fact that two geodesics will never cross one another (though they may possibly stick together for some time). 

    It remains to verify the tail bound in Item (3) of the lemma statement. The moment bound of \cref{dov1.7} implies that for any fixed $x\in \mathbb{R}$ the probability of any geodesic path from $x\to x$ varying more than a distance $u>0$ from its starting point decays like $Ce^{-du^3}$ for some $C,d>0$. This directly implies Item $(3)$ through a union bound over the two sets corresponding to positive $x$ and negative $x$ in the definition of $X$ above. 
\end{proof}

We use \cref{l:mixing} to prove a \textit{strong mixing property} for the landscape under spatial shifts. 

\begin{prop}[Strong mixing property]\label{p:alphamixingestimate} The directed landscape $\mathcal L=(\mathcal L_{s,t})_{0\le s<t\le 1}$ is $\alpha$-mixing under spatial shifts, with mixing rate $\alpha(k)\leq Ce^{-dk^3}$ for some universal constants $C,d>0$. More precisely, we have the following stronger estimate. For $-\infty\leq a<b\leq \infty$ let $\mathcal G_{a,b}$ denote the $\sigma$-algebra generated by the random variables $\mathcal L_{s,t}(x,y)$ where $0\leq s<t\le 1$ and $x,y \in (a,b)$. Then $$|\Cov(F,G)| \leq  Ce^{-dk^3}\mathbb{E}[F^4]^{1/4}\mathbb{E}[G^4]^{1/4},$$ where $C,d>0$ are independent of $k$ and $F,G \in L^4(\Omega)$ that are respectively measurable with respect to $\mathcal{G}_{-\infty,0}$ and $\mathcal{G}_{k,\infty}.$ 
\end{prop}

The $\alpha$-mixing property is a very strong form of mixing that is uniform over all observables, see 
the survey by Bradley \cite{bradley} for additional discussion. Roughly speaking, the result says that under translation of the landscaped by $k$ spatial units, the decay of correlation is bounded above by $e^{-d k^3}$, with $d$ being uniform over all observables and $k>0$. Moreover, the remarkably fast cubic-exponential rate of decay  means that we can treat the field as virtually independent on well-separated scales. 

\begin{proof} By translation invariance, it suffices to prove the claim with the assumption ``$F$ measurable with respect to $\mathcal G_{-\infty,0}$" replaced by the statement ``$F$ measurable with respect to $\mathcal G_{-\infty,-k}$" (this replacement is just for notational convenience). Therefore, we will consider $F,G$ that are respectively measurable with respect to $\mathcal G_{-\infty,-k}$ and $\mathcal G_{k,\infty}.$ 

We assume that $\mathcal L^0,\mathcal L^1,\mathcal L^2$ are the same couplings and $X$ is the same random variable as in \cref{l:mixing}. We will use $\mathbb{E}_{\mathrm{triple}}[\cdot]$ to denote the expectation with respect to the measure on the coupled space from \cref{l:mixing}. Consider the measurable functions $F,G: C(\Bbb R^4_\uparrow)\to \Bbb R$ as above, and write $F_i:=F(\mathcal L^i),G_i:=G(\mathcal L^i)$ for $i=0,1,2$. Then by independence, $\mathbb{E}_{\text{triple}}[F_2G_1] =\mathbb{E}_{\text{triple}}[F_2]\mathbb{E}_{\text{triple}}[G_1],$ thus 
\begin{align}
    |\Cov(F,G)| &= \notag |\mathbb{E}_{\text{triple}}[ F_0G_0-F_2G_1]| \\ &\leq \notag |\mathbb{E}_{\text{triple}}[ (F_0-F_2+F_2) (G_0-G_1+G_1) -F_2G_1 ]| \\ \notag &= |\mathbb{E}_{\text{triple}}[(F_0-F_2)(G_0-G_1)] +\mathbb{E}_{\text{triple}}[(F_0-F_2)G_1] + \mathbb{E}_{\text{triple}}[F_2(G_0-G_1)]| \\ \notag &\leq \mathbb{E}_{\text{triple}}[(F_0-F_2)^2]^{1/2} \mathbb{E}_{\text{triple}}[(G_0-G_1)^2]^{1/2} + \mathbb{E}_{\text{triple}}[(F_0-F_2)^2]^{1/2} \mathbb{E}_{\text{triple}}[G_1^2]^{1/2} 
    \\  & \;\;\;\;\;\;+ \mathbb{E}_{\text{triple}}[F_2^2]^{1/2} \mathbb{E}_{\text{triple}}[(G_0-G_1)^2]^{1/2}. \label{e:triple}
\end{align}
Now recall that $F$ and $G$ are measurable with respect to $\mathcal G_{-\infty,-k}$ and $\mathcal G_{k,\infty}$, respectively. Therefore, $F_0=F_2$ and $G_0=G_1$ on the event that $X \le k$. By \cref{l:mixing} we thus have that \begin{align*}\mathbb{E}_{\text{triple}}[(F_0-F_2)^2]^{1/2}&=  \mathbb{E}_{\text{triple}}[(F_0-F_2)^2\mathbf 1_{\{X>k\}}]^{1/2}\\&\leq 2 \mathbb{E}[F^4]^{1/4} \mathbb{P}_{\text{triple}}(X>k)^{1/2} \leq 2\mathbb{E}[F^4]^{1/4}Ce^{-\frac12 d k^3}.
\end{align*} 
We applied the bound $(F_0-F_2)^2 \leq 2(F_0^2+F_2^2)$ and then Cauchy-Schwarz to obtain the first inequality. Likewise, we may prove that $\mathbb{E}_{\text{triple}}[(G_0-G_1)^2]^{1/2} \leq 2\mathbb{E}[G^4]^{1/4} Ce^{-\frac12 dk^3}$. Plugging these bounds back into \eqref{e:triple}, and using Jensen to say that $\mathbb{E}_{\text{triple}}[F_2^2]^{1/2} \leq \mathbb{E}_{\text{triple}}[F_2^4]^{1/4} = \mathbb{E}[F^4]^{1/4}$ and likewise $\mathbb{E}_{\text{triple}}[G_1^2]^{1/2} \leq \mathbb{E}_{\text{triple}}[G_1^4]^{1/4} = \mathbb{E}[G^4]^{1/4}$, we will obtain the stated bound. 
\end{proof}

The appearance of $L^4$ norms on the right side is equivalent to $\alpha$-mixing, see \cite{Peli}. It is natural to ask what the strongest or optimal form of mixing is for spatial shifts in the landscape, such as $\beta$-mixing or $\phi$-mixing \cite{bradley}. We did not pursue this question. 

Although \cref{l:mixing} and \cref{p:alphamixingestimate} provide a flavor of how to prove mixing estimates, we will actually need a stronger version of these results. 

\begin{thm}\label{mixing}
    Fix $T>0$, and consider any compact disjoint intervals $I_1,J_1,...,I_N,J_N\subset \mathbb{R}$. Define the $\sigma$-fields $$\mathcal F_I:= \sigma( \{\mathcal L_{s,t}(x,y): 0\leq s<t\le T, x\text{ and }y \text{ are in the same interval } I_r \text{ for some } r\le N\},$$
    $$\mathcal F_J:= \sigma( \{\mathcal L_{s,t}(x,y): 0\leq s<t\le T,  x\text{ and }y \text{ are in the same interval } J_r \text{ for some } r\le N\}.$$
    Define the positive real number $$D:= \inf\bigg\{|x-y| \bigg| x\in \bigcup_1^N I_r,\;\; y\in \bigcup_1^N J_r\bigg\}.$$
    Then we have the bound $$|\Cov(F,G)| \leq  CNe^{-dT^{-2}D^3} \mathbb{E}[F^4]^{1/4} \mathbb{E}[G^4]^{1/4}.$$
    Here $C,d>0$ are universal constants that are independent of $N,T,\{I_r,J_r\}_{r=1}^N,D$ and $F,G \in L^4(\Omega)$ that are respectively measurable with respect to $\mathcal F_I$ and $\mathcal F_J.$
\end{thm}

This proof will follow the framework of the proof of \cref{l:mixing}. We will define three fields of random variables, $\omega^1,\omega^2,\omega^0$. Here, 
$\omega^{0}$ will agree with $\omega^{1}$ on some enlargement of $\bigcup_{r=1}^{N} I_{r}$, and with with $\omega^{2}$ on an enlargement of $\bigcup_{r=1}^{N} J_{r}$. A stronger mixing lemma of this form is essential to the arguments in \cref{s:prooft1} because sets of this form appear when we consider the sets where $f\circ\mathcal{L}_{a,b}\circ g$ and $u\circ \mathcal{L}_{a,b}\circ v$ are jointly close their maxima. We remark that not all $I_r$ need to be strictly to the left of all $J_r$ or vice versa: they may very well be mixed amongst each other in a nontrivial way.

\begin{proof}
Without loss of generality, we may assume that $T=1$. This is because by scale invariance of the directed landscape \cref{d:dl} (4), all of the intervals could be replaced by their $T^{-2/3}$-rescalings. changes the value of $D$ to $T^{-2/3}D$, which respects the estimate.

In the case that $T=1$, the proof uses an explicit coupling construction that is very similar to the proofs of \cref{l:mixing} and \cref{p:alphamixingestimate}. We now sketch the details of this coupling construction. Without loss of generality, we assume that the intervals are ordered $I_1,J_1,I_2,J_2,...,I_N,J_N$ so that the supremum of $I_1$ is less than or equal to the infimum of $J_1$, the supremum of $J_1$ is less than the infimum of $I_2$, and so on. This does not lose any generality because we can always combine consecutive instances of the $I_r$ or $J_r$ without changing the value of $D$.

We relabel $I_1,J_1,I_2,J_2,...,I_N,J_N$ as $U_1,...,U_{2N}$ and define $a_r$ to be the midpoint of the segment in the complement of $\bigcup_{i\le 2N} U_i$ that is between $U_r$ and $U_{r+1}$. We define a fattening of $I_{r}$ and $J_{r},$ by setting $\bar I_r:= (a_{2r-1},a_{2r})$ and $\bar J_r:= (a_{2r},a_{2r+1})$ for $r=1,...,N$. In this notation, we say that $a_{2N+1}:=\infty$. We also define  $b_r,c_r$ to be those points in the segment in the complement of $\bigcup_{i\le 2N} U_i$ that is between $U_r$ and $U_{r+1}$ that are (respectively) exactly one-third and two-thirds of the way across this segment. 

If $L^1,L^2\in C(\mathbb{R}_\uparrow^4)$ and $I\subset \mathbb{R}$, 
then we will say that $L^1,L^2$ agree on $I\times[0,T]$ if $L^1_{a,b}(x,y)=L^2_{a,b}(x,y)$ for all $0\le a<b\le T$ and all $a,b\in I.$ Just as we did in \cref{l:mixing}, we construct three copies $\mathcal L^0,\mathcal L^1,\mathcal L^2$ of the directed landscape, so that $\mathcal L^1, \mathcal L^2$ are independent, and furthermore, with high probability, $\mathcal L^0$ agrees with $\mathcal L^1$ on $ \bigcup_{r=1}^N I_r \times [0,T]$ and $\mathcal L^0$ agrees with $\mathcal L^2$ on $ \bigcup_{r=1}^N J_r \times [0,T]$.

Consider fields of IID exponential random variables $\omega^i = \{\omega^i(t,x)\}_{t\ge 0,x\in \mathbb{Z}}$ for $i=1,2$, independent of one another. Define a third field of IID variables $\omega_\varepsilon^0=\{\omega_\varepsilon^0(t,x)\}_{t\ge 0,x\in \mathbb{Z}}$ which agrees with $\omega^1$ on $\mathbb{Z}_{\ge 0} \times \bigcup ((\varepsilon^{-2/3} \bar I_r)\cap \mathbb{Z})$ and agrees with $\omega^2$ on $\mathbb{Z}_{\ge 0} \times \bigcup ((\varepsilon^{-2/3} \bar J_r)\cap \mathbb{Z})$. Just as we did in \cref{l:mixing}, we take a joint limit point of the last passage times for all three fields (together with all of the geodesic paths associated to the finite collection of endpoints given by $b_1,c_1,...,b_N,c_N$) as $\varepsilon\to 0$. We will obtain three copies of the directed landscape $\mathcal L^0,\mathcal L^1, \mathcal L^2$ coupled onto the same space. It is important to note that the geodesics of the prelimit converge jointly with each of the three copies of the landscapes to the geodesics of the limiting landscape, just as in the proof of \cref{l:mixing}. 

\begin{figure}[t]
{\centering{\includegraphics[scale=0.57]{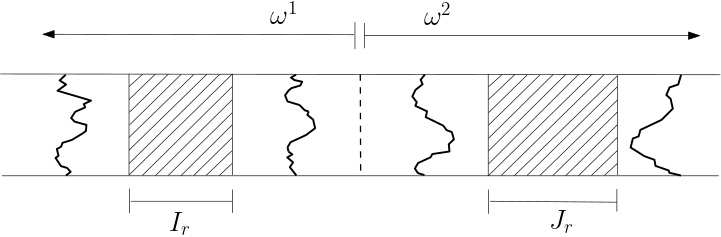}}
\caption{\label{f:1} A visual depiction of the ``geodesic separation" argument in the proof of \cref{mixing}. We view the horizontal axis as a spatial axis, and likewise view the vertical axis as a temporal one. If all of the prelimiting geodesic paths from $b_r$ to $b_r$ and from $c_r$ to $c_r$ do not intersect the dotted lines or the endpoints of the intervals, then the non-crossing property of the geodesics in the prelimiting model implies that all last passage values between two space-time points inside the left shaded region will agree with those determined by $\omega^1$, while all last passage values between two space-time points inside the right shaded region will agree with those determined by $\omega^2$.}}
\end{figure}

Consider the event $E$ in which all of the $4N$ geodesic paths from the points $b_r$ to $b_r$ and $c_r$ to $c_r$  in $\mathcal L^0$ do not cross any of the values $a_r$, and also do not cross the endpoints of the intervals $I_r,J_r$ (which in particular implies that they do not intersect). By the same argument as in \cref{l:mixing}, the values $\mathcal L^0_{s,t}(x,y)$ will agree with either of $\mathcal L^1$ or $\mathcal L^2$ as long as $x,y$ lie strictly in between any two given consecutive geodesics. See \cref{f:1} for a geometric depiction of this event. 

To bound the event that at least one of the $4N$ geodesics traverses a distance of more than $D$, we take a union bound over this event for all $4N$ geodesics. We then apply \cref{dov1.7} to obtain a total bound of $(4N) \cdot Ce^{-dD^3}$. This means that the event $E$ occurs with probability greater than $1-C(4N)e^{-dD^3}$, where $C,d$ do not depend on $N,D$ or the choice of intervals. Consequently, by repeating the exact argument as that in \cref{p:alphamixingestimate} (replacing the event $\{X>k\}$ appearing there with the event $E$ here) we obtain the required bound. 
\end{proof}

Specializing \cref{mixing} to the case $T=1$ and projecting the result to the two-time marginal $(s,t)=(0,1)$, we obtain the following corollary about strong mixing of the Airy sheet under diagonal shifts.

\begin{cor}[Generalized strong mixing property] \label{c:diagonalmixing}
    Consider the Airy sheet $\mathcal S = \mathcal L_{0,1}$. Fix $T>0$, and consider any compact disjoint intervals $I_1,...,I_N,J_1,...,J_N\subset \mathbb{R}^2$. Define the $\sigma$-fields 
    \begin{align*}\mathcal G_I & := \sigma( \{\mathcal S(x,y) |  (x,y)\in I_r\times I_r \text{ for some } r\le N\}),
    \\ \mathcal G_J & := \sigma( \{\mathcal S(x,y) |   (x,y)\in J_r\times J_r \text{ for some } r\le N\}).
    \end{align*}
    Define $$D:= \inf\bigg\{|x-y| \bigg| x\in \bigcup_1^N I_r,\;\; y\in \bigcup_1^N J_r\bigg\}.$$
    Then we have the bound $$|\Cov(F,G)| \leq  CNe^{-dD^3} \mathbb{E}[F^4]^{1/4} \mathbb{E}[G^4]^{1/4}.$$
    Here $C,d>0$ are universal constants that are independent of $N,\{I_r,J_r\}_{r=1}^N,D$ and $F,G \in L^4(\Omega)$ that are respectively measurable with respect to $\mathcal G_I$ and $\mathcal G_J.$
\end{cor}

%

\section{Useful estimates on Airy sheets and argmaxes}\label{s:airy}

We begin by establishing some useful notation. 

\begin{defn}
For $f,g\in C(\mathbb R)$ and $L\in C(\mathbb R^2)$ we define 
$$f\circ L\circ g:= \sup_{x,y\in\mathbb{R}} f(x)+L(x,y)+g(y).$$ 
\end{defn}
In this section, we prove bounds 
(Theorems \ref{tbound1} $-$ \ref{tbound3}) on the covariance of $f\circ \mathcal{L}_{0,\epsilon}\circ g$ and $u\circ\mathcal{L}_{0,\epsilon}\circ v$ for deterministic functions $f,g,u,v$ all lying within a certain restricted domain, more precisely those functions $f,g,u,v$ lying in between the graphs of the functions $-2 x^2 -B$ and $-\frac12 x^2+B$ for some $B>0.$ {If $B$ is large, then with high probability, this domain includes the sample paths of parabolic Airy$_2$ processes, and intuitively we should think of the results of this section as estimates that will eventually be applied to those Airy$_2$ sample paths.} The reason for this was explained at the end of \cref{s:noisetoairy}: {our desired variance bound, \eqref{e:bncondfinal}, depends only on the part of $\mathcal{L}_{0,1}(x,y)$ which is measurable with respect to $\mathcal{F}_{a,b}^{\mathrm{DL}}$. Therefore by breaking $\mathcal{L}_{0,1}(x,y)$ into three components to isolate the part which lies on $\mathcal{F}_{a,b}^{\mathrm{DL}}$,} we obtain an expression in terms of integrals with respect to independent Airy$_2$ paths. Some of the bounds proved here may be of independent interest, since they give fairly sharp regularity estimates on the landscape, and on where the argmax is achieved in a given metric composition of the form $f\circ \mathcal L_{0,\epsilon}\circ g$. 



We introduce several quantities which will be important to our argument. 
\begin{defn} Let $f,g: \mathbb{R}\to\mathbb{R}$ be continuous functions decaying parabolically at infinity, that is, $-Ax^2 -B \leq f(x),g(x) \leq -Cx^2-D$ for some $A>C>0$ and $B,D\in\mathbb{R}.$ Assume that $f,g$ achieve their maximum values at unique input values $z_{\max}^f, z_{\max}^g$. 
For $\epsilon,M>0$ and $\alpha \in (0,1)$, we define
\begin{align}
\mathscr S(f;\epsilon)&:= \{z\in \mathbb{R}: f(z) \geq f(z_{\max}^f)-\epsilon\}, \label{sfe} \\
\label{gamm}
\Gamma(f,g;\epsilon)&:= \inf\{ |z-z'|: |f(z)-f(z_{\max}^f)|<\epsilon,\;\; |g(z')-g(z_{\max}^g)|<\epsilon\},\\
\mathfrak{Hol}(g,\alpha;M)&:= \sup_{\substack{s,t\in [-M,M]\\s\neq t}} \frac{|g(t)-g(s)|}{|t-s|^\alpha}. \nonumber
\end{align}
We define $N_\epsilon(f)$ to be the smallest number of intervals of length $\epsilon^2$ required to cover $\mathscr S(f;\epsilon).$ 
\end{defn}

In other words, $\mathscr{S}(f;\epsilon)$ is the set where $f(z)$ lies within $\epsilon$ of its maximum. $\Gamma(f,g;\epsilon)$ is the minimum distance between $\mathscr{S}(f;\epsilon)$ and $\mathscr{S}(g;\epsilon)$ and $\mathfrak{Hol}(g,\alpha;M)$ is the $\alpha$-H\"older seminorm of $g$ on the interval $[-M,M]$. 

The following tail bound on the Airy sheet from \cite{dsv} is useful in our arguments.

\begin{lemm}[Lemma 5.3, \cite{dsv}] \label{l:airyabsbound} There exists a deterministic $c>0$ and a random $\mathfrak{C}>0$ such that, uniformly over all $(x,y)\in\Bbb R^2$, one has the bound $$|\mathcal S(x,y)+(x-y)^2| \leq \mathfrak C+c\log^{2/3} (2+|x|+|y|).$$
Furthermore $\mathbb{E}[a^{\mathfrak C^{3/2}}]<\infty$ for some $a>1.$
\end{lemm}
\cref{l:airyabsbound} shows that the probability that $\mathcal S$ will make huge excursions away from the deterministic function $(x,y)\mapsto -(x-y)^2$ is very small. Since the parabolic Airy$_{2}$ process embeds into the Airy sheet by setting $y=0$, we can also apply this lemma to obtain bounds on the Airy$_2$ process. This will be very useful at many points in \cref{s:prooft1}. In this section, we use this lemma to prove a helpful property (\cref{apriori}) of the point $(x,y)$ which maximizes $f\circ\mathcal{L}_{0,\epsilon}\circ g$.  
\begin{prop} \label{apriori} Let $f,g: \mathbb{R}\to \mathbb{R}$ be deterministic continuous functions satisfying $-2 x^2 -B \leq f(x),g(x) \leq -\frac12 x^2+B$ for some $B>0.$ For $\epsilon \in (0,1)$, we define
the random variable 
\begin{align*}(X_{\epsilon},Y_\epsilon) := \mathrm{argmax}_{(x,y)\in \mathbb{R}^2} \;\;f(x) + \mathcal L_{0,\epsilon}(x,y)+g(y).
\end{align*}
We define $z_{\max}^{f+g}:= \mathrm{argmax}_{z\in \mathbb{R}} f(z)+g(z),$ and assume that this point is unique. Then for $M^2>1+100B$, $$\mathbb{P}( X_\epsilon^2+Y_\epsilon^2> M^2) \leq C e^{-d\epsilon^{-1/2}M^3},$$ where $C,d>0$ are universal constants not depending on $f,g,M,\epsilon, B$. 
\end{prop}
We will later build on this bound to estimate the probability that $X_{\epsilon}$ and $Y_{\epsilon}$ are far apart. 
In particular, we will be able to show that these values will, with high probability, lie within $\epsilon^{2/3 -\delta}$ of one another (with some further assumptions like H\"older regularity of $f,g)$. 
\begin{proof}
    By the scale invariance of the directed landscape, we can replace $\mathcal L_{0,\epsilon}(x,y)$ by the process $$\epsilon^{1/3} \mathcal L_{0,1} (\epsilon^{-2/3} x,\epsilon^{-2/3}y) = \epsilon^{1/3} \mathcal S (\epsilon^{-2/3} x,\epsilon^{-2/3}y).$$ Applying \cref{l:airyabsbound}, we have 
    \begin{equation}|\epsilon^{1/3} \mathcal S (\epsilon^{-2/3} x,\epsilon^{-2/3}y)+ \epsilon^{-1} (x-y)^2|\leq  \mathfrak C \epsilon^{1/3}+ c \epsilon^{1/3} \log^{2/3}(2+\epsilon^{-2/3}(|x|+|y|)), \label{e:1}
    \end{equation}where $c$ is deterministic and $\mathfrak C$ is a non-negative random variable satisfying $\mathbb{P}( \mathfrak C>u) \leq Ce^{-du^{3/2}}$ for $u \ge 0$. From now on, we use the notation $\mathcal S^\epsilon(x,y):=\epsilon^{1/3} \mathcal S (\epsilon^{-2/3} x,\epsilon^{-2/3}y)$. By \eqref{e:1},
    \begin{align*}\begin{split}(f+g)(z_{\max}^{f+g}) - \mathfrak C \epsilon^{1/3} - c\epsilon^{1/3}\log^{2/3}(2+2\epsilon^{-2/3}|z_{\max}^{f+g}|) & \leq f(z_{\max}^{f+g}) + \mathcal S^\epsilon(z_{\max}^{f+g},z_{\max}^{f+g}) +g(z_{\max}^{f+g}) \\& \leq f(X_\epsilon) + \mathcal S^\epsilon(X_\epsilon,Y_\epsilon) +g(Y_\epsilon).\end{split}
    \end{align*}

    By the assumptions of the proposition, we conclude that on the event $X_\epsilon^2+Y_\epsilon^2>M^2$,
    $$f(X_\epsilon)+g(Y_\epsilon) \leq 2B - \frac12(X_\epsilon^2+Y_\epsilon^2)\leq 2B-\frac12 M^2. $$ 
    Combining the previous two bounds results in
    \begin{align}\label{e:expres1}
        \mathcal{S}^{\epsilon}(X_{\epsilon},Y_{\epsilon}) \geq (f+g)(z_{\max}^{f+g}) - \mathfrak{C}\epsilon^{1/3} - 2B+\frac{1}{2}M^{2} - c\epsilon^{1/3}\log^{2/3}(2+2\epsilon^{-2/3}|z_{\max}^{f+g}|).
    \end{align}
    Likewise, \cref{l:airyabsbound} implies 
    \begin{align}\label{ee:expres2}\mathcal S^\epsilon(X_\epsilon,Y_\epsilon) \leq \mathfrak C \epsilon^{1/3} + c\epsilon^{1/3}\log^{2/3} (\epsilon^{-2/3}(2+ |X_\epsilon|+|Y_\epsilon|)).
    \end{align}
     From the conditions imposed on $f,g$ by the statement of the proposition, we note that $z_{\max}^{f+g}\in[-2\sqrt{B},2\sqrt{B}]$. Combining this with \eqref{e:expres1} and \eqref{ee:expres2} we obtain a bound on $\mathfrak{C}$, 
     \begin{align*}2\mathfrak C \epsilon^{1/3} &\geq (f+g)(z_{\max}^{f+g}) - 2B+\frac12 M^2 +c\epsilon^{1/3}\bigg(- \log^{2/3} (2+ \epsilon^{-2/3}(|X_\epsilon|+|Y_\epsilon|)) - \log^{2/3} (2+2\epsilon^{-2/3}|z_{\max}^{f+g}|)\bigg) \\ &\geq -(z_{\max}^{f+g})^2 -4B +\frac12 M^2 +c\epsilon^{1/3}\bigg(- \log^{2/3} (2+ \epsilon^{-2/3}(|X_\epsilon|+|Y_\epsilon|)) - \log^{2/3} (2+2\epsilon^{-2/3}|z_{\max}^{f+g}|)\bigg) \\ &\geq -8B +\frac12M^2 +c\epsilon^{1/3}\bigg(- \log^{2/3} (2+ \epsilon^{-2/3}(|X_\epsilon|+|Y_\epsilon|)) - \log^{2/3} (2+2\epsilon^{-2/3}B)\bigg).
     \end{align*}
     The deterministic constant $c$ may grow larger from line to line. Therefore, for sufficiently small $\epsilon>0$ (below some threshold that is deterministic and independent of $B\ge 2$), $c \epsilon^{1/3}\log^{2/3} (2+2\epsilon^{-2/3}B) \leq B$. 
     
     Under the additional assumption that $X_\epsilon^2+Y_\epsilon^2 \leq (2M)^2 $, we find that $c \epsilon^{1/3}\log^{2/3} (2+ \epsilon^{-2/3}(|X_\epsilon|+|Y_\epsilon|)) \leq \frac14 M^2$ for sufficiently small $\epsilon>0$ (below some threshold that is deterministic and independent of $M\ge 2$, for instance, any $\epsilon$ such that $c\epsilon^{2/3} \log^{2/3} (6\epsilon^{-2/3}) \leq \frac14 $). Therefore, the last expression admits a lower bound by $-9B +\frac14 M^2.$ We have shown that $$\mathbb{P}( M^2 \leq X_\epsilon^2 +Y_\epsilon^2 \leq (2M)^2) \leq \mathbb{P}( 2\mathfrak C \epsilon^{1/3} \geq -9B +\frac14M^2),$$ for $\epsilon\leq \epsilon_0$ with $\epsilon_0$ deterministic and independent of $B,M.$ Summing over dyadically sized concentric annuli, we obtain 
     \begin{align*}
         \mathbb{P}(X_\epsilon^2 +Y_\epsilon^2 >M^2) &= \sum_{k=0}^\infty \mathbb{P}( 4^k M^2 < X_\epsilon^2 +Y_\epsilon^2 \leq 4^{k+1}M^2 ) \\ &\leq \sum_{k=0}^\infty \mathbb{P}( 2\mathfrak C \epsilon^{1/3} \geq -9B + 4^{k-1} M^2) \\ &\leq \sum_{k=0}^\infty Ce^{-d\epsilon^{-1/2}(4^{k-1}M^2 - 9B)^{3/2} } \\ &\leq \sum_{k=0}^\infty Ce^{-d\epsilon^{-1/2} 8^{k-2} (M^2-36B)^{3/2} }.
         \end{align*}
      For the final inequality, we used the bound $9B\leq 9\cdot 4^k B$. We can apply the bound $ e^{-u} \leq 2u^{-1} e^{-u/2} \leq 2u^{-1} e^{-u/8^k}$ to obtain a factor of $8^{-k}$ in front of the exponential terms: $e^{-d\epsilon^{-1/2} 8^{k-2} (M^2-36B)^{3/2}} \leq 2 \cdot 8^{-k}e^{-\frac12 d \epsilon^{-1/2} (M^2 -36B)^{3/2} /64} $. We note that this infinite sum is finite, and is bounded above by 
      $$C\epsilon^{1/2}(M^2-36B)^{-3/2} e^{-d \epsilon^{-1/2}(M^2-36B)^{3/2}},$$ where the value of $d$ may be smaller than in the previous display. If $M^2$ is larger than $1+100B$, then $(M^2-36B)^{-3/2}$ can be bounded above by $1$. Furthermore, $M^2>1+100B>72B$ implies that $M^2-36B >\frac12 M^2,$ and therefore $e^{-d \epsilon^{-1/2}(M^2-36B)^{3/2}} \leq e^{- d \epsilon^{-1/2} M^3},$ where once again, we can make $c$ larger and $d$ smaller if desired. 
\end{proof}
We use \cref{apriori} to obtain the following bound on $|X_{\epsilon}-Y_{\epsilon}|$.
\begin{prop}\label{4.4}
    Let $f,g: \mathbb{R}\to \mathbb{R}$ be deterministic continuous functions satisfying $-2 x^2 -B \leq f(x),g(x) \leq -\frac12 x^2+B$ for some $B>0.$ Define
the random variable 
\begin{align*}(X_{\epsilon},Y_\epsilon) := \mathrm{argmax}_{(x,y)\in \mathbb{R}^2} \;\;f(x) + \mathcal L_{0,\epsilon}(x,y)+g(y).
\end{align*}
We define $z_{\max}^{f+g}:= \mathrm{argmax}_{z\in \mathbb{R}} f(z)+g(z),$ and assume that this point is unique. Then $$\mathbb{P}(|X_\epsilon-Y_\epsilon|>\epsilon^{\frac12-\delta})\leq Ce^{-d \epsilon^{-1/2}(\epsilon^{-2\delta} - 108 B)^{3/2}} + Ce^{-d  \epsilon^{-1/2}},$$
where $C,d>0$ are universal constants independent of $f,g,B, \epsilon\in (0,1].$
\end{prop}

Intuitively, the bound $\epsilon^{\frac12-\delta}$ should not be optimal, as one would expect based on the scaling properties of the landscape that the typical value of $|X_\epsilon-Y_\epsilon|$ should be of order $\epsilon^{2/3}$. Without additional regularity assumptions on $f$ or $g$, this is simply not true. However, with additional H\"older regularity assumptions we will later show that this is indeed true, see \cref{opt} below.

\begin{proof}
    As in previous arguments, we replace $\mathcal L_{0,\epsilon}(x,y)$ by $\epsilon^{1/3} \mathcal L_{0,1} (\epsilon^{-2/3} x,\epsilon^{-2/3}y) = \epsilon^{1/3} \mathcal S (\epsilon^{-2/3} x,\epsilon^{-2/3}y)$. From \cref{l:airyabsbound}, we have the bound 
    \begin{equation}\label{e:2}|\epsilon^{1/3} \mathcal S (\epsilon^{-2/3} x,\epsilon^{-2/3}y)+ \epsilon^{-1} (x-y)^2|\leq  \mathfrak C \epsilon^{1/3}+ c \epsilon^{1/3} \log^{2/3}(2+\epsilon^{-2/3}(|x|+|y|)), 
    \end{equation}
    where $c$ is deterministic and $\mathfrak C$ is a non-negative random variable satisfying $\mathbb{P}( \mathfrak C>u) \leq Ce^{-du^{3/2}}$ for $u \ge 0$. 
    
    We estimate the probability of the event $K_\epsilon:=\{|X_\epsilon-Y_\epsilon|>\epsilon^{\frac12-\delta}\}$. On one hand, on the event $K_\epsilon$, $\epsilon^{-1}(X_\epsilon - Y_\epsilon)^2 > \epsilon^{-2\delta}$. From the conditions on $f,g$ imposed by the statement of the proposition, it follows that $|z_{\max}^{f+g}| \leq \sqrt 2 B.$ Therefore, on $K_\epsilon$ we find that \begin{align*}-\epsilon^{-2\delta}&+2B \geq -\epsilon^{-2\delta} +\sup_x f(x) + \sup_y g(y)\\&\geq f(X_\epsilon) -\epsilon^{-1}(X_\epsilon - Y_\epsilon)^2+g(Y_\epsilon) \\&\geq \big[f(X_\epsilon) +\epsilon^{1/3}\mathcal S (\epsilon^{-2/3} X_\epsilon,\epsilon^{-2/3}Y_\epsilon)+g(Y_\epsilon)\big] - \mathfrak C \epsilon^{1/3}- c \epsilon^{1/3} \log^{2/3}(2+\epsilon^{-2/3}(|X_\epsilon|+|Y_\epsilon|))\\&\geq \big[f(z_{\max}^{f+g})+\epsilon^{1/3} \mathcal S(\epsilon^{-2/3}z_{\max}^{f+g},\epsilon^{-2/3} z_{\max}^{f+g})+g(z_{\max}^{f+g})\big] - \mathfrak C \epsilon^{1/3}- c \epsilon^{1/3} \log^{2/3}(2+\epsilon^{-2/3}(|X_\epsilon|+|Y_\epsilon|)) \\&\geq f(z_{\max}^{f+g})+g(z_{\max}^{f+g}) -\bigg(\mathfrak C\epsilon^{1/3} + c\epsilon^{1/3} \log^{2/3} (2 +2\epsilon^{-2/3} |z_{\max}^{f+g}|)\bigg) 
    \\ & \;\;\;\;\;- \mathfrak C \epsilon^{1/3}- c \epsilon^{1/3} \log^{2/3}(2+\epsilon^{-2/3}(|X_\epsilon|+|Y_\epsilon|)) \\&\ge (-B - \frac12(z_{\max}^{f+g})^2)+(-B - \frac12(z_{\max}^{f+g})^2) -2\mathfrak C\epsilon^{\frac13} 
    \\ & \;\;\;\;\;- c\epsilon^{1/3} \bigg( \log^{2/3} (2+2\epsilon^{2/3}|z_{\max}^{f+g}|) + \log^{2/3} ( 2+\epsilon^{-2/3}(|X_\epsilon|+|Y_\epsilon|) ) \bigg) \\ &\geq -4B -2\mathfrak C \epsilon^{1/3}  -c\epsilon^{1/3} \bigg( \log^{2/3} (2+2\epsilon^{-2/3}B)  + \log^{2/3} ( 2+\epsilon^{-2/3}(|X_\epsilon|+|Y_\epsilon|) )\bigg).
    \end{align*} The first inequality is true by assumptions on $f,g$. The second line was explained in the text preceding this chain of inequalities. The third line follows from \eqref{e:2}. The fourth line is true since $(X_\epsilon,Y_\epsilon)$ is the argmax of $(x,y)\mapsto f(x)+\epsilon^{1/3} \mathcal S(\epsilon^{-2/3} x,\epsilon^{-2/3} y) + g(y)$. The fifth inequality used \eqref{e:2} again. The next-to-last inequality uses the bounds which we have assumed on $f,g$ in the statement of the proposition. The last line follows from  $|z_{\max}^{f+g}| \leq \sqrt 2 B,$ which was explained above.
    
    Since $c$ and $B$ are deterministic, for sufficiently small deterministic $\epsilon>0$ (below some threshold independent of $B\ge 2$ say) we have an elementary bound $c\epsilon^{1/3} \log^{2/3} (2+2\epsilon^{-2/3}B) \leq 2B$. We define the event $$H_{\epsilon}:=\{|X_\epsilon|+|Y_\epsilon| \leq 100B\}.$$ Applying \cref{apriori} with $M=100B >(1+100B)^{1/2}$, we find that $\left(1-\mathbb{P}(H_\epsilon)\right) \leq Ce^{-d \epsilon^{-1/2}}.$ On the event $H_\epsilon$ we likewise have that $c\epsilon^{1/3}\log^{2/3} ( 2+\epsilon^{-2/3}(|X_\epsilon|+|Y_\epsilon|) )\leq 100B$, again assuming that $\epsilon$ is below some sufficiently small positive deterministic threshold. On $K_\epsilon\cap H_\epsilon$ we find that $2\mathfrak C \epsilon^{1/3} \geq \epsilon^{-2\delta} -108B .$ Therefore, we obtain 
    \begin{align*}\mathbb{P}(K_\epsilon) &\leq \mathbb{P}(H_\epsilon\cap K_\epsilon)+\mathbb{P}(H_\epsilon^c) \\&\leq \mathbb{P}( 2\mathfrak C \epsilon^{1/3}\geq \epsilon^{-2\delta} -108B) + Ce^{-d \epsilon^{-1/2}} \\ &\leq Ce^{-d \epsilon^{-1/2}(\epsilon^{-2\delta} - 108 B)^{3/2}} + Ce^{-d  \epsilon^{-1/2}}.
    \end{align*}
    This completes the proof. 
\end{proof}

%
As mentioned previously, we can actually obtain a much better estimate on $|X_{\epsilon}-Y_{\epsilon}|$ if we impose some additional H\"older regularity assumptions on $f,g$. Since we ultimately plan to apply these results to sample paths of Airy$_2$ processes, these H\"older regularity assumptions will be guaranteed in the cases where we will apply this lemma in \cref{s:prooft1}. We will rely heavily on this near-optimal bound, in order to prove an estimate on the $L^{p}$ norm of $|f\circ \mathcal{L}_{0,\epsilon}\circ g -\max{(f+g)}|$, which in turn will be needed in the proofs of the crucual results \cref{tbound1}, \cref{tbound2}, and \cref{tbound3}.

\begin{prop}\label{opt} Let $f,g: \mathbb{R}\to \mathbb{R}$ be deterministic functions satisfying $-2x^2 -B \leq f(x),g(x) \leq -\frac12x^2+B$ for some $B\ge 1.$ Define
the random variable 
\begin{align*}(X_{\epsilon},Y_\epsilon) := \mathrm{argmax}_{(x,y)\in \mathbb{R}^2} \;\;f(x) + \mathcal L_{0,\epsilon}(x,y)+g(y).
\end{align*}
We also denote $z_{\max}^{f+g}:= \mathrm{argmax}_{z\in \mathbb{R}} f(z)+g(z),$ which we assume is unique. Fix $\delta\in (0,1/3]$. Then 
\begin{align*}
\mathbb{P}\big(|X_\epsilon-Y_\epsilon|>\epsilon^{\frac23-\delta}\big)& \leq Ce^{-d  \epsilon^{-9\delta/10}}+\mathbf 1_{\{ B\geq d\epsilon^{-1/3} \}} +\mathbf 1_{\big\{\mathfrak{Hol}\big(g,\;\frac12(1-\frac3{10}\delta),\; 100B\big) >\epsilon^{-\delta/10}\big\}},
\end{align*}
where $C,d>0$ may depend on $ \delta$ but are independent of $f,g,B$ and $\epsilon\in (0,1]$.
\end{prop}

The exponent of $\frac23-\delta$ on the left side is optimal. However the various constants on the right side, such as $3/10,2/5,6/5,100$, are not optimized. Here and in later propositions, we do not attempt to optimize such constants. The current values suffice for arguments in \cref{s:prooft1}.

\begin{proof}
    By scale invariance of the directed landscape we can replace $\mathcal L_{0,\epsilon}(x,y)$ by $\epsilon^{1/3} \mathcal L_{0,1} (\epsilon^{-2/3} x,\epsilon^{-2/3}y) = \epsilon^{1/3} \mathcal S (\epsilon^{-2/3} x,\epsilon^{-2/3}y)$, where $\mathcal S$ is the Airy sheet as usual. Applying \cref{l:airyabsbound}, we have 
    \begin{equation}\label{e:3}|\epsilon^{1/3} \mathcal S (\epsilon^{-2/3} x,\epsilon^{-2/3}y)+ \epsilon^{-1} (x-y)^2|\leq  \mathfrak C \epsilon^{1/3}+ c \epsilon^{1/3} \log^{2/3}(2+\epsilon^{-2/3}(|x|+|y|)), \end{equation} where $c$ is deterministic and $\mathfrak C$ is a non-negative random variable satisfying $\mathbb{P}( \mathfrak C>u) \leq Ce^{-du^{3/2}}$ for $u \ge 0$. 
    We choose $v:= \frac{3\delta}{10-6\delta}> \frac3{10}\delta$ and define an inductive sequence of exponents $\kappa_n$ for $n=0,1,2,3,...$ as follows: $$\kappa_0=\frac13\;\; \text{ and }\;\; \kappa_{n+1} := \frac{1-v}2 +\frac12 \bigg(\frac12 -v\bigg)\kappa_{n}.$$
    As $n\to\infty$, $\kappa_n\uparrow \frac23-\frac12 \delta$. This convergence occurs exponentially quickly, by standard iteration theory for linear recurrences. 
    For $n=0,1,2,3,...$ define the events $$K^{(n)}_\epsilon:= \{ |X_\epsilon-Y_\epsilon| \le \epsilon^{\kappa_n}\}.$$
    Fix some large $N\in \Bbb N$. We will now prove by induction on $n$ that for all $n\ge 1$ and $\epsilon\in (0,1]$, 
    \begin{equation}
        \label{e:induction} 
        \mathbb{P}(K_{\epsilon}^{(n-1)}) \leq \mathbf 1_{\{ B\geq d_n\epsilon^{-1} \}} + \mathbf 1_{\{\mathfrak{Hol}(g;\frac12-\frac{v}2; 100B)> \epsilon^{-\delta\kappa_{n-1}/2}\}} + C_ne^{-d_n\epsilon^{-3v}}.
    \end{equation} Here $C_n,d_n$ may depend on $n$ and on $\delta$ but not on $\epsilon,f,g,B$. If we can prove this, then the proposition will be proved simply by taking $n$ large enough.
    
    Our base case is $n=0$. By using \cref{4.4} and the fact that $\kappa_0$ is chosen to be 1/3, $$\mathbb{P}( K^{(0)}_\epsilon) \leq Ce^{-d \epsilon^{-1/2}(\epsilon^{-1/3} - 108 B)^{3/2}} \leq Ce^{-d\epsilon^{-5/6}} + \mathbf 1_{\{ 200B\geq \epsilon^{-1/3}\}}.$$
    The right side is bounded by an expression of the form of the right side of \eqref{e:induction}. 
    
    Now we assume such a bound as in \eqref{e:induction} holds for all steps up to $n-1.$ We note from \eqref{e:3} that
    \begin{align*}
    f(X_\epsilon) & -\epsilon^{-1}(X_\epsilon - Y_\epsilon)^2+g(Y_\epsilon) \\&\geq 
    f(X_\epsilon) +\epsilon^{1/3}\mathcal S (\epsilon^{-2/3} X_\epsilon,\epsilon^{-2/3}Y_\epsilon)+g(Y_\epsilon) - \mathfrak C \epsilon^{1/3}- c \epsilon^{1/3} \log^{2/3}(2+\epsilon^{-2/3}(|X_\epsilon|+|Y_\epsilon|))\\&\geq f(z_{\max}^{f+g})+\epsilon^{1/3} \mathcal S(\epsilon^{-2/3}z_{\max}^{f+g},\epsilon^{-2/3} z_{\max}^{f+g})+g(z_{\max}^{f+g}) - \mathfrak C \epsilon^{1/3}- c \epsilon^{1/3} \log^{2/3}(2+\epsilon^{-2/3}(|X_\epsilon|+|Y_\epsilon|)) \\&\geq f(z_{\max}^{f+g})+g(z_{\max}^{f+g}) -\bigg(\mathfrak C\epsilon^{1/3} + c\epsilon^{1/3} \log^{2/3} (2 +2\epsilon^{-2/3} |z_{\max}^{f+g}|)\bigg) 
    \\ & \;\;\;\;\;\;- \mathfrak C \epsilon^{1/3}- c \epsilon^{1/3} \log^{2/3}(2+\epsilon^{-2/3}(|X_\epsilon|+|Y_\epsilon|)) .
    \end{align*}  From the conditions on $f,g$ imposed by the statement of the proposition, it follows that $|z_{\max}^{f+g}| \leq \sqrt 2 B.$ Since $c$ is deterministic, for sufficiently small deterministic $\epsilon>0$, which does not depend on the choice of $B\ge 1$, we have an elementary bound $c\epsilon^{1/3} \log^{2/3} (2+2\epsilon^{-2/3}B) \leq C\epsilon^{\frac13 -\frac23v }B^v.$ This follows from the fact that $\log^{2/3}(2+u) \leq Cu^{v}$ for all $u\ge 1$ where $C=C(v)$ does not depend on $u.$

    We define $H_\epsilon$ to be the event $$H_{\epsilon}:=\{|X_\epsilon|+|Y_\epsilon| \leq 100B\}.$$ By \cref{apriori}, we conclude 
    \begin{equation}\label{e:hepsilon}(1-\mathbb{P}(H_\epsilon) )\leq Ce^{-d \epsilon^{-1/2}}.
    \end{equation}
    Now we consider what happens when we restrict to the event $H_\epsilon$. In this case, we have $c\epsilon^{1/3}\log^{2/3} ( 2+\epsilon^{-2/3}(|X_\epsilon|+|Y_\epsilon|) )\leq C \epsilon^{\frac13-\frac23v} B^v$. Here we have assumed again that $\epsilon$ is below some sufficiently small positive deterministic threshold independent of $B,\epsilon$. Therefore, still only considering the case that we are on the event $H_\epsilon$, we obtain the following bound
    \begin{equation}\label{fchain}f(X_\epsilon) -\epsilon^{-1}(X_\epsilon - Y_\epsilon)^2+g(Y_\epsilon) \geq f(z_{\max}^{f+g})+g(z_{\max}^{f+g}) -2\mathfrak C\epsilon^{1/3} -C\epsilon^{\frac13-\frac23v} B^v.\end{equation}

    On the event $K^{(n-1)}_\epsilon \cap H_\epsilon$, when we have $\mathfrak{Hol}(g;\frac12-\frac v2; 100B)\leq \epsilon^{-v\kappa_{n-1}/2}$ 
    then $|g(X_\epsilon)-g(Y_\epsilon)| \leq C \epsilon^{-v \kappa_{n-1} /2} |X_\epsilon-Y_\epsilon|^{\frac12-\frac v2} \leq \epsilon^{(\frac12 - v)\kappa_{n-1}} $  When we combine this bound with \eqref{fchain}, on the event $H_\epsilon\cap K^{(n-1)}_\epsilon$, we obtain the bound
    \begin{align*}f(z_{\max}^{f+g}) -\epsilon^{-1}(X_\epsilon-Y_\epsilon)^2 +g(z_{\max}^{f+g}) &\ge f(X_\epsilon) -\epsilon^{-1}(X_\epsilon-Y_\epsilon)^2 +g(X_\epsilon) \\&\geq f(X_\epsilon) -\epsilon^{-1}(X_\epsilon-Y_\epsilon)^2 +g(Y_\epsilon) - |g(X_\epsilon)-g(Y_\epsilon)| \\ &\geq f(z_{\max}^{f+g})+g(z_{\max}^{f+g}) -2\mathfrak C\epsilon^{1/3} -C\epsilon^{\frac13-\frac23v} B^v -  \epsilon^{(\frac12 -v)\kappa_{n-1}} \\ &\geq f(z_{\max}^{f+g})+g(z_{\max}^{f+g}) -2\mathfrak C\epsilon^{1/3} -C\epsilon^{(\frac12 - v)\kappa_{n-1}}B^v.
    \end{align*}
    In the final line, we use the fact that $B\ge 1$ and that $(\frac12-v)\kappa_n$ is always less than $\frac13-\frac23v.$ 
    
    As long as $\mathfrak{Hol}(g;\frac12-\frac v2; 100B)\leq \epsilon^{-v\kappa_{n-1}/2}$, we can apply the recursive definition of the $\kappa_n$ that $1+(\frac12 - v)\kappa_{n-1} = 2\kappa_n+v$ to the expression above to obtain  $$(X_\epsilon-Y_\epsilon)^2 \leq 2\mathfrak C \epsilon^{4/3} +C\epsilon^{1+(\frac12 - v)\kappa_{n-1}} B^v=2\mathfrak C \epsilon^{4/3} +C\epsilon^{2\kappa_n+v} B^v,$$ for some large enough deterministic constant $C=C(v)$ that does not depend on $B,f,g,\epsilon.$ When $\mathfrak C \leq \frac14 \epsilon^{-2v}$ 
    and $B\leq \frac12\epsilon^{-1} C^{-\frac1{v}},$ the right-hand side of the previous expression is less than $\epsilon^{2\kappa_n}$. All of this means that $\Bbb P( K_\epsilon^{(n)}) \leq \Bbb P( K_\epsilon^{(n-1)}\cap H_\epsilon) + \mathbb{P}(H_\epsilon^c)+\mathbb{P}(\mathfrak C > \frac12 \epsilon^{-2v})$. 
    
    Finally, we can combine our bounds to prove \eqref{e:induction}. To do so, we apply \eqref{e:hepsilon} to see that $\mathbb{P}(H_\epsilon^c) \leq Ce^{-d\epsilon^{-1/2}}$, and we apply \cref{l:airyabsbound} to see that $\mathbb{P}(\mathfrak C > \frac12 \epsilon^{-2v}) \leq Ce^{-d\epsilon^{-3v}}$.
\end{proof}

The next proposition tells us that the value of $f+g$ at each of the two values $X_{\epsilon}$ and $Y_{\epsilon}$ is very unlikely to be far from their joint maximum value. This proposition will be helpful in the proofs of \cref{varbound} and \cref{tbound1}, which will prove important bounds on the $L^{p}$ norm of $f\circ\mathcal{L}_{0,\epsilon}\circ g - \max{(f+g)}$ and on $\mathrm{Cov}(f\circ\mathcal{L}_{0,\epsilon}\circ g, u\circ \mathcal{L}_{0,\epsilon}\circ v)$, respectively.

\begin{prop}\label{opt2} Let $f,g: \mathbb{R}\to \mathbb{R}$ be deterministic functions satisfying $-2x^2 -B \leq f(x),g(x) \leq -\frac12x^2+B$ for some $B\ge 1.$ Define the random variable 
\begin{align*}(X_{\epsilon},Y_\epsilon) := \mathrm{argmax}_{(x,y)\in \mathbb{R}^2} \;\;f(x) + \mathcal L_{0,\epsilon}(x,y)+g(y).
\end{align*}
We also denote $z_{\max}^{f+g}:= \mathrm{argmax}_{z\in \mathbb{ R}} f(z)+g(z)$, which we assume is unique. Fix $\delta\in (0,1/3]$. Then
\begin{align*}\mathbb{P}\big(X_\epsilon \notin \mathscr S(f+g,\epsilon^{\frac13- \delta})\big) &\leq  Ce^{-d  \epsilon^{-9\delta/10}} +\mathbf 1_{\{ B\geq d\epsilon^{-1/3} \}} +\mathbf 1_{\big\{\mathfrak{Hol}\big(g,\;\frac12(1-\frac3{20}\delta),\;\; 100B\big) >\epsilon^{-\delta/20}\big\}},
\\ \mathbb{P}\big(Y_\epsilon \notin \mathscr S(f+g,\epsilon^{\frac13- \delta})\big) &\leq  Ce^{-d  \epsilon^{-9\delta/10}} +\mathbf 1_{\{ B\geq d\epsilon^{-1/3} \}} +\mathbf 1_{\big\{\mathfrak{Hol}\big(f,\;\frac12(1-\frac3{20}\delta),\;\; 100B\big) >\epsilon^{-\delta/20}\big\}},
\end{align*}
where $C,d>0$ may depend on $\delta$ but are independent of $f,g,B$ and $\epsilon\in (0,1]$. 
\end{prop}

\begin{proof}
The bound on $Y_\epsilon$ follows from the bound on $X_\epsilon$ by a symmetry argument. Therefore, we focus on the first bound.
By the scale invariance of the directed landscape, we replace $\mathcal L_{0,\epsilon}(x,y)$ by $\epsilon^{1/3} \mathcal L_{0,1} (\epsilon^{-2/3} x,\epsilon^{-2/3}y) = \epsilon^{1/3} \mathcal S (\epsilon^{-2/3} x,\epsilon^{-2/3}y)$. Applying \cref{l:airyabsbound}, we have $$|\epsilon^{1/3} \mathcal S (\epsilon^{-2/3} x,\epsilon^{-2/3}y)+ \epsilon^{-1} (x-y)^2|\leq  \mathfrak C \epsilon^{1/3}+ c \epsilon^{1/3} \log^{2/3}(2+\epsilon^{-2/3}(|x|+|y|)), $$ where $c$ is deterministic and $\mathfrak C$ is a non-negative random variable satisfying $\mathbb{P}( \mathfrak C>u) \leq Ce^{-du^{3/2}}$ for $u \ge 0$. 
We fix an arbitrary $\delta_{1}>0$ (the value of which will be specified at a later point). We define the event $$H_{\epsilon}:=\{|X_\epsilon|+|Y_\epsilon| \leq 100B\}.$$ We recall from \eqref{fchain} that on $H_\epsilon$, for $\epsilon$ below a sufficiently small positive deterministic threshold, 
    $$f(X_\epsilon) -\epsilon^{-1}(X_\epsilon - Y_\epsilon)^2+g(Y_\epsilon) \geq f(z_{\max}^{f+g})+g(z_{\max}^{f+g}) -2\mathfrak C\epsilon^{1/3} -C\epsilon^{\frac13-\frac23 \delta_{1}} B^{\delta_{1}}.$$
    We also define the event $$F_\epsilon:= \{ |X_\epsilon - Y_\epsilon| \leq \epsilon^{\frac23-\delta_{1}}\}.$$ By \cref{opt}, $F_{\epsilon}$ has probability greater than $1-Ce^{-d\epsilon^{-1/2}}$ as long as $B\leq d\epsilon^{-1/3}$ and $$\mathfrak{Hol}\big(g,\frac12(1-\delta_{2}),100B\big) \le \epsilon^{- \delta_{2}/3},$$ where $\delta_{2}:=\frac3{10}\delta_1$. In other words,
    \begin{equation}\label{e:fepsilon}
        \Bbb P(F_\epsilon^c) \leq Ce^{-d\epsilon^{-1/2}} + \mathbf 1_{\{ B>d\epsilon^{-1/3}\} } + \mathbf 1_{\{\mathfrak{Hol}\big(g,\frac12(1-\delta_{2}),100B\big) > \epsilon^{- \delta_{2}/3} \}}.
    \end{equation}
    From the expression above, if $\mathfrak{Hol}\big(g,\frac12(1-\delta_{2}), 100B\big) \le \epsilon^{- \delta_{2}/3}$, then on the event $F_\epsilon\cap H_\epsilon$ we have the bound
    \begin{align*}
        f(X_\epsilon) + g(X_\epsilon) &> f(X_\epsilon) -\epsilon^{-1}(X_\epsilon - Y_\epsilon)^2+g(X_\epsilon) \\&\ge f(X_\epsilon) -\epsilon^{-1}(X_\epsilon - Y_\epsilon)^2+g(Y_\epsilon) - |g(Y_\epsilon)-g(X_\epsilon)| \\&\geq f(X_\epsilon) -\epsilon^{-1}(X_\epsilon - Y_\epsilon)^2+g(Y_\epsilon) - \epsilon^{-\delta_{2}/3} |X_\epsilon-Y_\epsilon|^{\frac12-\delta_{2}} \\&\geq f(z_{\max}^{f+g})+g(z_{\max}^{f+g}) -2\mathfrak C\epsilon^{1/3} -C\epsilon^{\frac13-\frac23 \delta_{1}} B^{\delta_{1}} - \epsilon^{\frac13 -\delta_{2} - \frac12\delta_{1} +\delta_{1}\delta_{2}}.
    \end{align*}
    With $\delta$ as in the proposition statement, we make a specific choice of $\delta_{1}$, for instance, $\delta_{1}=\frac{1}{2}\delta$, so that $\frac23\delta_{1}<\frac12\delta$ and $\delta_{2}+\frac12\delta_{1} -\delta_{1}\delta_{2}<\frac12\delta.$ In order to have $X_\epsilon \notin \mathscr S(f+g,\epsilon^{\frac13- \delta})$ for $\epsilon<3^{-2/\delta}$, the previous expression implies that either $\mathfrak C> \tfrac16\epsilon^{-\delta}$ or $B>\tfrac16\epsilon^{-\frac{2}{3}} C^{-\frac{1}{\delta_1}}. $ Therefore, $$\mathbb{P}(X_\epsilon \notin \mathscr S(f+g,\epsilon^{\frac13- \delta}) \leq \mathbb{P}\big(\mathfrak C> \tfrac16\epsilon^{-\delta}\big) + \mathbf 1_{\{B>\tfrac16\epsilon^{-\frac23} C^{-\frac1{\delta_1}}\}} + \mathbb{P}(H_\epsilon^c) + \mathbb{P}(F_\epsilon^c). $$ To finish the proof, we apply the tail bounds on $\mathfrak C$ and the upper bounds on the latter two terms, all of which have already been discussed in \eqref{e:fepsilon} and \eqref{e:hepsilon}.
    \end{proof}

The following propositions will be one of the more important inputs to ultimately proving that the directed landscape is a black noise. It may also be of some independent interest, as it gives fairly sharp temporal regularity estimate on the landscape started from deterministic initial data. 
\begin{prop}\label{varbound} Let $f,g: \mathbb{R}\to \mathbb{R}$ be deterministic functions satisfying $-2x^2 -B \leq f(x),g(x) \leq -\frac12x^2+B$ for some $B\ge 1.$ Then for any $\delta>0$ and $p>1$,
\begin{multline*}\mathbb{E}[ |f\circ \mathcal L_{0,\epsilon} \circ g - \max(f+g)|^p]^{1/p} \\  \leq CB^{2 \delta}\epsilon^{\frac13-\delta} +CB\bigg(Ce^{-d  \epsilon^{-9\delta/10}} +\mathbf 1_{\{ B\geq d\epsilon^{-1/3} \}} +\mathbf 1_{\big\{\mathfrak{Hol}\big(g,\;\frac12(1-\frac3{40} \delta),\;\; 100B\big) >\epsilon^{-\delta/40}\big\}}\bigg).
\end{multline*} Here $C,d$ may depend on $p$ but are uniform over all $f,g,B,$ and $\epsilon\in(0,1]$.
\end{prop}

\begin{proof} As before, by scale invariance of the directed landscape, we use the substitution, $\mathcal L_{0,\epsilon}(x,y)$ by $\epsilon^{1/3} \mathcal L_{0,1} (\epsilon^{-2/3} x,\epsilon^{-2/3}y) = \epsilon^{1/3} \mathcal S (\epsilon^{-2/3} x,\epsilon^{-2/3}y)$. Applying \cref{l:airyabsbound}, we have 
\begin{equation}\label{ref}|\epsilon^{1/3} \mathcal S (\epsilon^{-2/3} x,\epsilon^{-2/3}y)+ \epsilon^{-1} (x-y)^2|\leq  \mathfrak C \epsilon^{1/3}+ c \epsilon^{1/3} \log^{2/3}(2+\epsilon^{-2/3}(|x|+|y|)), \end{equation} where $c$ is deterministic and $\mathfrak C$ is a non-negative random variable satisfying $\mathbb{P}( \mathfrak C>u) \leq Ce^{-du^{3/2}}$ for $u \ge 0$. Let $X_\epsilon,Y_\epsilon,z_{\max}$ be as in \cref{opt} and \cref{opt2}. We note that  
    \begin{align*}|f(X_\epsilon) &+\epsilon^{1/3} \mathcal S (\epsilon^{-2/3} x,\epsilon^{-2/3}y)+g(Y_\epsilon) - (f(z_{\max}^{f+g})+g(z_{\max}^{f+g}))| \\&\le |f(X_\epsilon) + g(X_\epsilon) - (f(z_{\max}^{f+g})+g(z_{\max}^{f+g}))| + |\epsilon^{1/3} \mathcal S (\epsilon^{-2/3} X_\epsilon,\epsilon^{-2/3}Y_\epsilon)+ \epsilon^{-1} (X_\epsilon-Y_\epsilon)^2| \\ &\;\;\;\;\;+ |g(Y_\epsilon)-g(X_\epsilon)| +\epsilon^{-1} (X_\epsilon-Y_\epsilon)^2.
    \end{align*} We can now define the event $$J_\epsilon:= \{ X_\epsilon \in \mathscr S(f+g,\epsilon^{\frac13- \delta})\} \cap\{|X_\epsilon-Y_\epsilon| \leq \epsilon^{\frac23-\delta}\}\cap \{|X_\epsilon|+|Y_\epsilon| \le 100B\}.$$ 
    On $J_\epsilon$, under the assumption that $\mathfrak{Hol}(g,\frac12-\delta, 100B)<{\epsilon^{-\delta/3}},$ the preceding inequality has an upper bound (term-by-term, for each of the four terms) by
    \begin{multline*}\epsilon^{\frac13-\delta} +\big( \mathfrak C \epsilon^{1/3}+ c \epsilon^{1/3} \log^{2/3}(2+\epsilon^{-2/3}(|X_\epsilon|+|Y_\epsilon|))\big) + \epsilon^{-\delta/3}|X_\epsilon-Y_\epsilon|^{\frac12-\delta} +\epsilon^{-1} (X_\epsilon-Y_\epsilon)^2 \\ \leq \epsilon^{\frac13-\delta} +(\mathfrak C\epsilon^{1/3} +\epsilon^{\frac13-\frac23\delta}B^\delta) + \epsilon^{(\frac23-\delta)(\frac12-\delta)-\frac13\delta} + \epsilon^{\frac13-2\delta}.
    \end{multline*}
    For the second term we used the following argument: since $c$ is deterministic, for sufficiently small deterministic $\epsilon>0$ which does not depend on the choice of $B\ge 1$, we have an elementary bound $c\epsilon^{1/3} \log^{2/3} (2+2\epsilon^{-2/3}B) \leq C\epsilon^{\frac13 -\frac23\delta }B^\delta.$ This bound follows immediately from the fact that $\log^{2/3}(2+u) \leq Cu^{\delta}$ for all $u\ge 1$, where $C=C(\delta)$ does not depend on $u.$ Note that all powers of $\epsilon$ in the last expression are at least $\frac13-2\delta$. Since $\mathfrak C$ is the only random term in the last expression, these bounds imply that 
    $$\mathbb{E}[ |f\circ \mathcal L_{0,\epsilon} \circ g - \max(f+g)|^p\cdot \mathbf 1_{J_\epsilon}]^{1/p} \leq CB^{\delta}\epsilon^{\frac13-2\delta}.$$ 
    Replacing $\delta$ by $2\delta$ we obtain a bound of the same form as in the proposition statement. It remains to bound what occurs outside $J_\epsilon.$ Using \cref{apriori} and \cref{opt} and \cref{opt2}, and the elementary inequality $(a+b)^{1/2}\leq a^{1/2}+b^{1/2}$, we obtain \begin{equation}\label{e:je} \mathbb{P}(J_\epsilon^c)^{1/2} \leq Ce^{-d  \epsilon^{-1/2}} +\mathbf 1_{\{ B\geq d\epsilon^{-1/3} \}} +\mathbf 1_{\big\{\mathfrak{Hol}\big(g,\;\frac12(1-\frac3{20}\delta),\;\; 100B\big) >\epsilon^{-\delta/20}\big\}}.\end{equation} By the Cauchy-Schwarz and Minkowski inequalities, we have 
    \begin{equation}\label{e:lastone}\mathbb{E}[ |f\circ \mathcal L_{0,\epsilon} \circ g - \max(f+g)|^p\cdot \mathbf 1_{J_\epsilon^c}]^{1/p} \leq \bigg( \mathbb{E}[|f\circ \mathcal L_{0,\epsilon} \circ g|^{2p}]^{1/2p} + \max(f+g)\bigg) \cdot \mathbb{P}(J_\epsilon^c)^{1/2}. 
    \end{equation} We note that $\max(f+g)\leq 2B.$ We will also show that 
    \begin{equation}\label{flg}\mathbb{E}[|f\circ \mathcal L_{0,\epsilon} \circ g|^{2p}]^{1/2p} \leq CB
    \end{equation}
    for some absolute constant $C>0$ not depending on $B,\epsilon,f,g$. As before, we write $\mathcal L_{0,\epsilon}$ as $\mathcal S^\epsilon(x,y):=\epsilon^{1/3} \mathcal S (\epsilon^{-2/3} x,\epsilon^{-2/3}y).$ We apply \eqref{ref} to obtain \begin{align*}f(x)+\mathcal S^\epsilon(x,y)+g(y) &\leq 2B-\frac12(x^2+y^2) +\mathfrak C \epsilon^{1/3}+ c \epsilon^{1/3} \log^{2/3}(2+\epsilon^{-2/3}(|x|+|y|)) \\&\leq 2B - \frac12(x^2+y^2) +\mathfrak C \epsilon^{1/3} +2+|x|+|y|,
    \end{align*}
    where the last bound is valid for sufficiently small $\epsilon>0$ below some deterministic threshold. Let $m$ denote the maximum value achieved by the function $(x,y)\mapsto 2+|x|+|y|- \frac12(x^2+y^2)$ over all of $\mathbb{R}^{2}$. If the final expression were to exceed $2B+m+u$, it would be necessary for the constant $\mathfrak C$ to exceed $u$. Therefore,
    \begin{equation}\label{akn}\mathbb{P}( f\circ \mathcal L_{0,\epsilon}\circ g> 2B+m+u) \leq Ce^{-du^{3/2}},\end{equation} which implies that  $\mathbb{E}[(f\circ \mathcal L_{0,\epsilon} \circ g)_+^{2p}]^{1/2p} $ (where $u_+:=\max\{u,0\}$) is bounded above by $B$ multiplied by an absolute constant. To obtain a bound on the negative part, note that the event $f\circ \mathcal{L}_{0,\epsilon}\circ g < -2B-u$ would require a large value for $\mathfrak{C}$
    , by considering what happens at $(z_{\max},z_{\max})$ for instance, and then using very similar arguments. This yields an analogous bound on $\mathbb{E}[(f\circ \mathcal L_{0,\epsilon} \circ g)_-^{2p}]^{1/2p} $ (where $u_- = \max\{0,-u\}$). 
    
    To finish the proof of the proposition, we plug this upper bound \eqref{flg}, and the one on $\mathbb{P}(J_{\epsilon}^{c})^{1/2}$ from \eqref{e:je} into the expression \eqref{e:lastone}. 
\end{proof}

We now come to the main result of this section, which we obtain by applying the mixing results of \cref{s:spatialmixing} with the estimates proved in this section. This covariance bound is crucial to the proof of \cref{t:1}. 

\begin{thm}[The key estimate]\label{tbound1}
    Fix $\delta\in [0,1/3].$ Consider functions $f,g,u,v$ each lying in between $-B-2x^2$ and $B-\frac12x^2$. Recall $\Gamma$ and $N_\epsilon$ as defined in \eqref{gamm}. Then there exists $C=C(\delta)>0$ \begin{align*}|\Cov(f\circ \mathcal L_{0,\epsilon} \circ g, u \circ \mathcal L_{0,\epsilon}\circ v)|  & \leq CB^{2\delta}\min\big\{ \epsilon^{\frac23 - \delta }, \big(N_{\epsilon^{\frac13-\delta}}(f+g)+N_{\epsilon^{\frac13-\delta}}(u+v)\big)e^{-d \epsilon^{-2}\Gamma\big(f+g,u+v;\epsilon^{\frac13-\delta}\big)^3}\big\} \\& \;\;\;\;\;+CB^2 \cdot( \mathscr E_\epsilon(f;B)+\mathscr E_\epsilon(g;B)+\mathscr E_\epsilon(u;B)+\mathscr E_\epsilon(v;B)). 
    \end{align*}
    Here $C,d$ are universal constants uniform over all $f,g,u,v$, $\epsilon,B$ and the ``error term" is given by $$\mathscr E_\epsilon(f,B):= e^{-d  \epsilon^{-9\delta/10}} +\mathbf 1_{\{ B\geq d\epsilon^{-1/3} \}} + \mathbf 1_{\big\{\mathfrak{Hol}\big(f,\;\frac12(1-\frac3{40} \delta),\;\; 100B\big)>\epsilon^{-\delta/40}\big\}}. $$
\end{thm}

Note that a naive (1/3)-H\"older regularity estimate of the directed landscape would only give a bound of $\epsilon^{\frac23-\delta}$. \cref{tbound1} gives a better bound. The term with a factor of $e^{-d \epsilon^{-2}\Gamma\big(f+g,u+v;\epsilon^{\frac13-\delta}\big)^3}$ is fairly sharp. We get a term of this form as a consequence of the mixing property proved in \cref{s:spatialmixing}.

\begin{proof}By Cauchy-Schwartz, and the fact that $f,g,u,v$ are deterministic,
\begin{align*}
    |\Cov(f\circ \mathcal L_{0,\epsilon} \circ g, u \circ \mathcal L_{0,\epsilon}\circ v)|&= \big|\Cov\big(f\circ \mathcal L_{0,\epsilon} \circ g - \max(f+g), u \circ \mathcal L_{0,\epsilon}\circ v - \max(u+v)\big)\big| \\&\leq \mathbb{E}[(f\circ \mathcal L_{0,\epsilon} \circ g - \max(f+g))^2 ]^{1/2} \mathbb{E}[(u\circ \mathcal L_{0,\epsilon} \circ v - \max(u+v))^2 ]^{1/2} .
\end{align*}
    Using \cref{varbound}, the last expression is bounded above by $CB^{2 \delta}\epsilon^{\frac23-\delta} +CBe^{-d\epsilon^{-1/2}}$ whenever (1) $B\leq d\epsilon^{-1/3}$, (2) $\mathfrak{Hol}\big(g,\;\frac12(1-\frac3{40} \delta),\;\; 100B\big)\le\epsilon^{-\delta/40}$, and (3) $\mathfrak{Hol}\big(v,\;\frac12(1-\frac3{40} \delta),\;\; 100B\big)\le\epsilon^{-\delta/40}.$
    
    We define the event $J_\epsilon^{f,g}:= \{ X_\epsilon^{f,g},Y_\epsilon^{f,g} \in \mathscr S(f+g,\epsilon^{\frac13- \delta})\} \cap\{|X_\epsilon^{f,g}-Y_\epsilon^{f,g}| \leq \epsilon^{\frac23-\delta}\}\cap \{|X_\epsilon^{f,g}|+|Y_\epsilon^{f,g}| \le 100B\}$, where $X_\epsilon^{f,g},Y_\epsilon^{f,g}$ are the argmaxes as defined in \cref{opt} and \cref{opt2} respectively. 
    We define the event $J^{u,v}_\epsilon$ in the same way, but with $f,g$ replaced by $u,v$. Finally, we define $J_\epsilon^{f,g,u,v}:= J^{f,g}_\epsilon\cap J^{u,v}_\epsilon$. 
    
    We condition on the event $J_\epsilon^{f,g,u,v}$, and this allows us to use our mixing property \cref{mixing} to obtain a bound of the form $\big(N_{\epsilon^{\frac13-\delta}}(f+g)+N_{\epsilon^{\frac13-\delta}}(u+v)\big)e^{-c \epsilon^{-2}\Gamma\big(f+g,u+v;\epsilon^{\frac13-\delta}\big)^3}.$ To elaborate on exactly how this works, we note that by setting $T:=\epsilon$ in \cref{mixing}, the value of $D$ in that theorem is exactly $\Gamma(f+g,u+v,\epsilon^{\frac13-\delta}).$ Furthermore, $N_{\epsilon^{\frac13-\delta}}(f+g)+N_{\epsilon^{\frac13-\delta}}(u+v)$ is an upper bound for the value of $N$ which appears in \cref{mixing}. 
    Consequently, we can apply \cref{mixing} to obtain the bound \begin{multline*}\bigg|\mathrm{Cov}\bigg(\big(f\circ \mathcal L_{0,\epsilon} \circ g\big)\cdot \mathbf 1_{J_\epsilon^{f,g,u,v}}  , \big(u \circ \mathcal L_{0,\epsilon}\circ v\big)\cdot \mathbf 1_{J_\epsilon^{f,g,u,v}}\bigg)\bigg|  \leq \big(N_{\epsilon^{\frac13-\delta}}(f+g)+N_{\epsilon^{\frac13-\delta}}(u+v)\big)e^{-c \epsilon^{-2}\Gamma\big(f+g,u+v;\epsilon^{\frac13-\delta}\big)^3},
    \end{multline*}
    because on $J_\epsilon^{f,g,u,v}$ the maximum on all of $\mathbb R^2$ in the definitions of $f\circ \mathcal L_{0,\epsilon} \circ g,u\circ \mathcal L_{0,\epsilon} \circ v$ can be replaced by a maximum on some set of the form which appears in \cref{mixing}. 
    
    Now we prove the bound in the proposition. By writing $1 = \mathbf 1_{J_\epsilon^{f,g,u,v}} + \mathbf 1_{(J_\epsilon^{f,g,u,v})^c}$ and applying Cauchy-Schwarz twice we have 
    \begin{align*}
        \bigg|\mathrm{Cov}&\big(f\circ \mathcal L_{0,\epsilon} \circ g , u \circ \mathcal L_{0,\epsilon}\circ v\big)\bigg| \leq \bigg|\mathrm{Cov}\bigg(\big(f\circ \mathcal L_{0,\epsilon} \circ g\big)\cdot \mathbf 1_{J_\epsilon^{f,g,u,v}}  , \big(u \circ \mathcal L_{0,\epsilon}\circ v\big)\cdot \mathbf 1_{J_\epsilon^{f,g,u,v}}\bigg)\bigg|  \\&\;\;\;\; + 2\mathbb{E}[(f\circ \mathcal L_{0,\epsilon} \circ g - \max(f+g))^4 ]^{1/4} \mathbb{E}[(u\circ \mathcal L_{0,\epsilon} \circ v - \max(u+v))^4 ]^{1/4} \cdot (1-\mathbb P(J_\epsilon^{f,g,u,v}))^{\frac{1}{2}}.
    \end{align*} Using arguments similar to \eqref{e:je}, it follows that $(1-\mathbb P(J_\epsilon^{f,g,u,v}))^{\frac{1}{2}}$ yields an upper bound consisting of terms of the form appearing in $\mathscr E_\epsilon$. The upper bounds for all other terms have already been discussed in \cref{varbound}.
\end{proof}

The next theorem will provide a bound on terms that appear while proving \eqref{e:bncondfinal} when $n=2.$

\begin{thm}\label{tbound2}
    Fix $\delta\in (0,1/3]$. Consider functions $f,g,u,v,\bar u, \bar v$, each of which lie between $-B-2x^2$ and $B-\frac12x^2$. Define $\Gamma(\epsilon):=\min\{\Gamma\big(f+g,u+v;\epsilon^{\frac13-\delta_1}\big),\Gamma\big(f+g,\bar u+\bar v;\epsilon^{\frac13-\delta_1}\big)\}.$ Then there exists $C=C(\delta)>0$ such that 
    \begin{align*}\bigg|\mathrm{Cov}\bigg(f\circ &\mathcal L_{0,\epsilon} \circ g\;\;,\;\; (u \circ \mathcal L_{0,\epsilon}\circ v-\max(u+v)) ((\bar u \circ \mathcal L_{0,\epsilon}\circ \bar v-\max(\bar u+\bar v))\bigg)\bigg| \\ & \leq  CB^{2\delta} \min\big\{ \epsilon ^{1 - \delta }, \big(N_{\epsilon^{\frac13-\delta}}(f+g)+N_{\epsilon^{\frac13-\delta}}(u+v)+N_{\epsilon^{\frac13-\delta}}(\bar u+\bar v)\big)e^{-c \epsilon^{-2}\Gamma(\epsilon)^3}\big\} \\& \;\;\;\;\;\;\;\;\;\;+CB^3 \cdot( \mathscr E_\epsilon(f;B)+\mathscr E_\epsilon(g;B)+\mathscr E_\epsilon(u;B)+\mathscr E_\epsilon(v;B) + \mathscr E_\epsilon(\bar u;B)+\mathscr E_\epsilon(\bar v;B)). 
    \end{align*} 
    Here $C,d$ are universal constants uniform over all $f,g,u,v$, $\epsilon,B$, and $\mathscr E_\epsilon$ is the error term as defined in \cref{tbound1}. 
\end{thm}

\begin{proof}
The proof is very similar to that of the previous theorem. Applying Cauchy-Schwartz twice, we obtain
\begin{multline*}
    \bigg|\mathrm{Cov}\bigg(f\circ \mathcal L_{0,\epsilon} \circ g, (u \circ \mathcal L_{0,\epsilon}\circ v-\max(u+v)) ((\bar u \circ \mathcal L_{0,\epsilon}\circ \bar v-\max(\bar u+\bar v))\bigg)\bigg| \\ \le \mathbb{E}[(f\circ \mathcal L_{0,\epsilon} \circ g - \max(f+g))^2 ]^{1/2} \mathbb{E}[(u\circ \mathcal L_{0,\epsilon} \circ v - \max(u+v))^4 ]^{1/4} \mathbb{E}[(\bar u\circ \mathcal L_{0,\epsilon} \circ \bar v - \max(\bar u +\bar v))^4 ]^{1/4} .
\end{multline*}
    Using \cref{varbound}, the last expression has an upper bound of $CB^{2\delta}\epsilon^{1-\delta} +CBe^{-d\epsilon^{-1/2}}$ whenever (1) $B\leq d\epsilon^{-1/3}$, (2) $\mathfrak{Hol}\big(g,\;\frac12(1- \delta_1),\;\; 100B\big)\le\epsilon^{-\delta_1/3}$, and (3) $\mathfrak{Hol}\big(v,\;\frac12(1- \delta_1),\;\; 100B\big)\le\epsilon^{-\delta_1/3},$ where $\delta_1=\frac3{40}\delta.$

    We proceed exactly as in the previous theorem. We condition on the event $J_\epsilon^{f,g,u,v, \bar u,\bar v}$, which allows us to use \cref{mixing} to obtain a bound of the form $$\big(N_{\epsilon^{\frac13-\delta_1}}(f+g)+N_{\epsilon^{\frac13-\delta_1}}(u+v)+N_{\epsilon^{\frac13-\delta_1}}(\bar u+\bar v)\big)e^{-c \epsilon^{-2}\Gamma(\epsilon)^3}.$$ To apply \cref{mixing}, we again set $T:=\epsilon$. In this case, the value of $D$ corresponds to $\Gamma(\epsilon), $ and $N_{\epsilon^{\frac13-\delta_1}}(f+g)+N_{\epsilon^{\frac13-\delta_1}}(u+v)+N_{\epsilon^{\frac13-\delta_1}}(\bar u+\bar v)$ is an upper bound for the value of $N$ from \cref{mixing}. All that remains is to control what occurs on the complement of $J_\epsilon^{f,g,u,v, \bar u,\bar v}$. We apply the same approach to this bound as we did in \ref{tbound1}. As we observed there, this will only yield error terms of the form $\mathscr E_\epsilon$.
\end{proof}

The next theorem will provide a bound on terms that appear while proving \eqref{e:bncondfinal} when $n\ge 3.$

\begin{thm}\label{tbound3}
Consider functions $f_i,g_i,u_j,v_j$ ($1\le i\le k,1\le j\le \ell)$ each lying in between $-B-2x^2$ and $B-\frac12x^2$. 
Then for any $\delta_1\in (0,1/3]$ there exists $\delta_2>0$ and $C=C(\delta_1)>0$ so that \begin{multline*}\bigg|\mathrm{Cov}\bigg(\prod_{i=1}^k (f_i\circ \mathcal L_{0,\epsilon} \circ g_i-\max(f_i+g_i)), \prod_{j=1}^\ell (u_j \circ \mathcal L_{0,\epsilon}\circ v_j-\max(u_j+v_j))\bigg)\bigg| \\ \leq  CB^{\delta_2} \epsilon^{\frac13(k+\ell)-\delta_1} +CB^{k+\ell} \cdot( \sum_{i=1}^{k} \mathscr E_\epsilon(f_i;B)+ \mathscr E_\epsilon(g_i;B)+\sum_{j=1}^{k} \mathscr E_\epsilon(u_j;B) +  \mathscr E_\epsilon(v_j;B)). 
    \end{multline*}
    Here $C,d$ are universal constants uniform over all $f_i,g_i,u_i,v_i$, $\epsilon,B$, and $\mathscr E_\epsilon$ is the ``error term" as defined in \cref{tbound1}.
\end{thm}

\begin{proof}
    We apply Cauchy-Schwartz to the covariance, and then apply H\"older's inequality to obtain
    \begin{align*}
    \bigg|\mathrm{Cov}&\bigg(\prod_{i=1}^k (f_i\circ \mathcal L_{0,\epsilon} \circ g_i-\max(f_i+g_i)), \prod_{j=1}^\ell (u_j \circ \mathcal L_{0,\epsilon}\circ v_j-\max(u_j+v_j))\bigg)\bigg| \\&\leq \mathbb E\bigg[\prod_{i=1}^k (f_i\circ \mathcal L_{0,\epsilon} \circ g_i-\max(f_i+g_i))^2 \bigg]^{1/2} \mathbb E\bigg[ \prod_{j=1}^\ell (u_j \circ \mathcal L_{0,\epsilon}\circ v_j-\max(u_j+v_j))^2\bigg]^{1/2} \\&\leq \prod_{i=1}^k\mathbb E\bigg[ (f_i\circ \mathcal L_{0,\epsilon} \circ g_i-\max(f_i+g_i))^{2k} \bigg]^{1/(2k)} \prod_{j=1}^\ell\mathbb E\bigg[ (u_j \circ \mathcal L_{0,\epsilon}\circ v_j-\max(u_j+v_j))^{2k}\bigg]^{1/(2k)}
    .
    \end{align*}
    Then use \cref{varbound} to obtain the claim (note here that unlike the previous two theorems, we do not need \cref{mixing} as we are simply using the unrefined bound without leveraging the mixing property, as this will be sufficient later).
\end{proof}

\begin{rem}\label{rem1}
    In all propositions above, the constants $1/2$ and $2$ in the parabolic decay rates of $f,g$ could have been replaced by $A^{-1}$ and $A$ for arbitrary $A>1$. This would affect the constants such as $100$ and $108$ as well as the constants $C,d$. While we cannot control the behavior of these constants as $A\uparrow \infty$, there is never an issue for a fixed finite $A$, no matter how large. This will be important later due to the nature of \cref{c:airycondition}, because we will want $f,g$ to be ``typical" sample paths of parabolic Airy$_2$ processes of scales within the interval $[\eta,1]$ for some fixed but very small $\eta>0$. In this case $A$ can be as large as $C\eta^{-1}$ for some universal $C>0.$

\end{rem}

\section{Estimates on three-dimensional Bessel processes}\label{s:bessel}

Our ultimate goal is to prove the bound in \eqref{e:bncondfinal}. As discussed at the end of \cref{s:noisetoairy}, we need to study quantities of the form $\mathrm{Cov}(f\circ \mathcal{L}_{0,\epsilon}\circ g,u\circ \mathcal L_{0,\epsilon}\circ v)$, where $f,g,u,v$ are independent parabolic Airy$_2$ processes, independent of $\mathcal L_{0,\epsilon}$. These Airy$_2$ processes are locally absolutely continuous with respect to Brownian motion, a fact which is due to several works (originally \cite{CH}, and with the strongest form proved in \cite{duncan2}). A Brownian motion re-centered around a local maximum looks like a three-dimensional Bessel process. The proof of \cref{t:1} will therefore rely on some estimates on Bessel processes, which we prove in this section. We refer to the discussion at the beginning of \cref{s:prooft1} for a more precise explanation as to \textit{why} exactly these estimates are the relevant ones to our argument.

\begin{lemm}\label{l:besselbound} Let $R$ be a three-dimensional Bessel process, and let $p\in \mathbb{N}$. Then there exists $C_p>0$ such that we have a uniform bound over all $s_1\leq ...\leq s_p$ and $\epsilon\in (0,1]$
\begin{align*}
\mathbb{P} \bigg( \inf_{t\in [s_i,s_i+1]} R(t) \leq \epsilon ,\; \forall 1\le i \le p\bigg) \leq C_p \epsilon^p \prod_{i=0}^{p-1} \min\big\{1,(s_{i+1}-s_i)^{-3/2}(1+\log_+^{3/2}(s_{i+1}-s_i))\big\},
\end{align*}where $s_0:=0$. 
\end{lemm}
This lemma will be the crucial input to later results in this section. It will be most relevant when the $s_i$ are large and $\epsilon$ is small. This is because the Airy$_2$ processes appearing in the discussion after \eqref{e:bncondfinal} look locally like three-dimensional Bessel process near local maxima, and this lemma can be used to yield estimates on how often and how likely a three-dimensional Bessel process is likely to get close to its minimum value of 0. 
\begin{proof} We will prove the bound first in the case that $\epsilon=1$. We interpret the Bessel process as the Euclidean norm of a three-dimensional standard Brownian motion, that is, $R(t) = |\vec B(t)|.$ It is a helpful fact that if $\vec B=(B_1,B_2,B_3)$, then through a union bound and well-known formulas for the distribution of the supremum of a Brownian motion on $[0,1]$, 
\begin{equation}\label{e:brownsup}\mathbb{P}( \sup_{t\in[0,1]}|\vec B(t)| >a) \leq 3\mathbb{P} (\sup_{t\in [0,1]} B_1(t) >a/3) \leq 6 e^{-a^2 /18}.\end{equation} 
It is also helpful to note 
\begin{equation}\label{e:brownuseful} \mathbb{P}\bigg( |\vec B(1)| >b, \;\inf_{t\in [0,1]} |\vec B(t)| \leq 1 \bigg||\vec B(0)|=a\bigg) \leq Ce^{-\frac1{36}(a^2+ b^2)},
\end{equation}
where $C>0$ does not depend on $a,b>0.$ This follows from the strong Markov property of Brownian motion. To see this, define the stopping time $\tau$ as the first time $t$ that $|\vec B(t)| = 1$ and the Brownian motion $\vec W(t):= \vec B(t+\tau)-\vec B(\tau)$. The event $\{|\vec B(1)| >b, \;\inf_{t\in [0,1]} |\vec B(t)| \leq 1\} $ is contained in the event $\{\sup_{t\in [0,1]} |\vec W(t)| > b-1\},$ and the second event is moreover independent of the $\mathcal F^B_\tau$-measurable event $\{\inf_{t\in [0,1]} |\vec B(t)|<1\}=\{\tau\le 1\}.$ Therefore, by conditioning on the latter event, we see that the left side of \eqref{e:brownuseful} is bounded above by $$\mathbb{P}(\sup_{t\in [0,1]} |\vec W(t)| > b-1)\cdot \mathbb{P}(\inf_{t\in [0,1]} |\vec B(t)|<1|\vec B(0)=a). $$ By \eqref{e:brownsup} this is bounded above by $36 e^{-\frac1{18}((a-1)^2+(b-1)^2)}. $ By applying $(b-1)^2 \ge \frac12 b^2 - 1^2$ and $(a-1)^2 \ge \frac12a^2-1^2$, we finally arrive at \eqref{e:brownuseful} with $C= 36 e^{\frac1{9}}$.

We return to proving the lemma in the case $\epsilon=1$. Without loss of generality, we will assume that $s_{i+1}-s_i>3$. In order for $\inf_{t\in [s_i,s_i+1]} R(t) \leq 1$ for all $1\le i \le p+1$, at least one of the following three things must happen for each $i =1,...,p$:
\begin{enumerate}
\item  $\sup_{t\in [s_i,s_i+1]} |B(t)-B(s_i+1)| > 9\sqrt{\log(s_{i}-s_{i-1})}$.
\item  $|B(s_{i+1})-B(s_i+1)| \leq 18\sqrt{\log(s_{i}-s_{i-1})}.$
\item $\sup_{t\in [s_{i+1},s_{i+1}+1]} |B(t)-B(s_{i+1})| > 9\sqrt{\log(s_{i}-s_{i-1})}$.
\end{enumerate}
We partition $[s_1,s_p)$ into $2p$ disjoint intervals of the form $[0,s_1),[s_1,s_1+1),[s_1+1,s_2), [s_2,s_2+1),[s_2+1,s_3),..., [s_p,s_p+1),$ which we denote by $I_1,J_1,I_2,J_2,I_3,..., J_p$. For a given realization of the Brownian motion $\vec B$, we say that an interval $I_i$ or $J_i$ is \textit{active} if at least one of the conditions above holds on that interval. We have shown that among any three consecutive sub-intervals in the partition $I_1,J_1,I_2,J_2,I_3,..., J_p$, at least one must be active. Some intervals may furthermore be \textit{doubly active} in the sense that the inequalities within the above conditions may be satisfied for two distinct values of $i$. Furthermore, at least one of the first two sub-intervals $I_1,J_1$ must be active.

Consequently, by counting the doubly active intervals twice, in order for $\inf_{t\in [s_i,s_i+1]} R(t) \leq 1$ for all $1\le i \le p+1$, there must be at least 
$p$ sub-intervals from $I_1,J_1,I_2,J_2,I_3,..., J_p$ that are active, each corresponding to a unique value of $i$ in the list above. 
Notice that $e^{-(9\sqrt{\log(s_{i+1}-s_i)})^2/18} \leq 1\wedge (s_{i+1}-s_i)^{-2} \leq 1\wedge (s_{i+1}-s_i)^{-3/2}$, and the three-dimensional Gaussian ``small-ball probability" volume bound tells us that 
\begin{equation}\label{e:3dgauss}\mathbb{P}\big(|B(s_{i+1})-B(s_i+1)| \leq 18\sqrt{\log(s_{i+1}-s_i)}\big) \leq \max\{ 1,18^{3} \cdot (s_{i+1}-s_i)^{-3/2} (1+\log^{3/2}_+(s_{i+1}-s_i))\}.\end{equation}
By applying a union bound over all possible choices of $p$-tuples of active and doubly active sub-intervals from $I_1,J_1,I_2,J_2,I_3,..., J_p$, and applying the independence of increments of the three-dimensional Brownian motion, we obtain the bound in the lemma. This is a consequence of \eqref{e:brownuseful} for doubly active intervals and \eqref{e:brownsup} or \eqref{e:3dgauss} for active ones. Each of the bounds for the increments $[s_{i+1},s_i]$ will be used exactly once. This proves the claim when $\epsilon=1$.

For arbitrary $\epsilon>0$, we need to explain why we obtain a factor $\epsilon^p$. We begin by explaining this factor in the case that $p=1$. If we condition on reaching the value $1$ inside the interval $[s_i,s_i+1]$, the probability of reaching $\epsilon$ inside that interval is bounded by the probability of reaching $\epsilon$ ever again after time $s_i$. Using the Markov property, and the optional stopping theorem for the martingale $R^{-1}$ (stopping when this martingale hits $\epsilon^{-1})$, we show that the latter probability is exactly $\epsilon.$ Therefore, for the $p=1$ case we have the bound 
\begin{align*}\mathbb{P}\big( \inf_{t\in [s_i,s_i+1]} R(t)<\epsilon\big) &= \mathbb{P}\bigg( \inf_{t\in [s_i,s_i+1]} R(t)<\epsilon\bigg|\inf_{t\in [s_i,s_i+1]} R(t)<1\bigg)\mathbb{P}\big( \inf_{t\in [s_i,s_i+1]} R(t)<1\big) \\ &\leq \epsilon \;\;\cdot \;\;s_1^{-3/2} (1+\log_+^{3/2} s_1) . \end{align*} While this illustrates the bound in the case $p=1$, the same argument extends in a straightforward way to 
general values of $p$.
\end{proof}

\begin{prop}\label{numberneeded}
    Let $R$ be a standard two-sided three-dimensional Bessel process. Let $N_\epsilon(R)$ denote the smallest number of intervals of length $\epsilon^{2}$ needed to cover the random set $\{t\in \mathbb{R}: R(t)\le \epsilon\}.$ Then $\sup_{\epsilon\in (0,1]} \mathbb{E}[N_\epsilon(R)^p]<\infty$ for all $p\ge 1.$
\end{prop}

To bound the covariance in \cref{s:prooft1} using the mixing result of \cref{s:spatialmixing}, we need to obtain a good lower bound on the distance of two sets given by the union of intervals. The notation $\Gamma(f,g,\epsilon)$ from \eqref{gamm} is a distance between two sets, but those two sets may not be unions of intervals. To turn them into unions of intervals, we will use \cref{numberneeded} later.

\begin{proof}
The random variable $N_{\epsilon}(R)$ takes values in the positive integers and is bounded above by $$\epsilon^{-2} \int_\mathbb{R} \mathbf 1_{\{ \inf_{t\in [s,s+\epsilon^2]}R(t) \le \epsilon\}} ds .$$ By scale invariance of the Bessel process, we know that the above random variable has the same law as $$\int_\mathbb{R} \mathbf 1_{\{ \inf_{t\in [s,s+1]} R(t) \le 1\}} ds.$$ Let $p$ be an even integer, and take the $p^{th}$ moment, then apply Fubini to interchange the integral and expectation. We obtain \begin{align*}\mathbb{E}[N_\epsilon(R)^p]&\leq \int_{\mathbb{R}^p} \mathbb{P} \bigg( \inf_{t\in [s_i,s_{i+1}]} R(t) \leq 1,\; \forall 1\le i \le p\bigg) ds_1\cdots ds_p \\ &=2^p  p! \cdot \int_{\{0<s_1<...<s_p\}} \mathbb{P} \bigg( \inf_{t\in [s_i,s_i+1]} R(t) \leq 1,\; \forall 1\le i \le p\bigg) ds_1\cdots ds_p.
\end{align*}
By \cref{l:besselbound} the integrand is bounded above by 
\begin{align*}C_p\prod_{i=0}^{p-1} \min\big\{1,(s_{i+1}-s_i)^{-3/2}(1+\log_+^{3/2}(s_{i+1}-s_i))\big\},
\end{align*} where $s_0:=0$. The fact that previous integral is finite follow from this bound. 
\end{proof}


\begin{defn}
    We define a Bessel process pinned at $t\in \mathbb{R}$ to be the process $R(\bullet -t)$ where $R$ is a two-sided three-dimensional Bessel process, centered at zero. We will denote by $\mathbb{P}_{\mathrm{Bes3}}^{t_1,t_2}$ to be the law on the canonical space $C(\mathbb{R})^2,$ of two independent Bessel processes $R_1, R_2$ pinned at $t_1,t_2$ respectively. 
\end{defn}

We call the following estimate the ``key estimate" because it will be the main tool used in obtaining a precise lower bound on distance between the maximizing sets of $f$ and $g$, which appears in Theorems \ref{tbound1} $-$ \ref{tbound3} (denoted $\Gamma(f,g,\epsilon)$ there), which in turn will be the main obstacle in establishing the bound \eqref{e:bncondfinal} needed to establish \cref{t:1}.

\begin{prop}[The key estimate] \label{5.6} Fix $\delta\in (0,1/2)$ and $0<\bar \gamma<\gamma<1$. Then there exists $C=C(\delta,\gamma,\bar\gamma)>0$ so that for all $\epsilon\in (0,1]$ we have the estimate 
\begin{align*}\sup_{|t_1-t_2| \geq \epsilon^{-\gamma}} &\mathbb P_{\mathrm{Bes3}}^{t_1,t_2}\bigg( \mathrm{ there \;\;exist\;\; } s_1,s_2\in \mathbb R \mathrm{ \;\;such \;\;that \;\;} R_1(s_1) + R_2(s_2) \leq \epsilon^{- \delta} \mathrm{ \;\;and \;\;} |s_1-s_2|\leq 2\epsilon^{-\bar \gamma} \bigg) \\& \leq C\epsilon^{\frac32 \gamma -\delta(2+\gamma-\bar\gamma) -\frac12\bar \gamma}.
\end{align*}
\end{prop}

For the black noise result, we will see that this bound is most useful when $\gamma$ is slightly smaller than $1/3$ and $\delta,\bar\gamma$ are very close to 0. In this case, note that the right side can be made arbitrarily close to $\epsilon^{1/2}$. In practice, we will only use this estimate with $\gamma=4/15$ and $\bar\gamma=\delta$ close to 0. 
\begin{proof}
    The proof will proceed using a simple union bound. If such values of $s_1,s_2\in \mathbb{R}$ exist as specified above, then there exists some $n\in \mathbb{Z}$ such that all of the following simultaneously occur: $|s_1-n| \leq \epsilon^{-\bar \gamma}$, $|s_2-n| \leq \epsilon^{-\bar \gamma}$, $R_1(s_1) \leq \epsilon^{-\delta}, R_2(s_2) \leq \epsilon^{-\delta}. $ Applying this union bound over all $n\in \mathbb{Z}$, we find that the desired probability is bounded above by
        $$\sum_{n\in \mathbb Z}\mathbb P_{\mathrm{Bes3}}^{t_1,t_2} \bigg( \text{ there exist } s_1,s_2\in \mathbb R \text{ such that } |s_1-\epsilon^{-\bar \gamma}n| \leq \epsilon^{-\bar \gamma},|s_2-\epsilon^{-\bar \gamma}n| \leq \epsilon^{-\bar \gamma},R_1(s_1) \leq \epsilon^{-\delta}, R_2(s_2) \leq \epsilon^{-\delta}\bigg). $$
    Without any loss of generality we can replace $(t_1,t_2)$ by $(0,t_2-t_1)$ and assume that $T:=t_2-t_1>0,$ thus by assumption $T>\epsilon^{-\gamma}$. We can estimate the probability in the summand, using the scaling invariance of the Bessel process to rewrite 
    \begin{align*}& \mathbb P_{\mathrm{Bes3}}^{0,T} \bigg( \text{ there exist } s_1,s_2\in \mathbb{R} \text{ such that }|s_1-\epsilon^{-\bar \gamma}n| \leq \epsilon^{-\bar \gamma},|s_2-\epsilon^{-\bar \gamma}n| \leq \epsilon^{-\bar \gamma},R_1(s_1) \leq \epsilon^{-\delta}, R_2(s_2) \leq \epsilon^{-\delta}\bigg) \\ &= \mathbb P_{\mathrm{Bes3}}^{0,\epsilon^{\bar\gamma}T} \bigg( \text{ there exist } s_1,s_2\in \mathbb{R} \text{ such that }|s_1-n| \leq 1,|s_2-n| \leq 1,R_1(s_1) \leq \epsilon^{\frac12 \bar\gamma-\delta}, R_2(s_2) \leq \epsilon^{\frac12\bar \gamma-\delta}\bigg) \\ &= \prod_{i=1}^2 \mathbb P_{\mathrm{Bes3}}^{0,\epsilon^{\bar\gamma}T} \bigg( \text{ there exists } s_i\in [n-1,n+1] \text{ such that }R_i(s_i) \leq \epsilon^{\frac12 \bar\gamma-\delta}\bigg) \\ &\leq C \bigg[ \epsilon^{\frac12\bar\gamma-\delta} |n|^{-\frac32+\delta}\bigg]\bigg[\epsilon^{\frac12\bar\gamma -\delta} |\epsilon^{\bar\gamma} T-n|^{-\frac32+\delta}\bigg] \\&=C \epsilon^{\bar\gamma -2\delta} |n|^{-\frac32+\delta}|\epsilon^{\bar\gamma} T-n|^{-\frac32+\delta}.
    \end{align*}
    In the third line, we used the independence of $R_1$ and $R_2$, and \cref{l:besselbound} with $p=1$ in the fourth line (we replaced the logarithmic correction by a weaker bound of size $|n|^{\delta}$ for simplicity). 
    Now notice that uniformly over $N\in \mathbb{N}$ we have by symmetry $$\sum_{n=1}^N \frac1{n^{\frac32-\delta}(N-n)^{\frac32-\delta}} \leq 2\sum_{1\le n\le N/2} \frac1{n^{\frac32-\delta}(N-n)^{\frac32-\delta}} \leq 2\left(\frac{2}{N}\right)^{\frac32-\delta} \sum_{1\le n\le N/2}n^{-\frac32+\delta} \leq \frac{2^{5/2} \zeta(\tfrac32-\delta)}{N^{\frac32 -\delta}}.$$
    Using this bound with $N=\epsilon^{\bar\gamma}T\geq \epsilon^{\bar\gamma-\gamma},$ we find that the desired probability is bounded above by $C\epsilon^{\bar\gamma -2\delta} (\epsilon^{\bar\gamma}T)^{-\frac32+\delta} \leq C\epsilon^{\bar\gamma -2\delta} (\epsilon^{\bar\gamma-\gamma})^{-\frac32+\delta} = C\epsilon^{\frac12 (3\gamma-\bar\gamma) -\delta(2+\gamma-\bar\gamma)}.$ 
\end{proof}

\begin{prop}\label{5.1} 
    Consider a standard Brownian motion $B$ on $[0,1]$, and let $T$ denote its unique argmax. Fix $\alpha\in (0,1)$. Define the process $U:=(B(T)-B(T+t))_{t\in [-\alpha T, \alpha(1- T)]}$. Consider a (two-sided) three-dimensional Bessel process $R$ independent of $T$ and $B$, and define the process $V:= (R(t))_{t\in [-\alpha T, \alpha(1- T)]}$. Then $U$ is absolutely continuous with respect to $V$, and the Radon-Nikodym derivative is deterministically bounded by $C(1-\alpha)^{-1/2}$ for some universal constant $C>0$. 
\end{prop}

In the statement of \cref{5.1}, we can think of a continuous function on a random closed interval $I\subset [-1,1]$ as an element of $C[-1,1]$ that is constant away from $I$. The last statement is sharp: if $\alpha=1$ then the Radon-Nikodym is only in $L^{3-\delta}$ as opposed to $L^\infty$, as discussed in \cite[Theorem 2.15]{dsv}.

\begin{proof}
    Conditioned on $T$, the process $(B(T)-B(T+t))_{t\in [- T, 1- T]}$ is a two-sided Brownian meander, as discussed in \cite[Theorem 2.15]{dsv}. The left side is a meander of length $T$ and the right side is a meander of length $1-T$. When we condition on $T$, these two meanders are independent of one another. Therefore, it suffices to show that for a Brownian meander $X$ of length $1$, parametrized by $t\in [0,1]$, the Radon-Nikodym derivative $(X(t))_{t\in [0,\alpha]}$ is absolutely continuous with respect to a (one-sided) Bessel process on the same interval $(\tilde R(t))_{t\in [0,\alpha]},$ where $\alpha\in(0,1)$ is the fixed constant appearing in the proposition statement.

    In the proof of \cite[Theorem 2.15]{dsv}, the authors observe that if $\alpha=1$, then the Radon-Nikodym derivative is simply given by $\sqrt{\pi/2} \cdot \tilde R(1)^{-1}.$ Let $\mathcal F^{\tilde R}_t:=\sigma(\{\tilde R(s)|s\in[0,t]\})$ denote the filtration of the Bessel process. Then for general $\alpha \in [0,1]$ the Radon-Nikodym derivative of $X$ with respect to $\tilde R$ on the interval $[0,\alpha]$ is simply given by $$\sqrt{\pi/2} \cdot \mathbb{E}[ \tilde R(1)^{-1} | \mathcal F^{\tilde R}_\alpha] = \sqrt{\pi/2} \cdot f_\alpha(\tilde R(\alpha)),$$
    where the expectation is taken with respect to a three-dimensional Bessel Process, and $f_\alpha:[0,\infty)\to (0,\infty)$ is defined to be the unique function satisfying 
    $$f_\alpha(|\vec x|) = (2\pi(1-\alpha))^{-3/2}\int_{\mathbb{R}^3} e^{-\frac{|\vec y|^2}{2(1-\alpha)}}|\vec x-\vec y|^{-1}d\vec y.$$
    In other words, $\vec x\mapsto f_\alpha(|\vec x|)$ is the radial function on $\mathbb{R}^3$ given by convolving the three-dimensional heat kernel with the harmonic function $\vec x\mapsto |\vec x|^{-1}$. The equality above follows from the Markov property and the fact that the Bessel process $R$ can be realized as the Euclidean norm of a three-dimensional Brownian motion. A calculation shows that $\|f_\alpha\|_{L^\infty(\mathbb{R}^3)}\leq C(1-\alpha)^{-1/2}$ for a universal constant $C.$ 
\end{proof}

\section{Proof of the main result: Theorem \ref{t:1}}

In this section, we build towards the proof of Theorem \ref{t:1}. 
Recall from \cref{s:noisetoairy} that, to prove Theorem \ref{t:1}, it suffices to prove the estimate in \eqref{e:bncondfinal}. For $n=1$ we can rewrite the left side of \eqref{e:bncondfinal} in terms of an integral/expectation over four independent Airy$_2$ processes $f,g,u,v$, where the integrand is given by $\mathrm{Cov} ( f\circ \mathcal L_{0,\epsilon} \circ g, u \circ \mathcal L_{0,\epsilon} \circ v),$ and where the covariance is taken only over $\mathcal L_{0,\epsilon}$, which is independent of $f,g,u,v$. Recall that \cref{tbound1} provides very good bounds on that covariance, uniformly over $f,g,u,v$ which lie in a deterministic class of functions. Likewise, the results in \cref{tbound2} and \cref{tbound3} provide the relevant bounds for larger values of $n$.

However, we cannot immediately apply Theorems \ref{tbound1} $-$ \ref{tbound3}. In order to do so, we need to control various quantities appearing in that theorem, such as $N_{\epsilon^{\frac13 -\delta}} (f+g)$, $B$, $\Gamma(f+g,u+v,\epsilon^{\frac13-\delta}),$ and $\mathscr E_\epsilon (f,B)$. But if $f,g,u,v$ are Airy$_2$ processes, then these quantities are random variables. In order to bound these quantities, we note that Airy$_2$ sample paths are locally Brownian, and strongly resemble three-dimensional Bessel process near their local maxima. As a result, we can use results from \cref{s:bessel} to control $N_{\epsilon^{\frac13 -\delta}} (f+g)$ and $\Gamma(f+g,u+v,\epsilon^{\frac13-\delta}) $. \cref{5.2} $-$ \cref{errcontrol} are devoted to obtaining very precise bounds on these random quantities when $f,g,u,v$ are assumed to be Airy$_2$ processes. Finally, \cref{t:final} will use those bounds to prove the main result, in turn by proving that \eqref{e:bncondfinal} holds true for any $\varrho<1/15$.

\label{s:prooft1}

\begin{defn}
An Airy process of scale $s>0$ centered at $y\in \mathbb R$ is the process $$\mathcal A_y^s(x):= s^{1/3} \mathcal P_1(s^{-2/3}(x-y)),\;\;\;\;\;\; x\in\mathbb R$$ where $\mathcal P_1$ is the  parabolic Airy$_2$ process; the top curve of the line ensemble in \cref{d:ale}.
\end{defn}

\begin{prop}[A technical lemma on local absolute continuity]\label{5.2}
     Fix $\eta \in (0,1)$ and let $\mathcal A_x^t$ and $\tilde{\mathcal A_y^s}$ denote two independent Airy processes of scales $t,s \in (\eta,1]$ centered at $x,y\in\mathbb R$ respectively. We also define the notation $\mathcal V:=\mathcal A_x^t+\tilde{\mathcal A_y^s}$. Let $I=[a,b]$ be a compact interval. Then $(\mathcal V(t+a) - \mathcal V(a))_{t\in [0,b-a]}$ is absolutely continuous with respect to a Brownian motion of diffusion rate $4$. Furthermore, for all $p\ge 1$, the Radon-Nikodym derivative has a finite $p^{\text{th}}$ moment which is bounded above by $Ce^{Cp^2 (a^2+b^2)}$ for some universal constant $C=C(p,\eta,x,y)>0$ not depending on $p,a,b,s,t$ such that $0\leq b-a\leq 10$. 
\end{prop}

The choice $b-a\leq 10$ is arbitrary, but we will never use this bound for intervals of size larger than $10$.

\begin{proof}
    We reduce the claim to the case where $a=-b.$ Since the Airy$_2$ process is stationary without the addition of the parabola, and since the difference of two quadratic functions is linear, translating the interval $I$ to be centered at the origin amounts to adding a linear drift term of the form $mt+b$ to the process $\mathcal V(t)$, for some $m$ whose absolute value can be bounded above by $C(|a|+|b|)$. 

    The Airy process on $[-a,a]$ with $|a|\leq 10$ is absolutely continuous with respect to a standard Brownian motion on that interval, and moreover the Radon-Nikodym derivative is deterministically bounded due to \cite[Corollary 1.3]{duncan2}. 
    
    Brownian motion with a linear drift of slope $m$ on $[0,b-a]$ is absolutely continuous with respect to Brownian motion with no drift on that same interval, with Radon-Nikodym derivative given by $e^{\pm mB_{b-a} - \frac{m^2}2(b-a)^2}.$ The $p^{\text{th}}$ moment of this random variable is simply $e^{(p^2-p)m^2(b-a)^2}.$ Combining this with the result of the previous paragraph finishes the proof.
\end{proof}
\begin{prop}[Bounding the $\Gamma(f,g,\epsilon)$ terms] \label{crucial}
Fix $\eta\in (0,1)$ and recall $\Gamma$ from \eqref{gamm}. Let $\mathcal A_x^t, \tilde{\mathcal A_y^s}, \bar{\mathcal A_u^a},$ and $\hat{\mathcal A_v^b}$ denote four independent Airy processes of scale $t,s,a,b$ respectively, centered at fixed real numbers $x,y,u,v$ respectively. For any $\delta>0$, there exists a constant $C:=C(\eta, x,y,u,v,\delta)>0$, such that, uniformly over $\epsilon \in (0,1]$, $$\sup_{\eta\leq s,t,a,b\leq 1} \mathbb{P}\big( \Gamma(\mathcal A_x^t+\tilde{\mathcal A_y^s}, \bar{\mathcal A_u^a}+\hat{\mathcal A_v^b}, \epsilon^{\frac13-\delta}) \leq \epsilon^{\frac23-\delta}\big) \leq C\epsilon^{\frac25-4\delta}.$$
\end{prop}

We are not sure whether $\epsilon^{\frac25}$ is the optimal bound, however we are unable to improve it through the methods of this paper. The crucial input to proving that the directed landscape is a black noise is the fact that $\frac25-\delta$ is larger than $1/3$ for sufficiently small $\delta$. The bound in \cref{crucial} is sufficient because it will give us a bound almost as small as $\epsilon^{2/3} \cdot \epsilon^{2/5} = \epsilon^{16/15}$ in \eqref{e:covbound}. This will eventually imply \eqref{e:bncondfinal} with $\varrho$ smaller than $1/15$. 
\begin{proof}
In this proof, we leverage the cube-exponential tail decay of the argmaxes of the Airy processes (\cref{dov1.7}), and combine this with the results of \cref{5.1} and \cref{5.2} to write everything in terms of Bessel processes, so that we can ultimately apply the bound from \cref{5.6}. We make the following definitions:
\begin{itemize}
    \item From now on, denote $\mathcal V_1:= \mathcal A_x^t+\tilde{\mathcal A_y^s}$ and $\mathcal V_2:= \bar{\mathcal A_u^a}+\hat{\mathcal A_v^b}$.
    
    \item Let $G_1$ denote the argmax on $\mathbb{R}$ of $\mathcal V_1$ and let $G_2$ denote the argmax on $\mathbb{R}$ of $\mathcal V_2$. 

    \item For $k\in \mathbb{Z}$, let $E^{(1)}_k$ denote the event that $G_1\in [k,k+5]$ and let $E^{(2)}_k$ denote the event that $G_2 \in [k,k+5]$. Note that $E^{(i)}_k$ are independent for distinct $i$. Length 5 is not important here, it is just a convenient choice. 

    \item For $k_1,k_2,\ell \in \mathbb{Z}$ let $F_{k_1,k_2,\ell}$ denote the event that $G_1\in [k_1,k_1+5]$ , $G_2 \in [k_2,k_2+5]$, and furthermore the maximum values of both $\mathcal V_1,\mathcal V_2$ on the interval $[\ell,\ell+5]$ are within 1 of their absolute maximum values on all of $\mathbb{R}.$

    \item For $i=1,2$ and $k \in \mathbb{Z}$ let $H^{(i)}_k$ denote the event that there exist two points of distance at least $k$ from each other such that the values of $\mathcal V_i$ at both points are within 1 of $\max_{z\in \mathbb{R}}\mathcal V_i(z).$ Note that $H^{(i)}_k$ are independent for distinct $i$, and $$F_{k_1,k_2,\ell} \subset H^{(1)}_{k_1-\ell}\cap H^{(2)}_{k_2-\ell}.$$

    \item For $k\in \mathbb{Z}$ and $i=1,2$, let $G_i^{(k)}$ denote the argmax on $[k,k+5]$ of $\mathcal V_i$ and $M_k^{(i)} := \mathcal V_i(G^{(k)}_i)$ denote the maximum value of $\mathcal V_i$ on $[k,k+5]$.

    \item For $\ell \in \mathbb{Z}$ let $K_\ell^\epsilon$ denote the event that there are points $z_1\in [\frac45\ell+\frac15 G_1^{(\ell)},\frac45 (\ell+5) +\frac15 G_1^{(\ell)}]\subset [\ell,\ell+5]$ and $z_2\in [\frac45\ell+\frac15 G_2^{(\ell)},\frac45 (\ell+5) +\frac15 G_2^{(\ell)}]\subset [\ell,\ell+5]$ such that $|z_1-z_2| \leq \epsilon^{\frac25}$, and moreover $\mathcal V_i(z_i)> M^{(i)}_\ell -\epsilon^{\frac13-\delta}$ for $i=1,2.$
\end{itemize}

We use a union bound for the probability that $\Gamma(\mathcal A_x^t+\tilde{\mathcal A_y^s}, \bar{\mathcal A_u^a}+\hat{\mathcal A_v^b}, \epsilon^{\frac13-\delta}) \leq \epsilon^{\frac23-\delta}.$ The basis for this union bound is the following inclusion: 
\begin{align}\label{e:incl}\begin{split}\big\{\Gamma(\mathcal A_x^t+\tilde{\mathcal A_y^s}, \bar{\mathcal A_u^a}+\hat{\mathcal A_v^b}, \epsilon^{\frac13-\delta}) \leq \epsilon^{\frac23-\delta}\big\} &  \subset \bigcup_{k_1,k_2,\ell\in \mathbb{Z}} E_{k_1}^{(1)} \cap E_{k_2}^{(2)} \cap F_{k_1,k_2,\ell} \cap K^\epsilon_\ell \\ &  \subset \bigcup_{k_1,k_2,\ell\in \mathbb{Z}} E_{k_1}^{(1)} \cap E_{k_2}^{(2)} \cap H^{(1)}_{k_1-\ell}\cap H^{(2)}_{k_2-\ell} \cap K^\epsilon_\ell.
\end{split}
\end{align}
In this equation, we note that $E_{k_{1}}^{(1)}\cap E_{k_{2}}^{(2)}\cap F_{k_{1},k_{2},\ell}=F_{k_{1},k_{2},\ell}$. The purpose of writing the inclusion in this form will be clear in Step 4. The reason for the inclusion in \eqref{e:incl} is that intervals of the form $I_\ell:= \bigcap_{i=1,2}[\frac45\ell+\frac15 G_i^{(\ell)},\frac45 (\ell+4) +\frac15 G_i^{(\ell)}]$ cover all of $\mathbb{R}$ as $\ell$ varies through all of $\mathbb{Z}$, and successive intervals $I_\ell,I_{\ell+1}$ overlap in an interval of length at least $1$ (this is because each of the two intervals comprising $I_\ell$ has length $4$, and therefore $I_\ell$ itself has length at least $3$). Consequently, if $z_1,z_2\in \mathbb{R}$ such that $|z_1-z_2|<\epsilon^{\frac25}\leq 1$ then $z_1,z_2\in I_\ell$ for some $\ell.$ The union bound yields an infinite sum over $(k_1,k_2,\ell) \in \mathbb{Z}^3$, and then H\"older's inequality yields the desired bound. We break this argument into four steps, obtaining individual bounds on the events $E_{k_1}^{(1)} , E_{k_2}^{(2)} , H^{(1)}_{k_1-\ell}, H^{(2)}_{k_2-\ell} , K^\epsilon_\ell.$
\\
\\
\textbf{Step 1.} From \cref{dov1.7}, recall the notion of a \textit{geodesic path} in the directed landscape. Using the fact that the Airy$_2$ process embeds into the directed landscape by fixing both of the time coordinates and one of the spatial coordinates, the argmaxes $G_1,G_2$ can be interpreted as points where two independent geodesics of lengths $s+t$ and $a+b$ go through at respective times $s$ and $a$. Therefore, we may use \cref{dov1.7} to obtain $\mathbb{P}( E^{(i)}_k) \leq Ce^{-d|k|^3}$ for some $C,d>0$ independent of $k\in\mathbb{Z}$ and $i=1,2.$   
\\
\\
\textbf{Step 2.} We claim that $\mathbb{P}(H^{(i)}_k) \leq Ce^{-d|k|^3}$. This is almost immediate from \cref{l:airyabsbound}, since the parabolic Airy$_2$ process embeds into the Airy sheet by setting $(x,y)=(t,0)$. Either the maximum of $\mathcal V_i$ is smaller than $1-k$ (which has a probability less than $e^{-c|k|^3}$ since the maximum of each $\mathcal V_i$ is Tracy-Widom GUE distributed with a deterministic shift which depends on $x,y$, which we ignore), or else there exist two points of distance greater than $k$ from each other such that the value of $\mathcal V_1$ at both points is bigger than $-k$. In order for the latter event to occur, the constant $\mathfrak C$ in \cref{l:airyabsbound} would have to be of order $k^2$, which has a probability less than $Ce^{-d|k|^3}.$
\\
\\
\textbf{Step 3.} To bound $\mathbb{P}(K_\ell^\epsilon)$, we first note that the event $K_\ell^\epsilon$ only depends on the behavior of the processes $\mathcal V_i$ restricted to the interval $[\ell,\ell+5]$, which is the entire point of the union decomposition above. Using \cref{5.2}, and using H\"older's inequality to get rid of the Radon-Nikodym derivative, we obtain $\mathbb{P}(K_\ell^\epsilon) \leq Ce^{Cq^2\ell^2}\cdot \mathbb{P}_{BM^{\otimes 2}} ( K_\ell^\epsilon)^{1/p}$. The latter probability is taken with respect to two independent Brownian motions on $[0,5]$. Under these Brownian motions, the argmaxes $G_1^{(\ell)}$ and $G_2^{(\ell)}$ have arcsine laws, and we can verify that $\mathbb{P}_{BM^{\otimes 2}}( |G_1^{(\ell)} - G_2^{(\ell)}|<\epsilon^{\frac25}) <C \epsilon^{\frac25}.$ Therefore, $$ \mathbb{P}_{BM^{\otimes 2}}(K_\ell^\epsilon \cap \{|G_1^{(\ell)} - G_2^{(\ell)}|< \epsilon^{\frac25}\})^{1/p} < \epsilon^{p^{-1}\frac25}.$$ It only remains to estimate $\mathbb{P}_{BM^{\otimes 2}}( K_\ell^\epsilon \cap \{|G_1^{(\ell)} - G_2^{(\ell)}|\ge \epsilon^{\frac25}\})$. To do this, we apply the result of \cref{5.1} (with $\alpha =\frac45,$ since the interval $[\frac45\ell+\frac15 G_i^{(\ell)},\frac45 (\ell+5) +\frac15 G_i^{(\ell)}]\subset [\ell,\ell+5]$ has length $4$) to rewrite the probability in terms of Bessel processes, at the extra cost of a deterministic factor. Rescaling these Bessel processes by $\epsilon^{-2/3}$ in time and by $\epsilon^{1/3}$ in space puts us in the setting of \cref{5.6}, with $\gamma = \frac23 - \frac25 = \frac4{15}.$ We apply that proposition with $\bar\gamma=\delta$, where $\delta$ there is chosen the same as $\delta$ here. 
Applying \cref{5.6} yields a bound of $\epsilon^{\frac25-3\delta}$, since $\epsilon^{2/5} = (\epsilon^{4/15})^{3/2}$ and the remaining negative exponents in that proposition can be bounded above by $3\delta$. In summary, $$\mathbb{P}(K_\ell^\epsilon)\leq Ce^{C\ell^2} \epsilon^{p^{-1}(\frac25-3\delta)},$$ 
where $C$ is uniform over $\ell,\epsilon$.
\\
\\
\textbf{Step 4.} Finally, we apply the union bound in \eqref{e:incl}, then use H\"older's inequality with some other conjugate exponents $p',q'$ (possibly different from the exponents $p,q$ from earlier) and thanks to the results of Steps 1-3 we obtain 
\begin{align*}
    \mathbb{P}\big(\Gamma(\mathcal A_x^t+\tilde{\mathcal A_y^s}, \bar{\mathcal A_u^a}+\hat{\mathcal A_v^b}, \epsilon^{\frac13-\delta}) \leq \epsilon^{\frac23-\delta}\big) &\leq \sum_{k_1,k_2,\ell\in \mathbb{Z}} \mathbb{P}(E_{k_1}^{(1)} \cap E_{k_2}^{(2)} \cap H^{(1)}_{k_1-\ell}\cap H^{(2)}_{k_2-\ell} \cap K^\epsilon_\ell) \\ &\leq \sum_{k_1,k_2,\ell\in \mathbb{Z}} \mathbb{P}(E_{k_1}^{(1)} \cap E_{k_2}^{(2)})^{\frac1{2q'}} \mathbb{P}(H^{(1)}_{k_1-\ell}\cap H^{(2)}_{k_2-\ell})^{\frac1{2q'}} \mathbb{P} (K^\epsilon_\ell)^{\frac1{p'}} \\ &\leq \epsilon^{\frac1{pp'} (\frac25-3\delta)} \sum_{k_1,k_2,\ell\in \mathbb{Z}} e^{-\frac{c}{2q'}\left(|k_1|^3 + |k_2|^3+|\ell-k_1|^3 + |\ell-k_2|^3\right)} \cdot e^{C\frac{q^2}{p'}\ell^2}.
\end{align*}
For the term $e^{C\frac{q^2}{p'}\ell^2},$ we use the bound $\ell^2 = \frac12 \ell^2 +\frac12\ell^2 \leq k_1^2 +(\ell-k_1)^2 + k_2^2 + (\ell-k_2)^2.$ We perform the sums in the expression above, first over $\ell$ and then over $(k_1,k_2)$. The cube-exponential decay overwhelms the square-exponential growth terms in each case, which shows that the infinite sum converges to a finite value for any (fixed) choice of $p,p'$. We conclude the proof by taking $p,p'$ very close to $1$ (and thus $q,q'$ large) so that $\frac1{pp'} (\frac25-3\delta) = \frac25-4\delta.$
\end{proof}

\begin{lemm}[Bounding the $N_\epsilon(f)$ terms]\label{numberneeded2}
   Fix $\eta\in (0,1)$. Let $\mathcal A_x^t$ and $\tilde{\mathcal A_y^s}$ denote two independent Airy processes of scale $t,s$ centered at $x,y\in\mathbb R$ respectively, and denote $\mathcal V:=\mathcal A_x^t+\tilde{\mathcal A_y^s}$. As in \eqref{sfe}, let $N_\epsilon(\mathcal V)$ denote the smallest number of intervals of length $\epsilon^{2}$ needed to cover the random set $\mathscr S(\mathcal V,\epsilon).$ Then uniformly over all scales $s,t\in [\eta,1]$, $ \sup_{\epsilon\in (0,1]} \mathbb{E}[N_\epsilon(\mathcal V)^p]<\infty$ for all $p\ge 1.$
\end{lemm}

\begin{proof}

Let $N^k_\epsilon(\mathcal V)$ denote the smallest number of intervals contained in $[k,k+2]$ that are needed to cover the random set $\mathscr S(\mathcal V,\epsilon)\cap [k,k+2].$ We have a trivial bound $N_\epsilon(\mathcal V)\leq \sum_{k\in\mathbb{Z}} N^k_\epsilon(\mathcal V)$ and Minkowski's inequality implies 
\begin{equation}\label{e:intervalbound}\sup_{\epsilon\in (0,1]} \mathbb{E}[N_\epsilon(\mathcal V)^p]^{1/p} \leq \sum_{k\in\mathbb{Z}} \sup_{\epsilon\in (0,1]} \mathbb{E}[N^k_\epsilon(\mathcal V)^p]^{1/p}.
\end{equation}
Let $E_k$ denote the event that $N^k_\epsilon(\mathcal V)$ is nonzero, that is, the event that $\mathscr S(\mathcal V,\epsilon)\cap [k,k+2]$ is nonempty. Then either the maximum of $\mathcal V$ on all $\mathbb{R}$ is smaller than $1-k$ (which has a probability less than $e^{-d|k|^3}$ since we know that the maximum of $\mathcal V$ is Tracy-Widom GUE distributed plus a deterministic shift depending only on $x,y$ which we are ignoring) or else there exists a point in $[k,k+2]$ such that the value of $\mathcal V$ at this point is bigger than $-k$ (since $\epsilon\le 1$ by assumption). However, the constant $\mathfrak C$ in \cref{l:airyabsbound} would have to be of order $k^2$ for the latter event to occur, which has a probability lower than $Ce^{-d|k|^3}.$ Thus $\mathbb{P}(E_k) \leq Ce^{-d|k|^3}$ for some $C,d>0$.

We can bound $N_\epsilon^k(\mathcal V)$ by the number $M_\epsilon^k(\mathcal V)$ of intervals of size $\epsilon^2$ contained in $[k-1,k+3]$ needed to cover the set of points $z\in [k,k+2]$ such that $\mathcal V(z)$ is within $\epsilon$ of its maximum value on $[k-1,k+3]$. We apply \cref{5.2} and \cref{5.1} with $\alpha=3/4,$ and see that up to a deterministic factor of size $Ce^{Ck^2}$ and possibly enlarging $p$ due to H\"older's inequality, the process $\mathcal V$ can be replaced by a standard two-sided Bessel process. By this observation and \cref{numberneeded} we conclude that $\mathbb{E}[ M_\epsilon^k(\mathcal V)^p]\leq Ce^{Ck^2}$ for some absolute constants $C,d$ independent of $k,\epsilon$ but possibly depending on $p$. We find that $$\mathbb{E}[N^k_\epsilon(\mathcal V)^p] =\mathbb{E}[N^k_\epsilon(\mathcal V)^p\mathbf 1_{E_k}] \leq \mathbb{E}[M^k_\epsilon(\mathcal V)^{2p}]^{1/2}\mathbb{P}(E_k)^{1/2} \leq Ce^{-d|k|^3} \cdot Ce^{Ck^2}, $$ where $C,d$ are uniform over all $k,\epsilon.$ We sum over $k\in \mathbb{Z}$ and use \eqref{e:intervalbound} to obtain the desired bound.
\end{proof}

\begin{lemm}[Bounding the H\"older norms on sufficiently large intervals]\label{6.3}
 Let $\mathcal S$ be the Airy sheet, and for a function $f:\mathbb{R}^2\to \mathbb{R}$ let $[f]_{\alpha,M} :=\sup_{\vec s\neq\vec t \in [-M,M]^2} |\vec t-\vec s|^{-\alpha} |f(\vec t)-f(\vec s)|$ denote the $\alpha$-H\"older semi-norm of $f$ on $[-M,M]^2$. For any $\alpha<1/2$ and $\gamma\in (0,1/2)$, there exists $C,d>0$ such that uniformly over $M>0$ and $u>0$, $$ \mathbb{P}( [\mathcal S]_{\alpha,M} >u) \leq CM^{10}e^{-dM^{-\gamma}u^{3/2}}. $$
\end{lemm}

\begin{proof}
    This follows from \cite[Proposition 10.5]{dov}, by taking $\tau = 0$ and observing that $\log^{1/2}(b\xi^{-1}) \leq C(\delta)b^{\gamma}\xi^{-\gamma}$ for arbitrary $\gamma,\xi>0$.
\end{proof}

\begin{prop}[Controlling the error bounds] \label{errcontrol} We fix $\eta,\delta >0$ and recall the error terms $\mathscr E_\epsilon(f,B)$ from \cref{tbound1}. We consider two independent Airy processes $\mathcal A^s_x,\mathcal A^t_y$ of scales $s,t \in [\eta,1]$ and centered at $x,y\in\mathbb R$ respectively, and let $\mathcal V:= \mathcal A^s_x+\mathcal A^t_y$. We define the random variable $$\mathfrak B:= \inf\{ B>0 : -B-2\eta^{-1} u^2 \leq \mathcal V(u) \leq B-2\eta u^2, \text{ for all }u\in\mathbb{R}\}. $$ Fix $p\ge 1$. Then $\mathbb E[\mathfrak B^p]<\infty$, and there exist $C,d>0$ depending only on $\eta,\delta, p$ (but not on $s,t,\epsilon$) such that 
    $$\mathbb{E}[ \mathscr E_\epsilon(\mathcal V,\mathfrak B)^p] \leq Ce^{-d\log^{3/2}(1/\epsilon)}.$$
\end{prop}

\begin{proof}
     We first recall that $\mathscr E_\epsilon$ implicitly depends on $\delta$ even though we have suppressed this notation. By \cref{l:airyabsbound}, the random variable $\mathfrak B$ is almost surely finite, and in fact satisfies $\mathbb{P}(\mathfrak B>u) \leq Ce^{-du^{3/2}}$ for some $C,d>0$ that may depend on $\eta$ and $x_i,y_i$ (but not on $s,t$). We see that $\mathbb{E}[ \mathbf 1_{\{\mathfrak B>d\epsilon^{-1/3}\}}] \leq Ce^{-d \epsilon^{-1/2}}$ where the constants $d,C$ may have decreased or increased (respectively).

    It remains to bound the H\"older semi-norm term in the expression for $\mathscr E_\epsilon$. We set $\tilde\delta:=\frac3{40}\delta$, and note that the probability of $100 \mathfrak B> \log(1/\epsilon)$ decays like $Ce^{-d\log^{3/2}(1/\epsilon)}$. Consequently, we just need to show that 
    \begin{align*}
        \mathbb{P}\bigg( \mathfrak{Hol}\big(\mathcal V, \frac12(1-\tilde\delta), \log(1/\epsilon) \big)>\epsilon^{-\tilde\delta/3}\bigg) \leq Ce^{-d\log^{3/2}(1/\epsilon)}.
    \end{align*}
    We use \cref{6.3} to see that the left side is bounded above by $C\log^{40/5}(1/\epsilon) \cdot e^{-d \log^{-\delta}(1/\epsilon)\cdot \epsilon^{-\tilde\delta/2}}.$ This is far smaller than the required bound, since the logarithmic factors can be subsumed into the polynomial term $\epsilon^{-\tilde\delta/2}$. 
\end{proof}

Now we are finally in a position to prove \eqref{e:bncondfinal}.

\begin{thm}\label{t:final}
    The estimate \eqref{e:bncondfinal} holds for any $\varrho<\frac1{15}$, consequently the directed landscape is a black noise.
\end{thm}

\begin{proof} We recall the notation $f\circ L\circ g:= \sup_{x,y\in\mathbb{R}} f(x)+L(x,y)+g(y).$ Using the metric composition property of the directed landscape in \cref{d:dl} (2), and abbreviating $f_j^a:= \mathcal L_{0,a}(x_j,\cdot)$ and $g_j^b :=\mathcal L_{b,1}(\cdot,y_j)$, we can write the conditional expectation inside the variance of \eqref{e:bncondfinal} for $0\le a<b\le 1$ as 
\begin{align*}\mathbb{E}\left[ \prod_{j=1}^n \mathcal L_{0,1}(x_j,y_j) \bigg| \mathcal F^\mathrm{DL}_{a,b}\right] & = \mathbb{E}\left[ \prod_{j=1}^n f_j^a \circ \mathcal L_{a,b} \circ g_j^b\bigg|\mathcal F_{a,b}^\mathrm{DL} \right] 
\\ & = \int_{C(\mathbb R)^n} \int_{C(\mathbb R)^n} \prod_{j=1}^n (u_j \circ \mathcal L_{a,b} \circ v_j) \nu_a(d\vec u)\mu_b(d\vec v),
\end{align*} 
where $\nu_a$ and $\mu_b$ are respectively the joint laws of $(f_1^a,...,f_n^a)$ and $(g^b_1,...,g^b_n).$ Fix $\eta,\delta\in(0,1/2)$. From the last expression, and the temporal stationarity of the directed landscape, we deduce that to prove the theorem, it suffices to show that \begin{equation}
    \label{e:covbound} \int_{(C(\mathbb R)^n)^4} \mathrm{Cov}\left( \prod_{i=1}^n (f_i \circ \mathcal L_{0,b-a} \circ g_i), \prod_{j=1}^n (u_j \circ \mathcal L_{0,b-a} \circ v_j)\right)\nu_a^{\otimes 2}(d\vec f,d\vec u)\mu_b^{\otimes 2}(d\vec g,d\vec v) \le C(b-a)^{\frac{16}{15}-\delta}.
\end{equation}
where $\eta\le a<b\le 1-\eta$ and $C$ is allowed to depend upon $\eta,\delta \in (0,1/2)$ and on $x_j,y_j$ but not on $a,b$. 


In this proof, we abbreviate $[n]:=\{1,...,n\}$. By taking the Taylor expansion of the function $(x_1,...,x_n)\mapsto x_1\cdots x_n$ around some fixed value $(a_1,...,a_n),$ we rewrite it as $\sum_{S\subset [n]} a_S (x-a)_{S^c},$ where $x_S=x_{i_1}\cdots x_{i_r}$ if $S=\{i_1,...,i_r\}.$ 
We apply this fact with $x_i:=f_i\circ \mathcal L_{0,\epsilon}\circ g_i,$ and $a_i:=\max(f_i+g_i),$ which allows us to rewrite the products in \eqref{e:covbound} as linear combinations of quantities of the form $\max(f+g)_S\cdot (f\circ \mathcal L_{0,\epsilon}\circ g - \max(f+g))_{S^c},$ where, as before, $S$ is a subset of $[n]$ and the subscript denotes the product of the respective quantity over all indices which lie in $S$. This re-centering procedure is natural since $f\circ \mathcal L_{0,\epsilon}\circ g \to \max(f+g)$ as $\epsilon\to 0$. 

Consequently, to prove \eqref{e:covbound} it suffices to show that for all subsets $S,T\subset[n]$, \begin{align}
    \label{e:covbound2}\begin{split} \int_{(C(\mathbb R)^n)^4} \mathrm{Cov}\bigg( (f\circ  \mathcal L_{0,\epsilon}\circ g & - \max(f+g))_{S^c}, (u\circ \mathcal L_{0,\epsilon}\circ v - \max(u+v))_{T^c}\bigg) \\ & \cdot\max(f+g)_S\max(u+v)_T \nu_a^{\otimes 2}(d\vec f,d\vec u)\mu_b^{\otimes 2}(d\vec g,d\vec v) \le C\epsilon^{\frac{16}{15}-\delta}.
    \end{split}
\end{align}
For each \textbf{fixed} value of $i\in \{1,...,n\}$, under $\nu_a^{\otimes 2}\otimes \mu_b^{\otimes 2}$ the $f_i,g_i,u_i,v_i$ are distributed as four independent parabolic Airy$_2$ processes of different scales (these scales are always in $[\eta,1-\eta]$), perhaps recentered at some values depending on $x_i,y_i$ (which are fixed). Consequently, the individual maxima of $f_i+g_i$ and $u_i+v_i$ each have Tracy-Widom GUE laws (perhaps rescaled and recentered). Therefore, for all $q>1$ we have
 \begin{align*}
   \sup_{\eta\leq a<b\leq 1-\eta} \max_{S,T\subset[n]} \int_{(C(\mathbb R)^n)^4} \bigg|\max(f+g)_S\max(u+v)_T \bigg|^q\nu_a^{\otimes 2}(d\vec f,d\vec u)\mu_b^{\otimes 2}(d\vec g,d\vec v)<\infty.
\end{align*}
By H\"older's inequality (choosing $p$ very close to 1 and thus $q$ large), to prove \eqref{e:covbound2} it suffices to show that for all $S,T\subset[n]$ we have 
 \begin{align}\notag
   \bigg[\int_{(C(\mathbb R)^n)^4} \bigg|\mathrm{Cov}\bigg( (f\circ \mathcal L_{0,\epsilon}&\circ g - \max(f+g))_{S^c}, \\ &(u\circ \mathcal L_{0,\epsilon}\circ v - \max(u+v))_{T^c}\bigg)\bigg|^p\nu_a^{\otimes 2}(d\vec f,d\vec u)\mu_b^{\otimes 2}(d\vec g,d\vec v)\bigg]^{1/p}   \le C\epsilon^{\frac{16}{15}-\delta}. \label{e:covbound3} 
\end{align}
Let $\#S$ denote the cardinality of the set $S$. If $\# (S^c) + \#(T^c) \geq 4,$ then we may apply \cref{tbound3} to obtain a bound of size $\epsilon^{\frac43-\delta}$ plus some number of ``error terms" which by \cref{errcontrol} have a super-polynomial rate of decay as $\epsilon\to 0$. This achieves the goal in \eqref{e:covbound3}. 

By a symmetry argument, what remains is to prove \eqref{e:covbound3} in the case when $(\#(S^c),\#(T^c)) = (1,1)$ or $(\#(S^c),\#(T^c)) = (1,2)$.
We first consider the case when $(\#(S^c),\#(T^c)) = (1,1)$. In this case, the covariance in the integrand reduces to $\Cov(f_1\circ \mathcal L_{0,\epsilon}\circ g_1, u_1\circ \mathcal L_{0,\epsilon}\circ v_1)$. We apply \cref{tbound1} and \cref{tbound3} (with $k+\ell=2$). We disregard the error terms which arise in those estimates, because the error terms have a super-polynomial rate of decay as $\epsilon\to 0$ by \cref{errcontrol} (specifically, a rate of $O(e^{-d\log^{3/2}(1/\epsilon)})$). This yields an upper bound given by  \begin{align*}
     \bigg[\int_{(C(\mathbb R)^n)^4}& \mathfrak B^{2\delta} \min\bigg\{ \epsilon ^{\frac23 - \delta }, \\&\big(N_{\epsilon^{\frac13-\delta}}(f_1+g_1) + N_{\epsilon^{\frac13-\delta}}(u_1+v_1)\big)\cdot e^{-c \epsilon^{-2}\Gamma\big(f_1+g_1,u_1+v_1;\epsilon^{\frac13-\delta}\big)^3}\bigg\}\nu_a^{\otimes 2}(d\vec f,d\vec u)\mu_b^{\otimes 2}(d\vec g,d\vec v)\bigg]^{1/p} .
\end{align*}
Here $\mathfrak B$ is as in \cref{errcontrol}. We want to remove the factor of $\mathfrak{B}$ from this expression so that it is easier to analyze. Towards that end, we apply the H\"older inequality.We pick a new pair of conjugate exponents $p'$ and $q$ so that $(p')^{-1}+q^{-1}=1$ and such that $\mathbb E[\mathfrak B^{2\delta q}]<\infty$, so our expression above is bounded by a constant times the other multiplicative factor in the H\"older inequality. 
We use the elementary bound $\min\{a,b\} \leq a^\theta b^{1-\theta}$ for $\theta \in [0,1]$, and, since we are no longer using the original value of $p$, we relabel $p'\mapsto p$ to see that this expression is bounded above, for all $\theta\in [0,1]$, by
\begin{align} 
     \notag \bigg[\int_{(C(\mathbb R)^n)^4} \epsilon^{(1-\theta)(\frac23 - \delta) }&\cdot  \big(N_{\epsilon^{\frac13-\delta}}(f_1+g_1)+N_{\epsilon^{\frac13-\delta}}(u_1+v_1)\big)^{\theta} \\&\cdot e^{-c \theta \epsilon^{-2}\Gamma\big(f_1+g_1,u_1+v_1;\epsilon^{\frac13-\delta}\big)^3}\nu_a^{\otimes 2}(d\vec f,d\vec u)\mu_b^{\otimes 2}(d\vec g,d\vec v)\bigg]^{1/p} 
     ,\label{e:covbound4}
\end{align}

We split the integral into two parts: $(1)$ the case that $\Gamma(f_1+g_1,u_1+v_1,\epsilon^{\frac13-\delta})<\epsilon^{\frac23-\delta}$  and $(2)$ the case that $\Gamma(u_1+v_1,f_1+g_1,\epsilon^{\frac13-\delta})\ge \epsilon^{\frac23-\delta}.$ 
On event $(2)$, \eqref{e:covbound4} will have an exponential rate of decay of order $e^{-c\epsilon^{-3\delta}},$ which decays to $0$ at a super-polynomial rate in the limit $\epsilon\to 0.$ 

On event $(1)$, we apply \cref{numberneeded2} and H\"older's inequality (with $p$ close to 1 and $q$ large) to argue that we can disregard the term $\big(N_{\epsilon^{\frac13-\delta}}(f_1+g_1)+N_{\epsilon^{\frac13-\delta}}(u_1+v_1)\big)^{\theta}$ at the cost of making the value of $p$ slightly larger. Then we apply \cref{crucial} to see that the integral in \eqref{e:covbound4} will decay as $\epsilon^{(1-\theta)(\frac23 - \delta) }\cdot \epsilon^{\frac25-4\delta}$. In this setting, we can make the exponent on $\epsilon$ as close to $\frac{16}{15}$ as desired by taking $\theta,\delta$ close to $0$ and $p$ close to $1$.

Next, we consider the case when $(\#(S^c),\#(T^c)) = (1,2)$. The argument in this case is very similar, we apply \cref{tbound2} instead of \cref{tbound1}, and we still apply \cref{tbound3}, in this case with $k+\ell=3$. In \eqref{e:covbound4} this gives (modulo the error terms) a bound of size \begin{align}
     \notag \bigg[\int_{(C(\mathbb R)^n)^4} \epsilon^{(1-\theta)(1 - \delta) }\cdot  \big(N_{\epsilon^{\frac13-\delta}}(f_1+g_1)+N_{\epsilon^{\frac13-\delta}}(u_1&+v_1)+ N_{\epsilon^{\frac13-\delta}}(u_2+ v_2)\big)^{\theta} \\ &\cdot e^{-c \theta \epsilon^{-2}\Gamma(\epsilon)^3}\nu_a^{\otimes 2}(d\vec f,d\vec u)\mu_b^{\otimes 2}(d\vec g,d\vec v)\bigg]^{1/p}, \label{e:covbound5}
\end{align}
where $\Gamma(\epsilon):=\min\{\Gamma\big(f_1+g_1,u_1+v_1;\epsilon^{\frac13-\delta}\big),\Gamma\big(f_1+g_1,u_2+v_2;\epsilon^{\frac13-\delta}\big)\}.$ 
We split the integral into two parts according to $\Gamma(\epsilon)<\epsilon^{\frac23-\delta}$ or $\Gamma(\epsilon)\ge \epsilon^{\frac23-\delta}.$ The latter of these terms yields a super-polynomially fast rate of decay to zero in \eqref{e:covbound5}, whereas the former can be bounded by \cref{numberneeded} and \cref{crucial}, noting from \cref{crucial} and a union bound that $$\mathbb{P}(\Gamma(\epsilon)<\epsilon^{\frac23-\delta}) \leq \mathbb{P}\big(\Gamma\big(f_1+g_1,u_1+v_1;\epsilon^{\frac13-\delta}\big)<\epsilon^{\frac23-\delta}\big) +\mathbb{P}\big(\Gamma\big(f_1+g_1,u_2+v_2;\epsilon^{\frac13-\delta}\big) <\epsilon^{\frac23-\delta}\big) \leq C\epsilon^{\frac25-4\delta}.$$ This will yield an expression at most $\epsilon^{(1-\theta)(1 - \delta) }\cdot \epsilon^{\frac25-4\delta}$ in \eqref{e:covbound5}. By choosing $\delta,\theta$ close to $0$, the exponent on $\epsilon$ can be made as close to $\frac75$ as desired by taking $\theta$ close to $0$, which exceeds the desired threshold of $\epsilon^{\frac{16}{15}-\delta}$. 
\end{proof}

\section{Independence from Gaussian white noise} \label{s:proofadditional}

In this section, we prove \cref{t:2}. 

\begin{proof}[Proof of \cref{t:2}]
Assume that we have a noise $(\Omega, (\mathcal F_{s,t})_{s<t} , \mathbb P, (\theta_h)_h)$, and that $\mathcal L,\xi$ are, respectively, a directed landscape and a standard white noise satisfying the assumptions of \cref{t:2}. We claim for each fixed $s<t$ that $$\mathcal F^\xi_{s,t}\vee \mathcal F^\mathrm{DL}_{s,t} = \mathcal F_{s,t},$$ where $\mathcal{F}_{s,t}^{\xi}$ is the filtration generated by the white noise and $\mathcal{F}_{s,t}^{\mathrm{DL}}$ is the filtration generated by the directed landscape. The inclusion $\subseteq$ follows from the two bullet points of \cref{t:2}, and one may show the other inclusion by using \cref{d:noise} (1) and our assumption that $\xi$ and $\mathcal L$ together generate all of $\bigvee_{s<t} \mathcal F_{s,t}$. Our goal is to show that $\mathcal F^\xi,\mathcal F^\mathrm{DL}$ are independent of each other. We will break the proof down into five steps. 
\\
\\
\textbf{Step 1.} For $h\in L^2(\mathbb{R}^2)$, we define the pairing $(h,\xi) := \int_{\mathbb{R}^2} h(t,x)\xi(dt,dx).$ We claim that these are linear random variables on this probability space, in fact we have that $\mathbb{E}[(h,\xi)|\mathcal F_{s,t}] = (h\mathbf 1_{[s,t]\times \mathbb{R}},\xi)$. This would be clear if $\mathcal F^\xi_{s,t}$ was the full $\sigma$-algebra $\mathcal F_{s,t}$, but needs some verification in the present context, where $\mathcal F_{s,t}$ may contain strictly more information than $\mathcal F^\xi_{s,t}$.

Note that $(h\mathbf 1_{[s,t]\times\mathbb{R}},\xi)$ is $\mathcal F_{s,t}$-measurable by assumption, therefore the definition of a noise (\cref{d:noise}) implies that it is independent of $\mathcal F_{t,u}$ and $\mathcal F_{a,s}$ for any $a<s<t<u$. From this fact, we can prove the claim by writing the almost sure identity $$(h,\xi) = (h\mathbf 1_{(-\infty,s]\times\mathbb{R}},\xi) + (h\mathbf 1_{[s,t]\times\mathbb{R}},\xi) +(h\mathbf 1_{[t,\infty)\times\mathbb{R}]},\xi) ,$$ and then noting that the three terms on the right are respectively $\mathcal F_{-\infty,s}$-measurable, $\mathcal F_{s,t}$-measurable, and $\mathcal F_{t,\infty}$-measurable. By taking the conditional expectation given $\mathcal F_{s,t}$, and using independence of these three $\sigma$-algebras, the first and third term vanish while the second term is unaffected by the conditional expectation.
\\
\\
\textbf{Step 2.} We argue that $(h,\xi)^{\mathcal L}:=\mathbb E[(h,\xi)|\mathcal F^\mathrm{DL}]$ is a linear random variable on the restricted noise $(\Omega, (\mathcal F_{s,t}^\mathrm{DL})_{s<t} , \mathbb P, (\theta_h)_h)$ for any $h\in L^2(\mathbb{R}^2)$. The black noise property of $\mathcal L$ proved in \cref{t:1} will then imply that $\mathbb{E}[(h,\xi)|\mathcal F^\mathrm{DL}]=0$ for any $h\in L^2(\mathbb{R}^2)$. 
For $s<t<u$ we have
\begin{align*}
    \mathbb{E}\big[(h,\xi)^\mathcal L | \mathcal F^{\mathrm{DL}}_{s,u} \big] &= \mathbb{E}\big[(h,\xi) | \mathcal F^{\mathrm{DL}}_{s,u} \big] \\ &= \mathbb{E}\big[ \mathbb{E}[(h,\xi)|\mathcal F_{s,u}] | \mathcal F^{\mathrm{DL}}_{s,u} \big] \\ &= \mathbb{E}\bigg[ \big(\mathbb{E}[(h,\xi)|\mathcal F_{s,t}]+\mathbb{E}[(h,\xi)|\mathcal F_{t,u}]\big)\bigg|\mathcal F^\mathrm{DL}_{s,u}\bigg] \\ &= \Bbb E\big[ \Bbb E[ (h,\xi)|\mathcal F_{s,t}] \big|\mathcal F^\mathrm{DL}_{s,u}\big]+\Bbb E\big[ \Bbb E[ (h,\xi)|\mathcal F_{t,u}] \big|\mathcal F^\mathrm{DL}_{s,u}\big] \\&=\Bbb E\big[ \Bbb E[ (h,\xi)|\mathcal F_{s,t}] \big|\mathcal F^\mathrm{DL}_{s,t}\big]+\Bbb E\big[ \Bbb E[ (h,\xi)|\mathcal F_{t,u}] \big|\mathcal F^\mathrm{DL}_{t,u}\big] \\&= \Bbb E\big[  (h,\xi) \big|\mathcal F^\mathrm{DL}_{s,t}\big]+\Bbb E\big[  (h,\xi)\big|\mathcal F^\mathrm{DL}_{t,u}\big] \\ &= \mathbb{E}\big[(h,\xi)^\mathcal L | \mathcal F^{\mathrm{DL}}_{s,t} \big]+\mathbb{E}\big[(h,\xi)^\mathcal L | \mathcal F^{\mathrm{DL}}_{t,u} \big].
\end{align*}
In the first and second line, we used the tower property of conditional expectation. In the third line, we used the result of Step 1. In the fifth line, we are using the fact that $\mathcal F^\mathrm{DL}_{s,t},\Bbb E[(h,\xi)|\mathcal F_{s,t}]$ are independent of $\mathcal F^\mathrm{DL}_{t,u} \subset \mathcal F_{t,u}$, and likewise we are using the fact that $\mathcal F^\mathrm{DL}_{t,u}, \Bbb E[(h,\xi)|\mathcal F_{t,u}]$ are independent of $\mathcal F_{s,t}^\mathcal L\subset \mathcal F_{s,t}$. In the sixth and seventh line, we simply used the tower property of conditional expectation.
\\
\\
\textbf{Step 3.} Let $h_{s_j,t_j} \in L^2(\mathbb{R}^2)$ be supported on $(s_j,t_j) \times \mathbb{R}$, where $(s_j,t_j)$ are disjoint intervals. Using Step 2 and the noise property, \begin{align*}\mathbb{E}\bigg[ \prod_{j=1}^n F_{s_j,t_j}(\mathcal L) \cdot (h_{s_j,t_j},\xi) \bigg] = \prod_{j=1}^n \mathbb{E}\bigg[ F_{s_j,t_j}(\mathcal L)\cdot (h_{s_j,t_j},\xi)\bigg] = 0,\end{align*} where $F_{s_j,t_j}:C(\mathbb{R}^4_\uparrow) \to \mathbb{R}$ are arbitrary bounded functionals which are $\mathcal F_{s_j,t_j}^\mathcal L$-measurable for $1\le j\le n$, and we use the result of Step 2 in the last equality.
\\
\\
\textbf{Step 4.} As before, let $h_{s_j,t_j} \in L^2(\mathbb{R}^2)$ be supported on $(s_j,t_j) \times \mathbb{R}$, where $(s_j,t_j)$ are disjoint intervals. Using standard stochastic calculus tools for the Brownian motion, we will show that linear combinations of random variables of the form $c+\prod_{j=1}^n(h_{s_j,t_j},\xi)$ are actually dense in $L^2(\Omega,\mathcal F^\xi,P)$, as we vary over all $c\in\mathbb{R}$, $n\in\mathbb{N}$ and disjoint intervals $(s_j,t_j)$.

There are many ways to prove the desired claim, but we will take the following approach. First choose an orthonormal basis $\{e_j\}_{j\ge 1}$ of $L^2(\mathbb{R})$ and define the IID Brownian motions $B^j(t):= (\mathbf 1_{[0,t]}\otimes e_j, \xi).$ Note that $\mathcal F^\xi_{s,t}$ is precisely the $\sigma$-algebra generated by $(B^j(u)-B^j(s))_{u\in [s,t],j\in \mathbb{N}}. $ Fix some natural numbers $\ell$ and $p_1<...<p_\ell$. Let $t^n_j = j2^{-n}$, and notice that by standard construction of stochastic integrals we have (as $n\to\infty$) the convergence
\begin{align*}&\sum_{\sigma\in S_{p_\ell}} \sum_{\substack{s2^n \le k_1<k_2<...<k_{p_\ell}\le t2^n\\k_i\in\mathbb Z}}  \prod_{j=1}^\ell \prod_{r=p_{j-1}}^{p_j-1} \big( B^j(t^n_{k_{\sigma(r)}+1})-B^j(t^n_{k_{\sigma(r)}})  \big) \\ \stackrel{L^2(\Omega)}{\longrightarrow} \sum_{\sigma\in S_{p_\ell}} &\int_{s \le s_1<s_2<...<s_{p_\ell}\le t} \prod_{j=1}^\ell \prod_{r=p_{j-1}}^{p_j-1} dB^j(s_{k_{\sigma(r)}+1})  =\prod_{j=1}^\ell H_{p_j-p_{j-1}}\big((t-s)^{-1/2}(B^j(t)-B^j(s))\big),\end{align*} 
where $H_p(x)$ is the $p^{th}$ standard Hermite polynomial and $S_k$ is the symmetric group of permutations of $\{1,...,k\}$. But for fixed $s<t$, the random variables of the form of the right side form an orthonormal set whose closed linear span in $L^2(\Omega)$ (as we vary over natural numbers $\ell$ and $0= p_0<...<p_\ell,$) includes all random variables of the form $\varphi\big(B^1(t)-B^1(s),..., B^n(t)-B^n(s)\big)$ where $n\in \mathbb N$ and $\varphi:\mathbb R^n\to \mathbb R$ is bounded measurable, 
by e.g. the standard theory of Gaussian chaos \cite{nualart}. Consequently, the linear span $E$ in $L^2(\Omega)$ of random variables of the form $c+\prod_{j=1}^n(h_{s_j,t_j},\xi)$ contains any random variable of the form $\prod_{j=1}^m \varphi_j\big(B^1(t_j)-B^1(s_j),..., B^{n}(t_j)-B^{n}(s_j)\big) $, where $m,n\in\mathbb N$ and $\varphi_j:\mathbb R^{n}\to\mathbb R$ are bounded measurable and $(s_j,t_j)$ are disjoint intervals. Therefore, by the tensor product property of $L^2$ spaces (explained in the proof of \cref{lem2}) the linear span $E$ contains all random variables of the form $\Psi \big( (B^i(t_j)-B^i(s_j))_{1\le j\le m, 1\le i\le n}\big)$, where $m,n\in\mathbb N$ and $\Psi:\mathbb R^{m\times n}\to\mathbb R$ is bounded measurable and $(s_j,t_j)$ are disjoint intervals with dyadic endpoints. This proves the claim, as random variables of the latter form can be used to approximate arbitrary functionals $G(\xi)$, by letting $m,n\uparrow\infty$ and using e.g. the result of \cref{new}.
\\
\\
\textbf{Step 5.} Using the results of Steps 3 and 4, we apply a density argument to show that $\mathbb E[F(\mathcal L)G(\xi)] = \mathbb E[F(\mathcal L)]\mathbb E[G(\xi)]$ for all bounded measurable $F,G$ on the respective canonical spaces, completing the proof. By the result of Step 4, it suffices to prove the claim when $G(\xi) = c+\prod_{j=1}^r(h_{s_j,t_j},\xi),$ where $c\in \mathbb{R}$ and $(s_j,t_j)$ are disjoint intervals with endpoints taking values in $\{k2^{-n}:k\in \mathbb{Z}\}$ for some $n\in\mathbb{N}.$ 

In turn, it suffices to prove the claim when $F$ is a finite linear combination of functionals of the form $\prod_{j=0}^{r+1} F_{s_j,t_j}(\mathcal L)$ where $F_{s_j,t_j}:C(\mathbb R^4_\uparrow)\to \Bbb R$ are $\mathcal F_{s_j,t_j}^\mathcal L$-measurable and the extreme endpoints are defined by $s_0:=-\infty$, $t_0:=s_1,$ $s_{r+1}:=t_n$, $t_{r+1}:=+\infty$. 
Indeed, the finite linear span of such functionals is dense in $L^2(\Omega, \mathcal F^\mathrm{DL},\mathbb P)$ by the tensor product property of $L^2$ spaces (see proof of \cref{lem2}). 
But this is immediate from the result of Step 3.
\end{proof}

\bibliographystyle{alpha}
\bibliography{mybibliography.bib}

\end{document}